\documentclass[12pt]{amsart}
\usepackage{geometry}     
\usepackage{graphicx}
\usepackage{color}
\usepackage{amsthm}
\usepackage{amsmath}
\usepackage{amsfonts}
\usepackage{enumitem}
\usepackage{amssymb}
\usepackage{epstopdf}
\usepackage{hyperref}
\usepackage{mathrsfs}

\newtheorem{theorem}{Theorem}[section]
\newtheorem{proposition}[theorem]{Proposition}
\newtheorem{lemma}[theorem]{Lemma}

\newtheorem{example}[theorem]{Example}

\DeclareMathOperator{\Vol}{Vol}
\DeclareMathOperator{\interior}{int}

\DeclareMathOperator{\inj}{inj}

\DeclareMathOperator{\length}{L}

\DeclareMathOperator{\support}{spt}

\DeclareMathOperator{\mass}{\mathbf{M}}
\DeclareMathOperator{\diameter}{diam}
\DeclareMathOperator{\Cone}{Cone}
\DeclareMathOperator{\dist}{dist}
\DeclareMathOperator{\centers}{Centers}
\DeclareMathOperator{\faces}{Faces}
\DeclareMathOperator{\radius}{rad}

\DeclareMathOperator{\Exp}{Exp}
\DeclareMathOperator{\Refine}{Ref}

\theoremstyle{definition}
\newtheorem{definition}[theorem]{Definition}

\newtheorem{remark}[theorem]{Remark}

\usepackage{booktabs,amsmath}

\def\XXint#1#2#3{{\setbox0=\hbox{$#1{#2#3}{\int}$}
    \vcenter{\hbox{$#2#3$}}\kern-.5\wd0}}

\def\YYint#1#2#3{{\setbox0=\hbox{$#1{#2#3}{\int}$}
    \lower1ex\hbox{$#2#3$}\kern-.46\wd0}}
\def\YYYint#1#2#3{{\setbox0=\hbox{$#1{#2#3}{\int}$}
    \lower0.35ex\hbox{$#2#3$}\kern-.48\wd0}}

\def\ZZint#1#2#3{{\setbox0=\hbox{$#1{#2#3}{\int}$}
    \raise1.15ex\hbox{$#2#3$}\kern-.57\wd0}}
\def\ZZZint#1#2#3{{\setbox0=\hbox{$#1{#2#3}{\int}$}
    \raise0.85ex\hbox{$#2#3$}\kern-.53\wd0}}

\title{Parametric Coarea Inequality for 1-cycles}
\author{Bruno Staffa}
\begin{document}
\maketitle

\begin{abstract}
    We prove the Parametric Coarea Inequality for $1$-cycles conjectured by Guth and Liokumovich.
\end{abstract}

\tableofcontents

\section{Introduction}

Let $M$ be an $n$-dimensional compact Riemannian manifold with boundary. We consider the space $\mathcal{I}_{k}(M;\mathbb{Z}_{2})$ of flat $k$-chains in $M$ with $\mathbb{Z}_{2}$ coefficients as defined in \cite{Fleming66} and \cite{FedererGMT}[Section~4.2.26], and we denote by $\mass$ the mass functional on that space (which coincides with the $k$-dimensional volume for flat $k$-chains induced by $k$-dimensional submanifolds of $M$). Let $\tau\in\mathcal{I}_{k}(M;\mathbb{Z}_{2})$ be such that $\partial\tau$ is supported in $\Sigma=\partial M$. Notice that $\mass(\partial\tau)$ might be arbitrarily large compared to $\mass(\tau)$. An example is provided in Figure 1, where a piecewise linear flat $1$-chain $\tau$ in the unit square $M$ is represented, with the property that $\partial\tau$ supported in $\partial M$, $\mass(\tau)<\infty$ and $\mass(\partial\tau)=\infty$. 

\begin{figure}[h]
\centering
\includegraphics[scale = 0.4]{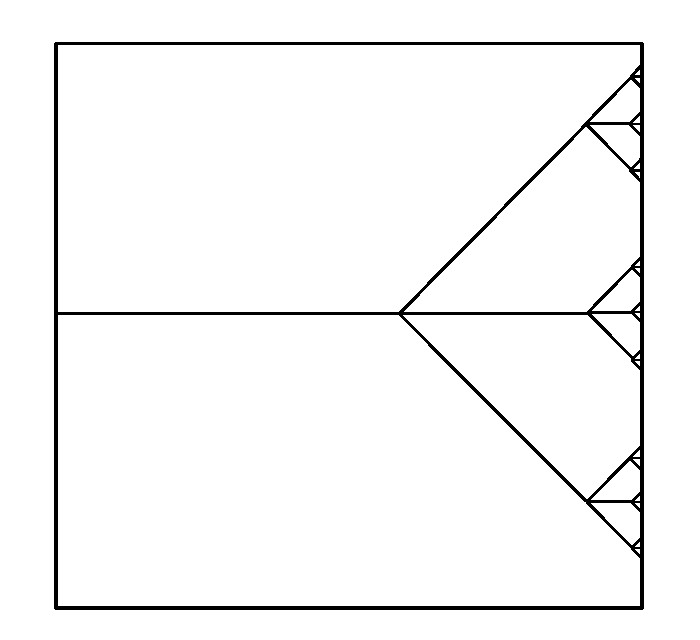}
\caption{}
\end{figure}

We can think of $\tau$ as an infinite tree, whose set of vertices is $V=\bigsqcup_{n\in\mathbb{N}}V_{n}$, each $V_{n}$ has $3^{n}$ elements and each vertex $v\in V_{n}$ has one incoming edge connecting it with a vertex $w\in V_{n-1}$ and three incoming edges connecting it with three different vertices of $V_{n+1}$. The length of the edges connecting $V_{n}$ with $V_{n+1}$ is at most $\frac{\sqrt{2}}{4^{n}}$, which yields $\mass(\tau)\leq\frac{2}{3}+3\sqrt{2}$. But $\partial\tau$ has infinite mass.

Nevertheless, the previous situation can be avoided if we cut a portion of $\tau$ close to $\Sigma$. By the Coarea Inequality applied to the chain $\tau$ and the distance function to $\Sigma$, given $\varepsilon>0$ we can find $s\in (0,\varepsilon)$ such that
\begin{equation}\label{Coarea Inequality}
    \mass(\partial[\tau\llcorner M_{s}])\leq\frac{\mass(\tau)}{\varepsilon}
\end{equation}
where $M_{s}=M\setminus N_{s}\Sigma$ and $N_{s}\Sigma$ is the $s$-tubular neighborhood of $\Sigma$. We will be interested in obtaining an inequality like (\ref{Coarea Inequality}) not for an individual chain $\tau$ but for a continuous family of $k$-chains with boundary supported in $\partial M$. We will regard such families as families of relative $k$-cycles in $(M,\partial M)$. We consider the flat topology on the space $\mathcal{Z}_{k}(M,\partial M;\mathbb{Z}_{2})$ of flat relative $k$-cycles, which is induced by the flat norm
\begin{equation*}
    \mathcal{F}(\tau)=\inf\{\mass(\alpha)+\mass(\beta):\tau=\alpha+\partial\beta,\alpha\in\mathcal{Z}_{k}(M,\partial M;\mathbb{Z}_{2}),\beta\in\mathcal{I}_{k+1}(M,\partial M;\mathbb{Z}_{2})\}.
\end{equation*}
Two relative cycles $\tau$ and $\tau'$ are close in flat topology if there exists a $(k+1)$-chain $\beta$ such that $\partial\beta=\tau-\tau'$ and $\mass(\beta)$ is small. We will consider multiparameter families in $\mathcal{Z}_{k}(M,\partial M;\mathbb{Z}_{2})$, with domain a certain cubical complex $X$ as defined in Section \ref{Section Almgren-Pitts}.


Observe that the number $s\in(0,\varepsilon)$ in (\ref{Coarea Inequality}) depends heavily on $\tau$. Moreover, it is not hard to construct a continuous family $F$ of relative $1$-cycles for which it is not possible to choose $s(x)\in(0,\varepsilon)$ continuously such that $\mass(\partial[F(x)\llcorner M_{s(x)}])\leq\frac{\mass(F(x))}{\varepsilon}$. In \cite{GL22} such example was constructed and it is illustrated in Figure 2. The family $F:[0,1]\to\mathcal{Z}_{1}([0,1]^{2},\partial[0,1]^{2};\mathbb{Z}_{2})$ consists of a small but tightly wound spiral moving from left to right on $M=[0,1]^{2}$ with constant speed as $x$ increases. Let $d\in(0,\varepsilon)$ be such that the two vertical lines at distance $d$ from the center of the spiral intersect $F(0)$ at $0$-chains of mass $\frac{\mass(F(0))}{\varepsilon}$. Then for every $x$ either $s(x)<1-(x+d)$ or $s(x)>1-(x-d)$. Moreover at $x=0$ the first must hold and at $x=1$ the second must hold. The previous forces $s(x)$ to be discontinuous.


\begin{figure}[h]
\centering
\includegraphics[scale=1.1]{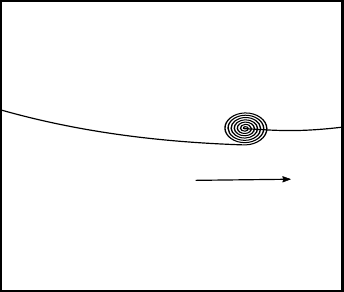}
\caption{}
\end{figure}

Given a continuous family $F:X\to\mathcal{Z}_{k}(M,\partial M)$, we may wonder whether it is possible to perturb it a little bit to obtain a new family $F'$ whose mass is not much larger than $\mass(F(x))$ and which verifies $\mass(\partial F'(x))\leq\frac{\mass(F(x))}{\varepsilon}$ for every $x\in X$. The previous example shows that this can not be done just choosing the coarea inequality cut $s(x)$ continuously, but other methods could work. In \cite{GL22}, Guth and Liokumovich positively aswered this question for families of $1$-cycles in $3$-dimensional manifolds. They studied this problem in the context of proving the Weyl law for the volume spectrum for cycles of codimension higher than $1$. The volume spectrum of a Riemannian manifold is a sequence of numbers $(\omega^{k}_{p}(M))_{p\in\mathbb{N}}$ which correspond to the volumes of certain (possibly singular) $k$-dimensional minimal submanifolds of $M$, $1\leq k\leq n-1$, $\dim(M)=n$. They are constructed via a min-max procedure in the space $\mathcal{Z}_{k}(M,\partial M;\mathbb{Z}_{2})$ of flat relative cycles on $M$ and $\omega_{p}^{k}(M)$ is called the ($k$-dimensional) $p$-width of $M$. For a detailed exposition on Almgren-Pitts Min-Max Theory, the construction of the $p$-widths and the Weyl law see \cite{FA62}, \cite{FA65} \cite{JP81}, \cite{GuthMinMax}, \cite{MN} and \cite{Willmore}. In the 1980's, Gromov \cite{Gromov86} suggested to think of the $p$-widths as non-linear analogs of the eigenvalues of the Laplacian on $M$ (see also \cite{Gromov02} and \cite{Gromov09}). He conjectured that they should satisfy the following Weyl law
\begin{equation}\label{Eq Weyl law}
    \lim_{p\to\infty}\omega_{p}^{k}(M)p^{-\frac{n-k}{n}}=\alpha(n,k)\Vol(M)^{\frac{k}{n}}
\end{equation}
for a certain universal constant $\alpha(n,k)$. The conjecture was resolved by Liokumovich, Marques and Neves for $k=n-1$ and $n$ arbitrary \cite{LMN}, and later by Guth and Liokumovich for $k=1$, $n=3$. In \cite{GL22}, the later two authors showed that the Weyl law for cycles of codimension higher than $1$ can be obtained by proving parametric versions of the Coarea Inequality and the Isoperimetric Inequality. They proved the corresponding versions for families of $1$-cycles in $3$-manifolds, which allowed them to resolve Gromov's conjecture (\ref{Eq Weyl law}) in that case. The following is the Parametric Coarea Inequality that they obtained.

\begin{theorem}[Guth-Liokumovich, Parametric Coarea Inequality for $1$-cycles in $3$-manifolds]
    Let $M$ be a compact Riemannian $3$-manifold with boundary. Fix $\eta>0$. For all $\varepsilon\in (0,\varepsilon_{0})$ and $p\geq p_{0}(\Omega,\varepsilon)$ the following holds. Let $F:X\to\mathcal{Z}_{1}(M,\partial M;\mathbb{Z}_{2})$ be a continuous family of relative $1$-cycles without concentration of mass, $\dim(X)=p$. Then there exists a continuous map $F':X\to\mathcal{Z}_{1}(M_{\varepsilon},\partial M_{\varepsilon};\mathbb{Z}_{2})$ such that
    \begin{enumerate}
        \item $\mathcal{F}(F(x)\llcorner M_{\varepsilon},F'(x))<\eta$ as relative cycles in $(M_{\varepsilon},\partial M_{\varepsilon})$.
        \item $\mass(F'(x))\leq\mass(F(x))+\frac{\mass(F(x))}{\varepsilon\sqrt{p}}+c\Vol_{2}(\partial M)\sqrt{p}$.
        \item $\mass(\partial F'(x))\leq c\Vol_{2}(\partial M)p$.
    \end{enumerate}
\end{theorem}

The strategy of Guth and Liokumovich was to first define $F'$ in the $0$-skeleton of $X$ using (\ref{Coarea Inequality}) and then extend it inductively skeleton by skeleton by interpolating on each cell $C$ of $X$ between the values of $F'$ at its vertices. It is crucial in their argument to first perturb $F$ to a $\delta$-localized family (see Section \ref{Section Preliminaries} for the definition and \cite{GL22}[Section~2] and \cite{StaWeyl}[Section~3] for a detailed exposition on $\delta$-localized families), which permits to reduce the problem to interpolating between different Coarea Inequality cuts applied to the same chain $F(x)$ (roughly speaking, the previous means interpolating between $F(x)\llcorner M_{s_{1}}$ and $F(x)\llcorner M_{s_{2}}$ for $s_{1},s_{2}\in (0,\varepsilon)$ satisfying (\ref{Coarea Inequality})). This interpolation happens on $\Sigma=\partial M$, which is $2$-dimensional. They considered a certain triangulation of $\Sigma$ of width $\rho=\frac{1}{\sqrt{p}}<<\varepsilon$, they added certain cones over $\partial F'(x)\llcorner D$ for each cell $D$ of $\Sigma$ and each $x\in X_{0}$ and then they were able to define an interpolation by suitably contracting the differences $F'(v)-F'(w)$ on each $D$ via certain carefully defined rescaling factors. 

Part of the upper bounds for $\mass(F'(x))$ in 
\cite{GL22} come from mod $2$ cancellation: any $1$-chain with $\mathbb{Z}_{2}$-coefficients supported in the codimension $1$ skeleton $\Sigma^{1}$ of $\Sigma$ has mass at most equal to the total length of $\Sigma^{1}$ which is $\sim \frac{\text{Area}(\Sigma)}{\rho^{2}}\rho\sim\sqrt{p}$. Nevertheless, when $\dim(\Sigma)>2$ (which corresponds to $\dim(M)>3)$ such cancellation does not happen because $\Sigma^{1}$ (the codimension $1$ skeleton of $\Sigma$) has dimension larger than $1$. This is the main reason of the restriction to the ambient manifold dimension to be $3$ in the result of Guth and Liokumovich. However, they conjectured parametric versions of the Coarea Inequality (and also of the Isoperimetric Inequality) for families of $k$-cycles in $n$-manifolds for arbitrary $k$ and $n$ (see \cite{GL22}[Conjecture~1.6]). In this paper, we solve their conjecture for $k=1$ and $n$ arbitrary, obtaining the following result. 

\begin{theorem}[Parametric coarea inequality for $1$-cycles]\label{Thm Parametric Coarea}
Let $M$ be a compact $n$-dimensional Riemannian manifold with boundary, $n\geq 4$. Let $\alpha\in(0,1)$. There exist $p_{0}=p_{0}(M,\alpha)$, a constant $c=c(n)$ and a sequence of numbers $\gamma_{p}=\gamma_{p}(n,\alpha)$ converging to $0$ such that the following is true. Let $p\geq p_{0}$ and let $F:X^{p}\to\mathcal{Z}_{1}(M,\partial M;\mathbb{Z}_{2})$ be a family of relative $1$-cycles which is continuous in the flat topology and has no concentration of mass. Then for every $\varepsilon>0$, there exists a continuous family $F':X^{p}\to\mathcal{I}_{1}(M;\mathbb{Z}_{2})$ of absolute chains in $M$ such that
    \begin{enumerate}
        \item $\support(\partial F'(x))\subseteq\partial M$. Therefore $F'$ induces a family of relative $1$-cycles $F':X^{p}\to\mathcal{Z}_{1}(M,\partial M;\mathbb{Z}_{2})$ which is denoted in the same way.
        \item $\mathcal{F}(F(x),F'(x))\leq\varepsilon$ as relative cycles in the space $\mathcal{Z}_{1}(M,\partial M;\mathbb{Z}_{2})$.
        \item $\mass(F'(x)) \leq(1+\gamma_{p})\mass(F(x))+c\Vol_{n-1}(\partial M)p^{\frac{n-2}{n-1}+\alpha}$.
        \item $\mass(\partial F'(x)) \leq c\Vol_{n-1}(\partial M)p^{1+\alpha}$.
    \end{enumerate}
\end{theorem}

\begin{remark}
    In fact, we prove Theorem \ref{Thm Parametric Coarea} for almost $1$-Lipschitz triangulable piecewise smooth Riemannian manifolds with boundary $(M,g)$. This is a class of manifolds which includes both the category of compact smooth manifolds with boundary and the category of compact PL manifolds with boundary. It will be the correct set up for our constructions because we are interested in proving results about compact Riemannian manifolds but our methods will involve considering triangulations with almost flat simplices. See a detailed discussion about piecewise smooth Riemannian manifolds in Section \ref{Section piecewise smooth}.
\end{remark}

Theorem \ref{Thm Parametric Coarea} is used in \cite{StaWeyl} to prove the Weyl law for the volume spectrum for $1$-cycles in manifolds of any dimension $n$. As an application of the previous, we obtained generic equidistribution of stationary geodesic nets in $n$-manifolds (see \cite{LiSta}, where the later result is proved assuming the Weyl law for $1$-cycles).

The paper is organized in the following way. In Section \ref{Section n=4}, we prove Theorem \ref{Thm Parametric Coarea} for $n=4$ in order to introduce the main ideas and constructions before getting into the technicalities that arise when the codimension is higher. In Section \ref{Section n arbitrary}, we generalize what was done for $n=4$ to arbitrary $n$. As mentioned above, the main obstruction to extend the Parametric Coarea Inequality for $1$-cycles to manifolds of dimension larger than $3$ is that there is no mod $2$ cancellation in $\Sigma^{1}$, the codimension $1$ skeleton of $\Sigma=\partial M$. This problem is solved by performing a sequence of cuts $s_{1}(x)$,...,$s_{n-2}(x)$ to each $F(x)$, $x\in X_{0}$ using the Coarea Inequality. Each cut is done at a different scale, taking $s_{j}(x)\in(0,\varepsilon_{j})$ with $\varepsilon_{1}>>\varepsilon_{2}>>...>>\varepsilon_{n-2}$. The cuts are done with respect to the distance function to the $A_{j}$, which are certain $(n-1)$-dimensional PL submanifolds of $M$ which ``enlarge'' $\Sigma^{j}$, the codimension $j$ skeleton of $\Sigma$. Those constructions are explained in sections \ref{Section cuts} to \ref{Section AD}. We perform, cell by cell, a very delicate interpolation procedure between the different cuts $C(F(x);s_{1}(x),...,s_{n-2}(x))$ at the different vertices $x\in X_{0}$, making sure to preserve control over $\mass(F'(x))$ and $\mass(\partial F'(x))$. This is explained in Sections \ref{Section Interpolation formula} and \ref{Section Proof of Coarea Ineq}.  To simplify the exposition, in Sections \ref{Section cuts} to \ref{Section Proof of Coarea Ineq} we work with rectangular domains, and later in Section \ref{Extension to triangulable domains} we extend our result to almost $1$-Lipschitz triangulable piecewise smooth Riemannian manifolds with boundary.

\textbf{Acknowledgments.} I am very grateful to Yevgeny Liokumovich for suggesting this problem and for all the valuable conversations we have had while I was working on it. I would also like to express my deep gratitude to the Hausdorff Research Institute for Mathematics, where part of this work was completed during my participation in the Trimester Program: Metric Analysis, funded by the Deutsche Forschungsgemeinschaft (DFG, German Research Foundation) under Germany's Excellence Strategy – EXC-2047/1 – 390685813.

\section{Preliminaries}\label{Section Preliminaries}

\subsection{Piecewise smooth Riemannian manifolds}\label{Section piecewise smooth}

We will need to work with a class of Riemannian manifolds with piecewise smooth boundary which admit a compatible PL structure. For that purpose, we introduce the following definitions.

\begin{definition}\label{Def Euclidean polyhedron}
    A Euclidean polyhedron $P$ is a linearly embedded finite simplicial complex in $\mathbb{R}^{N}$ for some $N\in\mathbb{N}$.
\end{definition}

\begin{definition}
    Let $P$ be a Euclidean polyhedron. A map $\phi:P\to\mathbb{R}^{N}$ is piecewise smooth if $\phi$ is continuous and for each face $F$ of a certain triangulation $T$ of $P$ the map $\phi|_{F}$ is smooth. $\phi$ is a piecewise smooth embedding if it is a topological embedding and $\phi|_{F}$ is a smooth embedding for each face $F$ of $P$.
\end{definition}

\begin{definition}
    An $n$-dimensional Riemannian polyhedron is a is a piecewise smoothly embedded $n$-dimensional polyhedron in $\mathbb{R}^{N}$ (i.e. the image of a Euclidean polyhedron under a piecewise smooth embedding into $\mathbb{R}^{n}$ for some value of $N$) equipped with the piecewise smooth Riemannian metric $g$ induced by the Euclidean metric in $\mathbb{R}^{N}$. 
\end{definition}

\begin{definition}\label{Def smooth triangulation}
    Let $P$ be a Riemannian polyhedron embedded in $\mathbb{R}^{N}$ and denote the corresponding embedding by $\iota:P\to\mathbb{R}^{N}$. A smooth triangulation of $P$ is a homeomorphism $\tau:P'\to P$ where $P'$ is a Euclidean polyhedron provided with a triangulation $T'$ and $\iota\circ\tau:P'\to\mathbb{R}^{n}$ is a piecewise smooth embedding with respect to the triangulation $T'$. A triangulated Riemannian polyhedron is a pair $(P,T)$ where $P$ is a Riemannian polyhedron and $T$ is the triangulation on $P$ induced by the $T'$ just defined. A refinement of $T$ is the triangulation induced by taking a refinement of the corresponding $T'$.
\end{definition}

\begin{remark}
    Riemannian polyhedra (and hence Euclidean ones) admit different triangulations. For example, we can obtain infinitely many of them from a single one by refinement. 
\end{remark}


\begin{definition}
    A piecewise smooth Riemannian manifold $(M^{n},g)$ is an $n$-dimensional Riemannian polyhedron provided with a topological manifold with boundary structure (i.e. with an atlas of class $C^{0}$). In other words, it is an $n$-dimensional topological submanifold with boundary of $\mathbb{R}^{N}$ (for some $N\in\mathbb{N}$) which admits a smooth triangulation and is equipped with a continuous Riemannian metric $g$ which is smooth on each simplex of that triangulation.
\end{definition}

\begin{definition}\label{Def almost 1-Lip triangulable}
    A piecewise smooth Riemannian manifold $(M,g)$ is almost $1$-Lipchitz triangulable if for every $\varepsilon>0$ there exists a smooth $(1+\varepsilon)$-bilipschitz triangulation $\tau_{\varepsilon}:P_{\varepsilon}\to M$ (here $P_{\varepsilon}$ is a Euclidean polyhedron according to Definition \ref{Def smooth triangulation}).
\end{definition}

\begin{remark}
    By the Nash Embedding Theorem and Theorems 1.1 and 1.2 of \cite{Bowditch}, every smooth compact Riemannian manifold with boundary is almost $1$-Lipschitz triangulable. So is every piecewise linear submanifold of $\mathbb{R}^{N}$. Thus the collection of all almost $1$-Lipschitz triangulable piecewise smooth Riemannian manifolds includes both the category of compact smooth manifolds with boundary and the category of compact PL manifolds with boundary. It will be the correct set up for our constructions because we are interested in proving results about compact Riemannian manifolds but our methods will involve considering triangulations with almost flat simplices.
\end{remark}

\begin{remark}
    Because of the works of Akopyan \cite{Akopyan} and Minemyer \cite{Minemyer}, every Euclidean simplicial complex as defined by Bowditch in \cite{Bowditch} can be piecewise linearly embedded in $\mathbb{R}^{N}$ and hence can be regarded as a Euclidean polyhedron.
\end{remark}

\begin{remark}
    As far as the author knows, the question whether every piecewise smooth Riemannian manifold is almost $1$-Lipschitz triangulable has not been studied. Maybe the methods of Bowditch \cite{Bowditch} also work in this case.
\end{remark}

\begin{remark}
    Notice that in Definition \ref{Def almost 1-Lip triangulable}, $P_{\varepsilon}$ inherits a topological manifold structure from $M$ and hence it is a compact PL manifold with boundary. So the condition of being almost $1$-Lipschitz triangulable can be re-expressed as being $(1+\varepsilon)$-bilipschitz diffeomorphic to a compact PL manifold with boundary for every $\varepsilon>0$. 
\end{remark}

\begin{definition}\label{Def tubular neighborhood}
    Let $(M,g)$ be a piecewise smooth Riemannian manifold. We denote $\Sigma=\partial M$ (which is a subpolyhedron of $M$) and $d$ the metric on $M$ induced by the Riemannian metric $g$. Given two subsets $A,B\subseteq M$ we define
    \begin{equation*}
        \dist(A,B)=\inf\{d(x,y):x\in A,y\in B\}.
    \end{equation*}
    Let $A\subseteq M$ be a subset of $M$. For each $r>0$, we denote
    \begin{align*}
        N_{r}A &=\{x\in M:\dist(x,A)< r\}\\
        D_{r}A & =\{x\in M:\dist(x,A)=r\}.
    \end{align*}
\end{definition}

\begin{definition}
    Given a piecewise smooth Riemannian manifold $(M,g)$, its injectivity radius $\inj(M,g)$ is the minimum among the injectivity radius of its faces.
\end{definition}


\subsection{Almgren-Pitts min-max theory}\label{Section Almgren-Pitts}

Let $(M,g)$ be an $n$-dimensional piecewise smooth Riemannian manifold with boundary. Let $\mathcal{Z}_{k}(M,\partial M)$ be the space of mod $2$ flat $k$-cycles on $M$ relative to $\partial M$ and let $\mathcal{I}_{k}(M,\partial M)$ the corresponding space of flat $k$-chains, $0\leq k\leq n$. We also consider the spaces $\mathcal{Z}_{k}(M)$ and $\mathcal{I}_{k}(M)$ of mod $2$ absolute flat cycles and chains on $M$ respectively as defined in \cite{Fleming66} and in \cite{FedererGMT}[Section~4.2.26]. The mass functional on the previous spaces with respect to the metric $g$ is denoted by $\mass$. Notice that for a relative chain $\tau\in\mathcal{I}_{k}(M,\partial M)$,
\begin{equation}\label{Mass of a relative chain}
    \mass(\tau)=\inf\{\mass(\tau'):\tau'\in\mathcal{I}_{k}(M),[\tau']=\tau\}
\end{equation}
where given an absolute $k$-chain $\tau'$, $[\tau']$ denotes its class as a relative $k$-chain in $(M,\partial M)$.
The flat norm associated to $\mass$ is denoted by $\mathcal{F}$, which for absolute $k$-chains is defined as
\begin{equation*}
    \mathcal{F}(\tau)=\inf\{\mass(\alpha)+\mass(\beta):\tau=\alpha+\partial\beta,\alpha\in\mathcal{I}_{k}(M),\beta\in\mathcal{I}_{k+1}(M)\}
\end{equation*}
and the analogous definition holds for relative chains.

\begin{definition}
    We denote $I=[0,1]$ and $I^{d}=[0,1]^{d}$ for every $d\in\mathbb{N}$. In addition, given $q\in\mathbb{N}$ we denote $I(q)$ the complex obtained by dividing $I$ into $q$ equal parts, and $I^{d}(q)=I(q)\times...\times I(q)$ ($d$ times) with the product structure. We say that $I^{d}(q)$ is obtained from $I^{d}$ by performing a $q$-refinement.
\end{definition}

\begin{definition}
    A cubical complex $X$ is a subcomplex of $I^{d}(q)$ for some $d,q\in\mathbb{N}$. Given $q'\in\mathbb{N},$ we denote $X(q')$ the complex obtained by performing a $q'$-refinement of each cell of $X$. Notice that $X(q')$ is a subcomplex of $I^{d}(qq')$.
\end{definition}

\begin{remark}
    The convention in the previous two definitions is different from that in 
    \cite{MN}, \cite{Willmore} and \cite{GL22}.
\end{remark}

\begin{definition}
    Given a $p$-dimensional cubical complex $X$ and $0\leq j\leq p$, we denote by $X_{j}$ the $j$-skeleton of $X$.
\end{definition}

\begin{definition}
    Given a cubical complex $X$ and a cell $C$ of $X$, we denote $V(C)$ the set of vertices of $C$.
\end{definition}

\begin{definition}
    Given a $p$-dimensional cubical complex $X$ and $0\leq j\leq p$, we denote by $\faces_{j}(X)$ the set of $j$-dimensional faces of $X$.
\end{definition}

\begin{definition}
    Let $X$ be a $p$-dimensional cubical complex, let $0\leq j\leq p$ and  let $F:X_{j}\to\mathcal{Z}_{k}(M,\partial M)$ be a continuous map in the flat topology. We say that $F$ is $\varepsilon$-fine if for every cell $C$ of $X$ and every $x,y\in C\cap X_{j}$
    \begin{equation*}
        \mathcal{F}(F(x),F(y))\leq\varepsilon.
    \end{equation*}
    When $j=p$, we say that $F$ has no concentration of mass if
    \begin{equation*}
        \limsup_{r\to 0}\sup_{x\in X}\sup_{p\in M}\{\mass(F(x)\llcorner B(p,r)\}=0.
    \end{equation*}
\end{definition}

\begin{definition}
    We denote by $\lambda_{k}\in H^{n-k}(\mathcal{Z}_{k}(M,\partial M;\mathbb{Z}_{2});\mathbb{Z}_{2})$ the fundamental cohomology class of $\mathcal{Z}_{k}(M,\partial M;\mathbb{Z}_{2})$ (see \cite{GuthMinMax} and \cite{LS}[Section~3]). It is defined by the property that a continuous family $F:X\to\mathcal{Z}_{k}(M,\partial M;\mathbb{Z}_{2})$ is a sweepout of $M$ (i.e. the gluing homomorphism described in \cite{GuthMinMax} is nontrivial) if and only if $F^{*}(\lambda_{k})\neq 0$.
\end{definition}

\begin{definition}
    Let $X$ be a cubical complex. We say that a continuous map $F:X\to\mathcal{Z}_{k}(M,\partial M)$ is a $p$-sweepout of $(M,\partial M)$ by $k$-cycles if $F^{*}(\lambda_{k}^{p})\neq 0$. We denote $\mathcal{P}_{p}^{k}=\mathcal{P}_{p}^{k}(M,\partial M)$ the collection of $p$-sweepouts on $(M,\partial M)$ by $k$-cycles with no concentration of mass.
\end{definition}

\subsection{$\delta$-localized families}
In order to prove the Parametric Coarea Inequality, it is convenient to work with $\delta$-localized families, which were first defined in \cite{GL22}[Section~2]. We proceed to introduce the definition as stated in \cite{StaWeyl}[Definition~3.1]. Here $(M,g)$ is a piecewise smooth Riemannian manifold with boundary.

\begin{definition}\label{Def delta admissible}
    Let $\delta\leq\inj(M,g)$. A $\delta$-admissible family on $M$ is a finite collection $\{U_{i}\}_{i\in I}$ of disjoint open balls such that if $r_{i}=\radius(U_{i})$ then
\begin{equation*}
    \sum_{i\in I}r_{i}<\delta.
\end{equation*}
An $(N,\delta)$-admissible family is a $\delta$-admissible family such that $|I|\leq N$.
\end{definition}

\begin{definition}\label{Def delta localized}
    Let $X$ be a $p$-dimensional cubical complex and $1\leq j\leq p$. A family of chains $F:X_{j}\to\mathcal{I}_{k}(M)$ is said to be $\delta$-localized if for every cell $C$ of $X$ there exists a $\delta$-admissible family $\{U_{i}^{C}\}_{i\in I_{C}}$ such that for every $x,y\in C\cap X_{j}$
    \begin{equation*}
        \support(F(x)-F(y))\subseteq \bigcup_{i\in I_{C}}U_{i}^{C}.
    \end{equation*}
     If each $\delta$-admissible family $\{U_{i}^{C}\}$ is also $(N,\delta)$-admissible, we say that $F$ is $(N,\delta)$-localized. We introduce the same definition for families of relative chains and for families of absolute and relative cycles by replacing $\mathcal{I}_{k}(M)$ by $\mathcal{I}_{k}(M,\partial M)$, $\mathcal{Z}_{k}(M)$ and $\mathcal{Z}_{k}(M,\partial M)$ respectively.
\end{definition}

\begin{remark}
    When $\partial M\neq\emptyset$, by $\support(F(x)-F(y))$ we mean $\support(\overline{F}(x)-\overline{F}(y))$ where $\overline{F}(x),\overline{F}(y)\in\mathcal{I}_{k}(X)$, $[\overline{F}(x)]=F(x)$, $[\overline{F}(y)]=F(y)$ and $F(x)\llcorner\partial M=F(y)\llcorner\partial M=0$ (i.e. $\overline{F}(x)$ and $\overline{F}(y)$ are absolute $k$-chains which represent $F(x)$ and $F(y)$ and have minimal mass, being $\mass(F(x))=\mass(\overline{F}(x))$ and $\mass(F(y))=\mass(\overline{F}(y))$ according to (\ref{Mass of a relative chain})).
\end{remark}

Having this control over how the family $F$ behaves on the cells of $X$ allows to produce new families with desired mass bounds and additional properties. For example, it permits to construct families inductively skeleton by skeleton, starting from a discrete family defined in the $0$-skeleton of $X$, keeping track of quantities like the masses of the chains involved, of its boundaries, $\mathcal{F}(F(x)-F(y))$ for $x,y\in C$ and the supports of the $F(x)$. In order to use these families for our constructions, we need to apply the following result which was proved in \cite{StaWeyl} building on the work \cite{GL22}: every continuous family in the flat topology can be arbitrarily well approximated by a $\delta$-localized one. We state the main approximation theorem, which is proved in \cite{StaWeyl}[Theorem~3.15].

\begin{theorem}\label{Thm delta localized approximation}
    Let $X$ be a $p$-dimensional cubical complex and let $F:X\to\mathcal{Z}_{k}(M,\partial M)$ be a continuous map without concentration of mass. Let $\varepsilon,\delta>0$. Then there exists a continuous map $F':X\to\mathcal{Z}_{k}(M,\partial M)$ without concentration of mass and a fine cubulation of $X$ such that
    \begin{enumerate}
        \item $F'$ is $\delta$-localized (and moreover $(N(k,p),\delta)$-localized for some $N(k,p)\in\mathbb{N}$ depending only on $k$ and $p$).
        \item $\mathcal{F}(F(x),F'(x))\leq\varepsilon$ for every $x\in X$.
        \item If $x\in C$ for some cell $C$ of $X$ then 
        \begin{equation*}
            \mass(F'(x))\leq\max\{\mass(F(v)):v\in V(C)\}+\varepsilon.
        \end{equation*}
        \item \begin{equation*}
    \sup_{x\in X}\sup_{p\in M}\{\mass(F(x)\llcorner B(p,\delta))\}\leq\varepsilon.
\end{equation*}
        \item Given a cell $C$ of $X$, let $\{U_{i}^{C}\}_{i\in I_{C}}$ be the $\delta$-admissible family associated to $C$. Then if $C\subseteq C'$,
        \begin{equation*}
            \bigcup_{i\in I_{C}}U_{i}^{C}\subseteq\bigcup_{i\in I_{C'}}U_{i}^{C'}.
        \end{equation*}
    \end{enumerate}
\end{theorem}

\subsection{Rectangular complexes}\label{Section rectangular complexes}
\begin{definition}\label{Def k-cube}
    A rectangle $Q$ in $\mathbb{R}^{n}$ is a set of the form
\begin{equation*}
    Q=\{p+v:v\in J_{1}\times...\times J_{n}\}
\end{equation*}
    where $p$ is some point in $\mathbb{R}^{n}$ and $J_{1},...,J_{n}$ are sets such that each $J_{i}$ is either $\{0\}$ or an interval $[0,a_{i}]$ for some $a_{i}>0$. The dimension of such rectangle $Q$ is
    \begin{equation*}
        \dim(Q)=\#\{i:J_{i}\neq\{0\}\}.
    \end{equation*}
\end{definition}

\begin{definition}\label{Def rectangular complex}
    A rectangular complex consists of a collection $T$ of rectangles in $\mathbb{R}^{n}$ such that
    \begin{enumerate}
        \item Each face of a rectangle $Q\in T$ also belongs to $T$.
        \item If $Q_{1},Q_{2}\in T$ have non-empty intersection, then $Q_{1}\cap Q_{2}\in T$.
    \end{enumerate}
    The geometric realization $M$ of $T$ is the following subset of $\mathbb{R}^{n}$
    \begin{equation*}
        M=\bigcup_{Q\in T}Q.
    \end{equation*}
\end{definition}

\begin{definition}
    We say that $M\subseteq\mathbb{R}^{n}$ is a rectangular polyhedron if $M$ is the geometric realization of a rectangular complex as in Definition \ref{Def rectangular complex}. Given a rectangular polyhedron, a rectangular structure on $M$ is a rectangular complex $T$ such that $M$ is the geometric realization of $T$. As the pair $(M,T)$ of a rectangular polyhedron and a rectangular structure on it is uniquely determined by the rectangular complex $T$, we will also call $(M,T)$ a rectangular complex.
\end{definition}

\begin{remark}
    A rectangular polyhedron is a Euclidean polyhedron whose cells are rectangles instead of linear simplices (see Definition \ref{Def Euclidean polyhedron}).
\end{remark}

\begin{remark}
    Notice that a rectangular polyhedron $M$ admits different rectangular structures. For example, we can fix one of such structures and obtain others by subdividing each edge in $q$ equal parts for any $q\in\mathbb{N}$. Such operation is a $q$-refinement of the rectangular structure.
\end{remark}

\begin{definition}\label{Def q-refinement}
    Given a rectangular complex $(M,T)$ and a number $q\in\mathbb{N}$, we denote by $\Refine^{q}T$ the rectangular structure on $M$ obtained by a $q$-refinement of $M$.
\end{definition}

\begin{definition}
    A rectangular domain in $\mathbb{R}^{n}$ is a rectangular polyhedron for which each maximal cell has dimension $n$. 
\end{definition}

\begin{remark}
    Notice that rectangular domains are piecewise smooth Riemannian manifolds with boundary, as each rectangle can be subdivided into linear simplices. The constructions performed to prove the Parametric Coarea Inequality are easier to describe in the presence of a rectangular structure than when we have a triangulation. That is why this result is first proved in detail for rectangular domains. In Section \ref{Extension to triangulable domains} we show how to extend it to almost $1$-Lipschitz triangulable piecewise smooth Riemannian manifolds.
\end{remark}


\begin{definition}
    Given two rectangular complexes $(M_{1},T_{1})$ and $(M_{2},T_{2})$, a rectangular map $f:(M_{1},T_{1})\to(M_{2},T_{2})$ is a function $f:M_{1}\to M_{2}$ such that for every $Q_{1}\in T_{1}$, $f(Q_{1})$ is a rectangle $Q_{2}\in T_{2}$ and the map $f|_{Q_{1}}:Q_{1}\to Q_{2}$ is the composition of an orthogonal projection onto a face of $Q_{1}$ and an orthogonal transformation of $\mathbb{R}^{n}$.
\end{definition}

\begin{definition}\label{Def width cubical complex}
    We say that a rectangular complex is uniform if each of its edges has the same length $\rho$. In that case, we say that the complex has width $\rho$.
\end{definition}

\begin{definition}\label{Def width rectangular complex}
    Given a rectangular complex $(M,T)$, we say that it has width at least $\rho$ if each of its edges has length greater or equal to $\rho$.
\end{definition}

\begin{definition}
    Given two rectangular structures $T$ and $T'$ on a polyhedron $M$, we say that $T'$ is a refinement of $T$ if for every $Q'\in T'$ there exists $Q\in T$ such that $Q'\subseteq Q$.
\end{definition}

The following lemma provides for each sufficiently small $\varepsilon>0$ a map $S:M\to M$ with Lipschitz constant very close to $1$, homotopic to the identity and having the following property: given a subcomplex $A$ of $(M,T)$, the $\varepsilon$-tubular neighborhood of $A$ is mapped onto $A$. It is crucial in the proof of the Parametric Coarea Inequality on rectangular domains. All tubular neighborhoods are taken with respect to the $|\cdot|_{\infty}$-metric.
\begin{lemma}\label{Lemma map S}  Let $(M,T)$ be a rectangular complex in $\mathbb{R}^{n}$ of width at least $\rho$. Let $0<\varepsilon<\frac{\rho}{3}$. Then there exists a map $S=S_{(M,T,\varepsilon)}:M\to M$ such that
    \begin{enumerate}
        \item $S$ is $\frac{\rho}{\rho-2\varepsilon}$-Lipschitz.
        \item Given $Q\in T$, $S(N_{\varepsilon}Q)\subseteq Q$ and in particular $S(Q)\subseteq Q$.
        \item There exists a refinement $\tilde{T}$ of $T$ of width at least $\varepsilon$ such that $S:(M,\tilde{T})\to(M,T)$ is a rectangular map.
    \end{enumerate}
    In addition, given $q\in\mathbb{N}$ there exists a refinement $T'$ of $\tilde{T}$ of width at least $\frac{\rho-2\varepsilon}{q}$ such that $S:(M,T')\to (M,\Refine^{q}T)$ is a rectangular map.
\end{lemma}

\begin{proof}
    We will first inductively construct $S$ in $N_{\varepsilon}M_{j}$ (here $M_{j}$ denotes the $j$-skeleton of $M$ with the structure $T$) for each $0\leq j\leq n$, denoting $S_{j}=S|_{N_{\varepsilon}M_{j}}$. For the $0$-skeleton, observe that $N_{\varepsilon}M_{0}$ is the disjoint union of the cubes $B_{\infty}(v,\varepsilon)$ for $v\in M_{0}$. On each cube, set $S$ to be the constant function with image $v$. Suppose $S$ has been defined in $N_{\varepsilon}M_{j-1}$ and let $Q$ be a $j$-face of $(M,T)$. Let
    \begin{equation*}
        Q^{\varepsilon} = \{x\in Q:\dist_{\infty}(x,\partial Q)\geq\varepsilon\}.
    \end{equation*}
    Given $A\subseteq Q$, denote
    \begin{equation*}
        N_{\varepsilon}^{\perp}A  =\{x+v:x\in A,v\perp Q,|v|_{\infty}\leq\varepsilon\}.
    \end{equation*}
    We will assume that $(M,T)$ is uniform with $\rho=2$ in order to avoid introducing extra notation due to the different lengths of the edges of the rectangles in $T$. Notice that $N_{\varepsilon}Q\setminus N_{\varepsilon}M_{j-1}=N_{\varepsilon}^{\perp} Q^{\varepsilon}$, thus it suffices to extend $S$ to $N_{\varepsilon}^{\perp}Q^{\varepsilon}$. In that set, define
    \begin{equation*}
        S_{j}(x+v)=\frac{1}{1-\varepsilon}\cdot_{Q}x.
    \end{equation*}
     We need to check that $S_{j}$ coincides with $S_{j-1}$ in
    \begin{equation*}
        N_{\varepsilon}^{\perp}Q^{\varepsilon}\cap N_{\varepsilon}M_{j-1}=N_{\varepsilon}^{\perp}\partial Q^{\varepsilon}=\{x+w:x\in\partial Q^{\varepsilon},w\perp Q,|w|_{\infty}\leq\varepsilon\}.
    \end{equation*}
    
    If $j=1$, $\partial Q=\{q_{1},q_{2}\}$ and $\partial Q^{\varepsilon}=\{q_{1}',q_{2}'\}$ with $q_{i}'=(1-\varepsilon)\cdot_{Q} q_{i}$ so $S_{j}(q_{i}'+w)=\frac{1}{1-\varepsilon}\cdot_{Q} q_{i}'=q_{i}=S_{j-1}(q_{i'}+w)$ as desired.

    If $j>1$, given $x+w\in N_{\varepsilon}^{\perp}\partial Q^{\varepsilon}$ let $F$ be a $(j-1)$-dimensional face containing $\frac{1}{1-\varepsilon}\cdot_{Q}x$. We can write $x=x'+v$ where $x'\in F^{\varepsilon}$ is the orthogonal projection of $x$ onto $F$ and $v\in F^{\perp}\cap Q$ verifies $|v|_{\infty}\leq\varepsilon$. Therefore
    \begin{equation*}
        S_{j-1}(x+w)=S_{j-1}(x'+(v+w))=\frac{1}{1-\varepsilon}\cdot_{F}x'=\frac{1}{1-\varepsilon}\cdot_{Q}x
    \end{equation*}
    thus $S_{j-1}(x+w)=S_{j}(x+w)$. This allows us to define $S$ in $N_{\varepsilon}M_{n-1}$. Given an $n$-dimensional cube $Q$ of the grid, we have $S(x)=\frac{1}{1-\varepsilon}\cdot_{Q}x$ for each $x\in\partial Q^{\varepsilon}$. We extend $S$ to $Q^{\varepsilon}$ with the same formula and hence we obtain a map with the desired properties.

    Now we construct the structure $\tilde{T}$. For each pair $(Q,F)$ where $Q,F\in T$, $\dim(Q)=n$, $0\leq \dim(F)\leq n$ and $F\subseteq Q$, we consider an $n$-rectangle $\tilde{Q}=\tilde{Q}_{(Q,F)}$. The collection of such $n$-rectangles will be the top cells of $\tilde{T}$. Their definition depends on the dimension of $F$. When $\dim(F)=0$, $F$ is a vertex $v$ of $T$ and $\tilde{Q}=B_{\infty}(v,\varepsilon)\cap Q$. When $1\leq \dim(F)\leq n-1$, $\tilde{Q}=N_{\varepsilon}^{\perp}F^{\varepsilon}\cap Q$. When $\dim(F)=n$, we have $F=Q$ and we set $\tilde{Q}=Q^{\varepsilon}$. It follows that $\tilde{T}$ is a refinement of $T$ of width at least $\varepsilon$ and $S:(M,\tilde{T})\to (M,T)$ is a rectangular map.

    Finally, to obtain $T'$ we subdivide the different cells $\tilde{Q}_{(Q,F)}$ of $\tilde{T}$ in the following way. When $\dim(F)=0$, no subdivision is done. When $1\leq j=\dim(F)\leq n-1$, $\tilde{Q}=N_{\varepsilon}^{\perp} F_{\varepsilon}\cap Q\cong F^{\varepsilon}\times[0,\varepsilon]^{n-j}$ and corresponding cells of $\tilde{T}$ are of the form $F'\times[0,\varepsilon]^{n-j}$ where $F'$ is a top cell of the $q$-refinement of $F^{\varepsilon}$. When $\dim(F)=n$, we perform a $q$-refinement of $Q^{\varepsilon}$.
\end{proof}

\begin{remark}
    When working with not-uniform rectangular complexes $(M,T)$ of width $\rho$, we replace each homotecy $\frac{1}{1-\varepsilon}\cdot_{Q}x$ by the linear map $L:Q_{\varepsilon}\to Q$ which stretches the central part $E^{\varepsilon}$ each edge $E$ of $Q$ of length $\tilde{\rho}$ onto $E$, with a Lipschitz constant of $\frac{\tilde{\rho}}{\tilde{\rho}-2\varepsilon}$. As each $\tilde{\rho}$ is greater or equal to $\rho$, these maps are still $\frac{\rho}{\rho-2\varepsilon}$-Lipschitz.
\end{remark}

\begin{remark}
    Given a cell $Q$ of the cubical domain $M$, the map $S|_{Q}$ can be described as follows: first expand $Q$ by a factor of $\frac{1}{1-\varepsilon}$ and then take the nearest point projection onto $Q$. This description will help us extend Lemma \ref{Lemma map S} to triangulable polyhedra, where we can not make use of the $d_{\infty}$ distance to obtain a nice descriptions of suitable tubular neighborhoods of the faces as in the case of cubical domains. The previous is done in Section \ref{Extension to triangulable domains}.
\end{remark}

\section{Motivation and sketch of the proof when $\dim(M)=4$}\label{Section n=4}

In this section, we will explain the main steps of the proof of the Parametric Coarea Inequality for $1$-cycles in $4$-manifolds (Theorem \ref{Thm Parametric Coarea} for $n=4)$. We will leave some details for the next section, where the argument is generalized to arbitrary $n$-manifolds.

Let $(M,g)$ be an $4$-dimensional smooth Riemannian manifold with boundary and denote $\Sigma=\partial M$ (the case when $M$ is piecewise smooth will be analyzed later). Let $\alpha\in (0,1)$. By Theorem \ref{Thm delta localized approximation}, we can assume that we start from a $\delta$-localized family of absolute cycles $F:X^{p}\to\mathcal{I}_{1}(M)$ with $\support(\partial F(x))\subseteq\Sigma$ for which we do not have an upper bound on $\mass(\partial(F(x)))$. For convenience, we will denote
\begin{equation*}
    \mathcal{I}_{1}(M)_{\Sigma}=\{\eta\in\mathcal{I}_{1}(M):\support(\partial\eta)\subseteq\Sigma\}
\end{equation*}
provided with the flat topology inherited from $\mathcal{I}_{1}(M)$. The goal is to produce a new family $F':X^{p}\to\mathcal{I}_{1}(M)_{\Sigma}$ such that $\mass(F'(x))$ is not too big with respect to $M(F(x))$ and also $\mass(\partial F'(x))$ is bounded. To be precise, we will define for every $p\in\mathbb{N}$ small coefficients
\begin{equation*}
    \varepsilon_{1}>>\rho_{1}>>\varepsilon_{2}>>\rho_{2}>0
\end{equation*}
with $\lim_{p\to\infty}\varepsilon_{1}(p)=0$ and $\rho_{2}>>\delta $, and we will construct $F'$ such that
\begin{itemize}
    \item $\mass(F'(x))\leq \mass(F(x))(1+\beta+\frac{\rho_{1}}{\varepsilon_{1}}+2\frac{\rho_{2}}{\varepsilon_{2}})+c\Vol_{3}(\Sigma)p^{\frac{2}{3}+\alpha}$
    \item $\mass(\partial F'(x))\leq c\Vol_{3}(\Sigma)p^{1+\alpha}$
\end{itemize}
for some $\beta=\beta(p)$ which verifies $\lim_{p\to\infty}\beta(p)=0$. Using the asymptotic behavior of the coefficients $\varepsilon_{l},\rho_{l}$ and $\beta$, the previous will imply Theorem \ref{Thm Parametric Coarea} for $n=4$. We are going to describe the construction of the map $F'$, which will clarify where the coefficients and the powers of $p$ come from. We will start by sketching the construction for $1$-cycles in $3$-manifolds done by Guth and Liokumovich in \cite{GL22}, to later point out the difficulties in extending those constructions when $\dim(M)=4$ and how to overcome them. As it is aimed to be later generalized to higher dimensions, the exposition that follows is slightly different from that in \cite{GL22}. First, using the Coarea Inequality for each $x\in X^{p}_{0}$ we can choose $0< s(x)<\varepsilon_{1}$  so that
\begin{equation}\label{coarea ineq1}
    \mass(F(x)\llcorner D_{s(x)}\Sigma)\leq \frac{\mass(F(x))}{\varepsilon_{1}}
\end{equation}
where according to Definition \ref{Def tubular neighborhood}, given $A\subseteq M$ and $\varepsilon>0$, we denote
\begin{align*}
    N_{\varepsilon}A & =\{x\in M:\dist(x,A)<\varepsilon\}\\
    D_{\varepsilon}A & =\{x\in M:\dist(x,A)=\varepsilon\}.
\end{align*}
We can also assume (by worsening (\ref{coarea ineq1}) slightly) that $D_{s(x)}\Sigma$ does not intersect any ball in the $\delta$-admissible families corresponding to the cells containing $x$.

Let $S_{1}:M\to M$ be a $(1+\beta_{1})$-Lipschitz map which retracts $N_{\varepsilon_{1}}\Sigma$ onto $\Sigma$. This can be constructed by taking $\rho_{0}=\rho_{0}(p)>>\varepsilon_{1}(p)$ less than the injectivity radius of $(M,g)$ (for sufficiently large $p$) and defining $S_{1}$ in the following way. We can identify $N_{\rho_{0}}\Sigma\cong\Sigma\times[0,\rho_{0}]$ where the identification is a $(1+\rho_{0})$-bilipschitz diffeomorphism, and set
\[
S_{1}(y)=\begin{cases}
    (x,\phi(t)) 
    & \text{if }y=(x,t)\in N_{\rho_{0}}\Sigma\\
    y & \text{if }y\notin N_{\rho_{1}}\Sigma
\end{cases}
\]
where
\[
\phi(t)=\begin{cases}
    0 & \text{if }0\leq t\leq\varepsilon_{1}\\
    \frac{\rho_{0}}{\rho_{0}-\varepsilon_{1}}(t-\varepsilon_{1}) & \text{if }\varepsilon_{1}\leq t\leq\rho_{0}
\end{cases}
\]
In this case, we obtain $1+\beta_{1}=\frac{1}{1-\frac{\varepsilon_{1}}{\rho_{0}}}$. After cutting $F(x)$ at $s(x)$ and applying $S_{1}$, we can get a new family $\tilde{F}$ defined in the $0$-skeleton of $X$ such that $\partial\tilde{F}(x)$ is supported in $\Sigma$. To be precise,
\begin{equation*}
    \tilde{F}(x)=C(F(x);s(x)):=S_{1}(F(x)\llcorner(M\setminus N_{s(x)}\Sigma)).
\end{equation*}
Observe that $\tilde{F}:X^{p}_{0}\to\mathcal{I}_{1}(M)_{\Sigma}$ has mass bounded by $(1+\beta_{1})\mass(F(x))$ and boundary mass bounded by $\frac{\mass(F(x))}{\varepsilon_{1}}$. The goal is to extend this map to the whole $X^{p}$ without worsening those bounds too much (i.e. obtaining estimates as the ones stated above). In \cite{GL22}, the extension was done inductively skeleton by skeleton, having to interpolate between the values of $\tilde{F}(x)$ at different vertices $x$ of each cell $C$ of $X$. Observe that as the family is $\delta$-localized, if $x$ and $y$ are the vertices of a $1$-cell $C$; the difference $\tilde{F}(y)-\tilde{F}(x)$ equals some absolute $1$-cycle 
\begin{align*}
    e(y,x) & =C(F(y);s(y))-C(F(x);s(y))\\
        & =S_{1}(F(y)-F(x)\llcorner(M\setminus N_{s(y)}\Sigma))
\end{align*}
supported in a $(1+\beta_{1})\delta$-admissible family $\{\tilde{U}_{i}^{C}\}_{i\in\tilde{I}_{C}}$ associated to $C$ (which is obtained from the family $\{U^{C}_{i}\}_{i\in I_{C}}$ where $F|_{C}$ is localized by removing those elements intersecting $N_{s(y)}\Sigma$ and then applying $S_{1}$) plus something supported in $\Sigma=\partial M$, which to be precise is
\begin{align*}
    C(F(x);\Delta s(x,y)) & =C(F(x);s(y))-C(F(x);s(x))\\
    & = S_{1}(F(x)\llcorner(N_{s(x)}\Sigma\setminus N_{s(y)}\Sigma))
\end{align*}
if $s(x)\geq s(y)$. The former is originated by (possibly) doing different cuts for $x$ and for $y$. The important part to interpolate between $C(F(x);s(x))$ and $C(F(x);s(y))$ is how to contract $C(F(x);\Delta s(x,y))$, as $e(y,x)$ can be contracted radially within each ball of $\{\tilde{U}_{i}^{C}\}_{i\in \tilde{I}_{C}}$ without increasing the mass. For simplicity, denote $\tau=C(F(x);\Delta s(x,y))$. We need to contract the $1$-chain $\tau$ whose boundary has bounded mass (by $2\frac{\mass(F(x))}{\varepsilon_{1}}$) in a way that does not increase $\mass(\tau)$ and $\mass(\partial\tau)$ too much. For that purpose, in \cite{GL22} a triangulation of $\Sigma$ was chosen with width $\rho_{1}=p^{-\frac{1}{2}}$ and then $\tau$ was contracted radially cell by cell. But if we do that directly, $\mass(\partial\tau)$ might become arbitrarily large. In Figure \ref{Figure 3}, the situation where $\Sigma$ has only two simplices $Q_{1}$ and $Q_{2}$ is represented. The upper picture part shows the $1$-cycle $\tau$ supported in $\Sigma$, the black dots correspond to $\partial\tau$ and the red dots to $\tau\llcorner\Sigma^{1}$, where we denote by $\Sigma^{1}$ the codimension-$1$ skeleton of $\Sigma$ (which in this case is $1$-dimensional). The lower picture shows the cycles we see as we contract $\tau\llcorner Q_{2}$ linearly and leave $\tau\llcorner Q_{1}$ fixed. Notice that the mass of the boundary is increased by twice $\mass(\tau)\llcorner\partial Q_{2}$, which might be arbitrarily large.

\begin{figure}[h]\label{Figure 3}
\centering
\includegraphics[scale = 0.7]{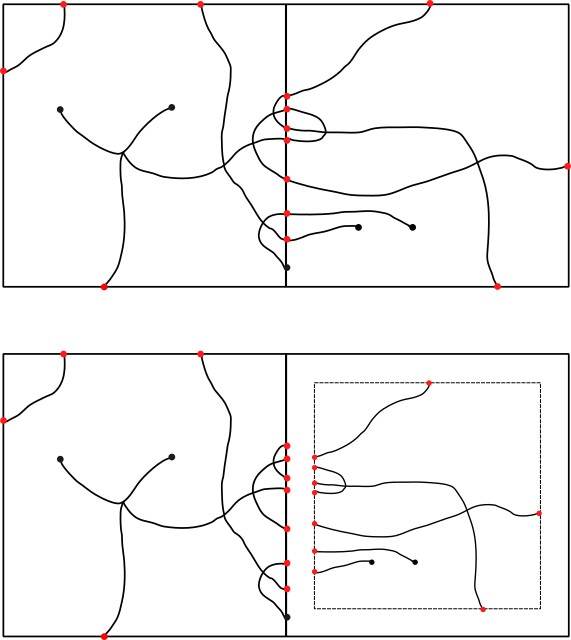}
\caption{}
\end{figure}

The solution proposed in \cite{GL22} was to add for each cell $D$ of $\Sigma$ a $1$-chain $\eta=f_{D}(x)-f_{D}(y)$ supported in $\partial D$ whose boundary cancels most of $\tau\llcorner\partial D$. To be precise, the $1$-chains $f_{D}(x)$ on $\partial D$ are defined as
\begin{equation*}
    f_{D}(x)=\sum_{E\text{ }1-\text{face of }D}\Cone_{q_{E}}(\partial C(F(x);s(x))_{D}\llcorner\interior(E))
\end{equation*}
where for each $1$-face $E$ of $D$ the point $q_{E}$ is chosen to be one of its vertices. We denoted $\partial C(F(x);s(x))_{D}=\partial[C(F(x);s(x))\llcorner D]$. Hence $\partial f_{D}(x)-\partial C(F(x);s(x))_{D}\llcorner\partial D$ is supported in the set of vertices of $D$. Notice that the other part of the boundary of $C(F(x);s(x))\llcorner D$ is $\partial C(F(x);s(x))\llcorner\interior(D)$ which is bounded in mass by $\frac{\mass(F(x))}{\varepsilon_{1}}$ when added over $D$. Then, denoting $q_{D}$ the center of a face  $D$ and setting
\begin{equation}\label{F' in dim 3}
    F'(x)=C(F(x);s(x))+\sum_{D\text{ }2-\text{cell of }\Sigma}\Cone_{q_{D}}(\partial C(F(x);s(x))\llcorner\interior(D))+ f_{D}(x)
\end{equation}
for $x\in X_{0}$, and interpolating between the vertices by radially rescaling their differences, it is possible to extend $F'$ to the $1$-skeleton of $X$. It is key for this argument that as the $f_{D}(x)$ are $1$-chains supported in $\Sigma^{1}$, by mod 2 cancellation the mass of their sum is bounded by the total mass (or length) of $\Sigma_{1}$ which is at most $C(\Sigma)p^{\frac{1}{2}}$ as the width of the triangulation is $\rho_{1}=p^{-\frac{1}{2}}$. But when $\dim(M)\geq 4$ and hence $\dim(\Sigma)\geq 3$, no matter how we define the $f_{D}(x)$ such cancellation of mass does not happen as $\dim(\partial D)\geq 2$ for every $D$ and the $f_{D}(x)$ are $1$-chains. Moreover, $\partial C(F(x);\Delta s(x,y))_{D}\llcorner\partial D$ may not admit a filling in $\partial D$ of bounded mass as $\mass(\partial C(F(x);\Delta s(x,y))_{D}\llcorner\partial D)$ could be arbitrarily large. Therefore, we need to do something else when $M$ has dimension $4$ or higher.

As the problem comes from the fact that $\partial C(F(x);s(x))_{D}\llcorner\partial D$ has unbounded mass, what we will do is to introduce a new cut to obtain control on the later, using the Coarea Inequality again. To be precise, apart from performing the cut at distance $s_{1}(x)$ away from $A_{1}=\Sigma$ with $0< s_{1}(x)< \varepsilon_{1}$, we will perform a certain cut $s_{2}(x)$ away from $A_{2}=\Exp(\Sigma_{2}\times[0,\varepsilon_{1}])$ (here $\Exp:\Sigma\times[0,\overline{\varepsilon}_{1}]\to M$ denotes the exponential map and $\Sigma_{2}$ denotes the $2$-skeleton of $\Sigma$) with $0< s_{2}(x)< \varepsilon_{2}$ such that
\begin{equation}\label{coarea ineq2}
    \mass(F(x)\llcorner D_{s_{2}(x)}A_{2})\leq\frac{\mass(F(x))}{\varepsilon_{2}}.
\end{equation}
Again we can  assume (by worsening (\ref{coarea ineq2}) slightly) that $D_{s_{2}(x)}A_{2}$ does not intersect any ball in the $\delta$-admissible families corresponding to the cells containing $x$. We define
\begin{equation*}
    \overline{C}(F(x);s_{1}(x),s_{2}(x))=F(x)\llcorner[M\setminus(N_{s_{1}(x)}A_{1}\cup N_{s_{2}(x)}A_{2})]
\end{equation*}
and
\begin{align*}
    \overline{C}(F(x);\Delta s_{1}(x,y),s_{2}(x)) & =\overline{C}(F(x);s_{1}(y),s_{2}(x))-\overline{C}(F(x);s_{1}(x),s_{2}(x))\\
    \overline{C}(F(x);s_{1}(y),\Delta s_{2}(x,y)) & =\overline{C}(F(x);s_{1}(y),s_{2}(y))-\overline{C}(F(x);s_{1}(y),s_{2}(x)).
\end{align*}
In general, $\Delta s_{l}(x,y)$ will represent $s_{l}(y)-s_{l}(x)$ for $l=1,2$,  and feeding a certain operator (like $\overline{C}$) at its $s_{l}$ coordinate with $\Delta s_{l}(x,y)$ will mean the evaluation of that operator at $s_{l}(y)$ minus the evaluation at $s_{l}(x)$. Let $S_{2}:M\to M$ a $(1+\beta_{2})$-Lipschitz map which retracts $N_{\varepsilon_{2}}A_{2}$ onto $A_{2}$. The construction of a map with such properties is non trivial in general, and is done in detail for the cases when $M$ is a rectangular domain in Section \ref{Section rectangular complexes} and when $M$ is a Euclidean polyhedron in Section \ref{Extension to triangulable domains}. In both cases, we take advantage of having a triangulation (by rectangular or triangular simplices) to define such maps. For now, we will just assume that such $S_{2}$ exists and continue describing the proof. We define
\begin{equation*}
    C(F(x);s_{1}(x),s_{2}(x))=S_{1}\circ S_{2}(\overline{C}(F(x);s_{1}(x),s_{2}(x))).
\end{equation*}
What we are doing is to first remove two portions of $F(x)$ (the part contained in $N_{s_{1}(x)}A_{1}$ and the part contained in $N_{s_{2}(x)}A_{2}$), then to expand (by applying $S_{2}$) the chain around $A_{2}$ to fill in the gap $N_{s_{2}(x)}A_{2}$ we generated  and finally to expand (by applying $S_{1}$) it around $A_{1}=\Sigma$ to fill in the gap $N_{s_{1}(x)}A_{1}$ we generated . By doing this, we recover a chain whose boundary lies on $\Sigma$ (as after applying $S_{2}$ the boundary is supported in $N_{\varepsilon_{1}}\Sigma$, which is retracted to $\Sigma$ by $S_{1}$) and is bounded in mass. But in addition, adding the second cut $s_{2}(x)$ allows to have control over the mass of the boundary of the chains
\begin{equation*}
    \overline{C}(F(x);\Delta s_{1}(x,y),s_{2}(x))_{D}=\overline{C}(F(x);\Delta s_{1}(x,y),s_{2}(x))\llcorner \Exp(D\times[0,\varepsilon_{1}])
\end{equation*}
for each $3$-cell $D$ of $\Sigma$ because of the following reasons. We can express
\begin{align*}
    N_{\varepsilon_{1}}A_{1}\setminus N_{s_{2}(x)}A_{2} & =\bigcup_{D\text{ }3-\text{cell of }\Sigma}\Exp(D\times[0,\varepsilon_{1}])\setminus N_{s_{2}(x)}A_{2}\\
    & =\bigcup_{D\text{ }3-\text{cell of }\Sigma}\Exp(D\times[0,\varepsilon_{1}])\setminus N_{s_{2}(x)}\Exp(\partial D\times[0,\varepsilon_{1}])
\end{align*}
as a disjoint union of compact domains. The $1$-chain $\overline{C}(F(x);\Delta s_{1}(x,y),s_{2}(x))$ is supported in $N_{\varepsilon_{1}} A_{1}\setminus N_{s_{2}(x)}A_{2}$, therefore we can write
\begin{align*}
    \overline{C}(F(x);\Delta s_{1}(x,y),s_{2}(x)) & = \sum_{D\text{ }3-\text{cell of }\Sigma}\overline{C}(F(x);\Delta s_{1}(x,y),s_{2}(x))\llcorner \Exp(D\times[0,\varepsilon_{1}])\\
    & =\sum_{D\text{ }3-\text{cell of }\Sigma}\overline{C}(F(x);\Delta s_{1}(x,y),s_{2}(x))_{D}
\end{align*}
and the boundary of $\overline{C}(F(x);\Delta s_{1}(x),s_{2}(x))_{D}$ has one part which comes from intersecting $F(x)$ with $N_{s_{2}(x)}A_{2}$ (whose mass adds up to at most $\frac{\mass(F(x))}{\varepsilon_{2}}$ when summed over $D$) and two other parts which come from intersecting $F(x)$ with $N_{s_{1}(x)}A_{1}$ and $N_{s_{1}(y)}A_{1}$ (being the total mass of each of them at most $\frac{\mass(F(x))}{\varepsilon_{1}}$ when added over $D$). These properties are inherited by the chains
\begin{equation*}
    C(F(x);\Delta s_{1}(x,y),s_{2}(x))_{D}=S_{1}\circ S_{2}(\overline{C}(F(x);\Delta s_{1}(x,y),s_{2}(x))_{D})
\end{equation*}
which are supported in $D$.

As in the case $n=3$, by $\delta$-localization of $F$ the interpolation between $C(F(x);s_{1}(x),s_{2}(x))$ and $C(F(y);s_{1}(y),s_{2}(y))$ reduces to contracting 
\begin{multline*}
    C(F(x);s_{1}(y),s_{2}(y))-C(F(x);s_{1}(x),s_{2}(x))\\
    =C(F(x);\Delta s_{1}(x,y),s_{2}(y))+C(F(x);s_{1}(x),\Delta s_{2}(x,y)).
\end{multline*}
Using the fact that 
\begin{equation*}
    C(F(x);\Delta s_{1}(x,y),s_{2}(y))=\sum_{D\text{ }3-\text{cell of }\Sigma}C(F(x);\Delta s_{1}(x,y),s_{2}(y))_{D}
\end{equation*}
where each  $C(F(x);\Delta s_{1}(x,y),s_{2}(y))_{D}$ has controlled boundary mass, it is possible to contract $C(F(x);\Delta s_{1}(x,y),s_{2}(y))$ radially cell by cell of $\Sigma$. What happens with $C(F(x);s_{1}(x),\Delta s_{2}(x,y))$?

The first important observation is that by definition
\begin{equation*}
    C(F(x);s_{1}(x),\Delta s_{2}(x,y))=S_{1}\circ S_{2}(\overline{C}(F(x);s_{1}(x),\Delta s_{2}(x,y)))
\end{equation*}
where if we assume for simplicity that $s_{l}(x)\geq s_{l}(y)$ for $l=1,2$,
\begin{equation*}
    \overline{C}(F(x);s_{1}(x),\Delta s_{2}(x,y))=F(x)\llcorner[N_{s_{2}(x)}A_{2}\cap M\setminus (N_{s_{1}(x)}A_{1}\cup N_{s_{2}(y)}A_{2})]
\end{equation*}
which is supported in $N_{s_{2}(x)}A_{2}\subseteq N_{\varepsilon_{2}}A_{2}$. Therefore, as $S_{2}$ retracts $N_{\varepsilon_{2}}A_{2}$ onto $A_{2}$; $S_{2}(\overline{C}(F(x);s_{1}(x),\Delta s_{2}(x,y)))$ is supported in $A_{2}=\Exp(\Sigma_{2}\times[0,\varepsilon_{1}])$ and hence $C(F(x);s_{1}(x),\Delta s_{2}(x,y))$ is supported in $\Sigma_{2}$. Therefore, to be able to extend $F'$ to the $1$-skeleton of $X$ we need to contract the $1$-chain $C(F(x);s_{1}(x),\Delta s_{2}(x,y))$ which lives in the $2$-dimensional PL submanifold $\Sigma_{2}$ of $M$ and has bounded boundary mass. In order to do that, for each $x\in X_{0}$ we can add $1$-chains $f_{E}(x)$ supported in $\partial E$ for each $2$-cell $E$ of $\Sigma_{2}$ and proceed as we described above for the case $\dim(\Sigma)=2$, following \cite{GL22}.

In the discussion so far, we were only concerned about how to interpolate between the values of $F'$ at two different vertices $x$ and $y$. This is sufficient to extend $F'$ to the $1$-skeleton of $X$, but not to higher dimensional skeletons. In general, we will have a $p'$-dimensional cell $C$ of $X$ with $0\leq p'\leq p$ and for each $x\in C$ we will want to write $F'(x)$ as some sort of convex combination of the values of $F'$ at the vertices of $C$. We proceed to motivate how to obtain such an interpolation formula. We start with the case $\dim(M)=3$ by describing what was done in \cite{GL22}. 

Given a $p'$-cell $C$ of $X$, pick an enumeration $x_{1},...,x_{2^{p'}}$ of its vertices such that $s(x_{1})\geq s(x_{2})\geq...\geq s(x_{2^{p'}})$. Given a $2$-cell $D$ of $\Sigma$ and a number $2\leq i\leq 2^{p'}$, we would like to assign a coefficient $\tilde{\mu}_{i}^{D}(x)\in[0,1]$ to each $x\in C$ so that $F'(x)$ is given by
\begin{equation*}
    F'(x)=e(x)+C(F(x_{1});s(x_{1}))+\sum_{D\text{ }2-\text{cell of }\Sigma}\sum_{i=2}^{2^{p'}}\tilde{\mu}^{D}_{i}(x)\cdot C(F(x_{1});\Delta s(x_{i}))_{D}
\end{equation*}
where $\Delta s(x_{i})=s(x_{i})-s(x_{i-1})$ and $e(x)$ is a short family of $1$-cycles supported in $\{\tilde{U}_{i}^{C}\}_{i}$ (and we assume for simplicity that $C(F(x);s(x))$ is transverse to $\Sigma_{1}$ for every $x\in X_{0}$ so that $\sum_{D}C(F(x);\Delta s(x_{i}))_{D}=C(F(x);\Delta s(x_{i}))$ for every $x\in X_{0}$). Notice that $\tilde{F}(x_{i+1})-\tilde{F}(x_{i})=C(F(x_{1});\Delta s(x_{i}))+e_{i}$ for some $e_{i}$ supported in the admissible family $\{\tilde{U}_{i}^{C}\}_{i}$ of the cell $C$. Hence, we can interpolate between the $F(x_{i})$ by rescaling $C(F(x_{1});\Delta s(x_{i}))$ by a factor $\tilde\mu_{i}^{D}(x)\in[0,1]$ at each cell $D$ of $\Sigma$ plus a $1$-chain $e(x)$ supported $\bigcup_{i}\tilde{U}_{i}^{C}$. That is where the formula above came from. The problem with this definition is that $\mass(\partial F'(x))$ is unbounded, as already the boundary of $C(F(x_{1});\Delta s(x_{i}))_{D}$ may have unbounded mass for each fixed value of $i$ as discussed before. To overcome this difficulty, the solution proposed in \cite{GL22} was to add cones over $\partial C(F(x_{1});\Delta s(x_{i}))_{D}$ as in (\ref{F' in dim 3}). Just taking the cone over $\partial C(F(x_{1});\Delta s(x_{i}))\llcorner D$ centered at the central point $q_{D}$ of $D$ does not work as $\partial C(F(x_{1});\Delta s(x_{i})_{D}\llcorner\partial D$ may have arbitrarily large mass. So we need to give a different treatment to the points of $\partial C(F(x_{1});\Delta s(x_{i})_{D}$ in $\interior(D)$ than to those in $\partial D$. For the first ones, we will take a cone centered at the central point $q_{D}$ of $D$, and for the second ones the center will be at some vertex $q_{E}$ of the edge $E$ in whose interior the point lies (which corresponds exactly to the definition of the $f_{D}(x)$ given above). This motivates the following definitions for each $x\in X_{0}$:
\begin{multline*}
    A(F(x);s(x))=C(F(x);s(x))+\sum_{D\text{ }2-\text{cell of }\Sigma}\bigg[\Cone_{q_{D}}(\partial C(F(x);s(x))\llcorner\interior(D))\\
    +\sum_{E\text{ }1-\text{face of }D}\Cone_{q_{E}}(\partial C(F(x);s(x))_{D}\llcorner\interior(E))\bigg]
\end{multline*}
and
\begin{multline*}
    A(F(x);\Delta s(x_{i}))_{D}=C(F(x);\Delta s(x_{i}))_{D}+\Cone_{q_{D}}(\partial C(F(x);\Delta s(x_{i}))_{D}\llcorner\interior(D))\\
    +\sum_{E\text{ }1-\text{face of }D}\Cone_{q_{E}}(\partial C(F(x);\Delta s(x_{i}))_{D}\llcorner\interior(E)).
\end{multline*}
Notice that
\begin{equation*}
    \sum_{D\text{ }2-\text{cell of }\Sigma}A(F(x);\Delta s(x_{i}))_{D}=A(F(x);\Delta s(x_{i})).
\end{equation*}
In addition,
\begin{equation*}
    \support(\partial A(F(x);\Delta s(x_{i}))_{D})\subseteq \{q_{D}\}\cup V(D)
\end{equation*}
where $V(D)$ denotes the set of vertices of $D$. Therefore, by adding the previous cones, we are ``transferring'' the boundary of $C(F(x);\Delta s_{i}(x))_{D}$ to  a $0$-chain supported in $\{q_{D}\}\cup V(D)$. The price that we have to pay for that is adding the extra mass of those cones. But observe that
\begin{align*}
    \sum_{D}\mass(\Cone_{q_{D}}(\partial C(F(x);s(x))_{D}\llcorner\interior(D))) & \leq \sum_{D}\rho_{1}\mass(\partial C(F(x);s(x))\llcorner\interior(D))\\
    & \leq \rho_{1}\mass(\partial C(F(x);s(x)))\\
    & \leq \rho_{1}\frac{\mass(F(x))}{\varepsilon_{1}}
\end{align*}
and that as $\support(\Cone_{q_{E}}(\partial C(F(x);s(x)))_{D}\llcorner\interior(E)))\subseteq E\subseteq \Sigma_{1}$, by mod $2$ cancellation
\begin{equation*}
    \mass\bigg(\sum_{E}\Cone_{q_{E}}(\partial C(F(x);s(x)))_{D}\llcorner\interior(E))\bigg)\leq\length(\Sigma_{1})\leq C(\Sigma)p^{\frac{1}{2}}
\end{equation*}
using the fact that $\rho_{1}=p^{-\frac{1}{2}}$. This leads to the following formula for $F'$ at a cell $C$
\begin{equation*}
    F'(x)=e(x)+A(F(x_{1});s(x_{1}))+\sum_{D\text{ }2-\text{cell of }\Sigma}\sum_{i=2}^{q}\tilde{\mu}^{D}_{i}(x)\cdot A(F(x_{1});\Delta s(x_{i}))_{D}.
\end{equation*}

In \cite{GL22} it is shown that the coefficients $\tilde{\mu}^{D}_{i}(x)$ can be constructed inductively skeleton by skeleton so that $F'$ is given by the previous formula along each cell $C$ of $X$ and it holds
\begin{enumerate}
    \item $\mass(F'(x))\leq \mass(F(x))(1+\beta_{1}+\frac{\rho_{1}}{\varepsilon_{1}})+C(\Sigma)p^{\frac{1}{2}}$
    \item $\mass(\partial F'(x))\leq C(\Sigma)p$.
\end{enumerate}

\begin{remark}
    The factor $(1+\beta_{1})$ comes from applying the $(1+\beta_{1})$-Lipschitz map $S_{1}$. $\mass(F(x))\frac{\rho_{1}}{\varepsilon_{1}}$ is contributed by the cones supported in the $2$-faces $D$ of $\Sigma$ and $C(\Sigma)p^{\frac{1}{2}}$ by the cones supported in the $1$-faces $E$ of $\Sigma$. The upper bound from $\mass(\partial F'(x))$ comes from the fact that it is supported in the $0$-skeleton of $\Sigma$ union the set of centers of its $2$-faces (plus certain rescaled copies of portions of that set, as we will see later).
\end{remark}

We now want to do something similar for $\dim(M)=4$. Recall that we have $\varepsilon_{1}>>\rho_{1}>>\varepsilon_{2}>>\rho_{2}$. We denote
\begin{itemize}
    \item $\Sigma^{1}=\Sigma$ provided with a triangulation $T_{1}$ of width $\rho_{1}$.
    \item $\Sigma^{2}$ the $2$-skeleton of $(\Sigma^{1},T_{1})$. We provide it with a finer triangulation $T_{2}$ of width $\rho_{2}$.
    \item $\Sigma^{3}$ the $1$-skeleton of $(\Sigma^{2},T_{2})$ with the inherited triangulation $T_{2}$ of width $\rho_{2}$.
    \item $\Sigma^{4}$ the $0$-skeleton of $(\Sigma^{2},T_{2})$.
    \item $\mathcal{F}^{j}$ the set of top dimensional cells of $\Sigma^{j}$.
\end{itemize}
Notice that $\dim(\Sigma^{j})=4-j$ and that if $\interior(\Sigma^{j})$ denotes the union of the interiors of the top-dimensional faces of $\Sigma^{j}$ then
\begin{equation*}
    \Sigma=\bigcup_{j=1}^{4}\interior(\Sigma^{j})
\end{equation*}
where the union is disjoint.

For $x\in X_{0}$, we will have
\begin{equation*}
    F'(x)=C(F(x);s_{1}(x),s_{2}(x))+\text{ certain cones}
\end{equation*}
and then $F'$ will be extended skeleton by skeleton by interpolating its values on the vertices of different cells. Thus we will be interested in interpolating between chains of the form $C(F(x);s_{1}(x),s_{2}(x))$. Using the fact that $F$ is $\delta$-localized, it will be enough to know how to interpolate between $C(\tau;s_{1},s_{2})$ for different values of $s_{1}$ and $s_{2}$ and $\tau\in\mathcal{I}_{1}(M)_{\Sigma}$ given (we will apply that for $\tau=F(x)$ for $x$ a vertex of a cell $C$).

Let us consider a fixed $\tau\in\mathcal{I}_{1}(M)_{\Sigma}$. Two sequences $s_{1}^{1}\geq s_{1}^{2}\geq ...\geq s_{1}^{q}$ and $s_{2}^{1}\geq s_{2}^{2}\geq...\geq s_{2}^{q}$ are said to be sequences of admissible cuts for $\tau$ if
\begin{equation*}
    \mass(\tau\llcorner D_{s_{1}^{i}}A_{1})\leq\frac{\mass(\tau)}{\varepsilon_{1}}
\end{equation*}
for every $1\leq i\leq q$ and
\begin{equation*}
    \mass(\tau\llcorner D_{s_{2}^{j}}A_{2})\leq\frac{\mass(\tau)}{\varepsilon_{2}}
\end{equation*}
for every $1\leq j\leq q$. We will denote $\underline{s}=(s_{i}^{j})_{\substack{1\leq i\leq q \\ 1\leq j\leq q}}$. Fix an admissible sequence of cuts $\underline{s}$ for $\tau$. 

\begin{definition}
    Let $1\leq i,j\leq q$. For $i\geq 2$, we denote
\begin{equation*}
    C(\tau;\Delta_{i}s_{1},s_{2}^{j})=C(\tau;s_{1}^{i},s_{2}^{j})-C(\tau;s_{1}^{i-1},s_{2}^{j}).
\end{equation*}
For $j\geq 2$, let
\begin{equation*}
    C(\tau;s_{1}^{i},\Delta_{j}s_{2})=C(\tau;s_{1}^{i},s_{2}^{j})-C(\tau;s_{1}^{i},s_{2}^{j-1}).
\end{equation*}
For $i,j\geq 2$ we set
\begin{align*}
    C(\tau;\Delta_{i}s_{1},\Delta_{j}s_{2}) & =C(\tau;\Delta_{i}s_{1},s_{2}^{j})-C(\tau;\Delta_{i}s_{1},s_{2}^{j-1})\\
    & =C(\tau;s_{1}^{i},\Delta_{j}s_{2})-C(\tau;s_{1}^{i-1},\Delta_{j}s_{2}).
\end{align*}
\end{definition}

Observe that given $1\leq i_{0},i_{1}\leq q$,
\begin{multline*}
    C(\tau;s_{1}^{i_{0}},s_{2}^{j_{0}})=C(\tau;s_{1}^{1},s_{2}^{1})+\sum_{j=2}^{j_{0}}C(\tau;s_{1}^{1},\Delta_{j}s_{2})\\
    +\sum_{i=2}^{i_{0}}\bigg[C(\tau;\Delta_{i}s_{1},s_{2}^{1})+\sum_{j=2}^{j_{0}}C(\tau;\Delta_{i}s_{1},\Delta_{j}s_{2})\bigg]
\end{multline*}
therefore the idea for the interpolation will be to rescale the summands of the form $C(\tau;s_{2}^{1},\Delta_{j}s_{2})$, $C(\tau;\Delta_{i}s_{1},s_{2}^{1})$ and $C(\tau;\Delta_{i}s_{1},\Delta_{j}s_{2})$ by different factors. Notice that $\support(C(\tau;\Delta_{i}s_{1},s_{2}^{j}))\subseteq \Sigma^{1}$ and $\support(C(\tau;s_{1}^{i},\Delta_{j}s_{2})),\support(C(\tau;\Delta_{i}s_{1},\Delta_{j}s_{2}))\subseteq \Sigma^{2}$. Therefore, we will contract them cell by cell of $\Sigma^{j}$, $j=1$ and $2$ respectively. In order to do that, we need the following definitions.

\begin{definition}\label{Definition C(s1s2)_D}
    Given a top-dimensional cell $D$ of $\Sigma^{1}$, we denote
    \begin{equation*}
        \overline{C}(\tau;\Delta_{i}s_{1},s_{2}^{1})_{D}=\overline{C}(\tau;\Delta_{i}s_{1},s_{2}^{1})\llcorner \Exp(D\times[0,\varepsilon_{1}])
    \end{equation*}
    which has controlled boundary mass as previously discussed. We let
    \begin{equation*}
        C(\tau;\Delta_{i}s_{1},s_{2}^{1})_{D}=S_{1}\circ S_{2}(\overline{C}(\tau;\Delta_{i}s_{1},s_{2}^{1})_{D})
    \end{equation*}
    which is supported in $D$. Additionally,
    \begin{equation*}
        \sum_{D\text{ top cell of }\Sigma^{1}}C(\Delta_{i}s_{1},s_{2}^{1})_{D}=C(\tau;\Delta_{i}s_{1},s_{2}^{1}).
    \end{equation*}
\end{definition}

\begin{definition}
    For each (closed) top cell $E$ of $\Sigma^{2}$, we consider a Borel set $\tilde{E}$ satisfying $\interior(E)\subseteq \tilde{E}\subseteq E$ such that the $\tilde{E}$ are disjoint and
    \begin{equation*}
        \bigcup_{E\in\mathcal{F}^{2}}\tilde{E}=\Sigma^{2}.
    \end{equation*}
\end{definition}

\begin{definition}
    Given a top-dimensional cell $E$ of $\Sigma^{2}$, we denote
    \begin{equation*}
        C(\tau;s_{1}^{1},\Delta_{j}s_{2})_{E}=C(\tau;s_{1}^{1},\Delta_{j}s_{2})\llcorner \tilde{E}
    \end{equation*}
    which is supported in $E$. It holds
    \begin{equation*}
        \sum_{E\in\mathcal{F}^{2}}C(\tau;s_{1}^{1},\Delta_{j}s_{2})_{E}=C(\tau;s_{1}^{1},\Delta_{j}s_{2}).
    \end{equation*}
\end{definition}

Regarding the terms of the form $C(\tau;\Delta_{i}s_{1},\Delta_{j}s_{2})$, we want to split them among cells of $\Sigma^{1}$ and of $\Sigma^{2}$ that support them. This can be done naturally by extending the previous two definitions.

\begin{definition}
    Given a top-dimensional cell $D$ of $\Sigma^{1}$, we denote
    \begin{equation*}
        \overline{C}(\tau;\Delta_{i}s_{1},\Delta_{j}s_{2})_{D}=\overline{C}(\tau;\Delta_{i}s_{1},\Delta_{j}s_{2})\llcorner \Exp(D\times[0,\varepsilon_{1}])
    \end{equation*}
    which has controlled boundary mass as previously discussed. We let
    \begin{equation*}
        C(\tau;\Delta_{i}s_{1},\Delta_{j}s_{2})_{D}=S_{1}\circ S_{2}(\overline{C}(\tau;\Delta_{i}s_{1},\Delta_{j}s_{2})_{D})
    \end{equation*}
    which is supported in $D$. Additionally,
    \begin{equation*}
        \sum_{D\in\mathcal{F}^{1}}C(\Delta_{i}s_{1},\Delta_{j}s_{2})_{D}=C(\tau;\Delta_{i}s_{1},\Delta_{j}s_{2}).
    \end{equation*}
\end{definition}

\begin{definition}
    Given top dimensional cells $D$ of $\Sigma^{1}$ and $E$ of $\Sigma^{2}$, we define
    \begin{equation*}
        C(\tau;\Delta_{i}s_{1},\Delta_{j}s_{2})_{DE}=C(\tau;\Delta_{i}s_{1},\Delta_{j}s_{2})_{D}\llcorner\tilde{E}
    \end{equation*}
    which is supported in $D\cap E$. It follows
    \begin{equation*}
        \sum_{E\text{ top cell of } \Sigma^{2}}C(\tau;\Delta_{i}s_{1},\Delta_{j}s_{2})_{DE}=C(\tau;\Delta_{i}s_{1},\Delta_{j}s_{2})_{D}
    \end{equation*}
    and hence
    \begin{equation*}
        \sum_{D\text{ top cell of }\Sigma^{1}}\sum_{E\text{ top cell of } \Sigma^{2}}C(\tau;\Delta_{i}s_{1},\Delta_{j}s_{2})_{DE}=C(\tau;\Delta_{i}s_{1},\Delta_{j}s_{2}).
    \end{equation*}
\end{definition}

The contraction of the chain $C(\tau;\Delta_{i}s_{1},\Delta_{j}s_{2})_{DE}$ will come from composing two linear rescalings: one homotecy which affects the whole $3$-cell $D$ centered at $q_{D}$ and another one which only affects $E$ and is centered at $q_{E}$. Therefore, to interpolate we will need the following rescaling factors for each $2\leq i,j\leq q$, $D$ top cell of $\Sigma^{1}$ and $E$ top cell of $\Sigma^{2}$: $\tilde{\mu}^{D}_{i}$, $\tilde{\mu}^{E}_{j}$ and $\tilde{\mu}^{DE}_{ij}$. For short, we denote $\tilde{\mu}$ the collection of such coefficients, and the $1$-cycle they induce is
\begin{multline}\label{Interp formula n=4}
    C(\tau;\underline{s},\tilde{\mu})=C(\tau;s_{1}^{1},s_{2}^{1})+\sum_{j=2}^{q}\sum_{E\in\mathcal{F}^{2}}\tilde{\mu}^{E}_{j}\cdot C(\tau;s_{1}^{1},\Delta_{j}s_{2})_{E}\\
    +\sum_{i=2}^{q}\sum_{D\in\mathcal{F}^{1}}\tilde{\mu}^{D}_{i}\cdot\bigg[C(\tau;\Delta_{i}s_{1},s_{2}^{1})_{D}+\sum_{j=2}^{n-2}\sum_{E\in\mathcal{F}^{2}}\tilde{\mu}^{DE}_{ij}\cdot C(\tau;\Delta_{i} s_{1},\Delta_{j}s_{2})_{DE} \bigg]
\end{multline}
In the previous, if $\support(\eta)\subseteq D$ then $\tilde{\mu}^{D}_{i}\cdot\eta$ is the rescaling of $\eta$ by a factor of $\tilde{\mu}^{D}_{i}$ centered at $q_{D}$, and similarly for $\tilde{\mu}^{E}_{j}\cdot\eta$ and $\tilde{\mu}^{DE}_{ij}\cdot\eta$ but in the later two cases the homotecy is centered at $q_{E}$ and $\eta$ must be supported in $E$. Recall that $\mathcal{F}^{j}$ denotes the set of top dimensional cells of $\Sigma^{j}$ for $j=1,2$.

\begin{example}
If we set the $\tilde{\mu}$ coefficients to be $1$ if and only if 
 $i\leq i_{0}$ and $j\leq j_{0}$ and $0$ otherwise, then $C(\tau;\underline{s},\tilde{\mu})=C(\tau;s_{1}^{i_{0}},s_{2}^{j_{0}})$.
\end{example}

Observe that as $S_{1}\circ S_{2}$ is $(1+\beta)$-Lipschitz (here $1+\beta=(1+\beta_{1})(1+\beta_{2})$) and the $s_{1}^{i}$ and the $s_{2}^{j}$ are in decreasing order, $\mass(C(\tau;\underline{s},\tilde{\mu}))\leq (1+\beta)\mass(\tau)$ but the mass of $\partial C(\tau;\underline{s},\tilde{\mu})$ can be very large. For example, if we set $\tilde{\mu}^{E}_{j},\tilde{\mu}^{DE}_{ij}=0$ for every $2\leq i,j\leq q$, and we take $\tilde{\mu}^{D}_{i}$ strictly decreasing in $i$ and independent of $D$, as a priori there is no cancellation we could get $\mass(\partial C(\tau;\underline{s},\tilde{\mu}))$ to be as large as
\begin{equation*}
    \sum_{i=1}^{q}\mass(\partial C(\tau;s_{1}^{i},s_{2}^{1})\llcorner\interior(\Sigma^{1}))
\end{equation*}
which might be of the order of $2^{p}\frac{\mass(\tau)}{\varepsilon_{1}}$ when $q=2^{p}$ is the number of vertices in a top cell of $X$. This problem already appears when $\dim(M)=3$, and it is sorted out in \cite{GL22} by adding cones over the boundaries of the chains $C(\tau;\Delta_{i}s_{1})_{D}$ as discussed before. We will extend that to the case $\dim(M)=4$. We need the following definition.

\begin{definition}
    Let $\eta\in\mathcal{I}_{1}(M)_{\Sigma}$. For $l\in\{1,2,3\}$ define
    \begin{equation*}
        \Cone_{l}(\eta)=\sum_{E\in\mathcal{F}^{l}}\Cone_{q_{E}}(\partial\eta\llcorner E^{0}).
    \end{equation*}
    where $E^{0}=\interior(E)$. We also let
    \begin{equation*}
        \Cone(\eta)=\sum_{l=1}^{3}\Cone_{l}(\eta).
    \end{equation*}
    and
    \begin{equation*}
        A(\eta)=\eta+\Cone(\eta).
    \end{equation*}
\end{definition}

\begin{remark}
    Observe that as
    \begin{equation*}
        \partial\eta=\sum_{l=1}^{4}\sum_{E\text{ top face of }\Sigma^{l}}\partial\eta\llcorner E^{0}
    \end{equation*}
    it follows
    \begin{equation*}
        \partial A(\eta)\subseteq \mathcal{A}=\bigcup_{l=1}^{4}\bigcup_{E\text{ top face of }\Sigma^{l}}\{q_{E}\}.
    \end{equation*}
    Hence adding $\Cone(\eta)$ has the effect of transferring $\partial\eta$ to the fixed discrete set $\mathcal{A}$.
\end{remark}

We denote $A(\tau;s_{1},s_{2})=A(C(\tau;s_{1},s_{2}))$ and similarly for the related chains. The interpolation formula which will allow us to have the desired mass bounds is the following.

\begin{definition} We denote
   \begin{multline*}
    A(\tau;\underline{s},\tilde{\mu})=A(\tau;s_{1}^{1},s_{2}^{1})+\sum_{j=2}^{q}\sum_{E\in\mathcal{F}^{2}}\tilde{\mu}^{E}_{j}\cdot A(\tau;s_{1}^{1},\Delta_{j}s_{2})_{E}\\
    +\sum_{i=2}^{q}\sum_{D\in\mathcal{F}^{1}}\tilde{\mu}^{D}_{i}\cdot\bigg[A(\tau;\Delta_{i}s_{1},s_{2}^{1})_{D}+\sum_{j=2}^{q}\sum_{E\in\mathcal{F}^{2}}\tilde{\mu}^{DE}_{ij}\cdot A(\tau;\Delta_{i} s_{1},\Delta_{j}s_{2})_{DE} \bigg]
\end{multline*}
and for $1\leq l\leq 3$ we define $\Cone_{l}(\tau;\underline{s},\tilde{\mu})$ to be the $1$-chain obtained by replacing $A$ by $\Cone_{l}$ in the right hand side of the previous expression.
\end{definition}

However, if we want to have control over $\mass(A(\tau;\underline{s},\tilde{\mu}))$
and $\mass(\partial A(\tau;\underline{s},\tilde{\mu}))$, we need to choose $\tilde{\mu}$ carefully. In order to control $\mass(\partial A(\tau;\underline{s},\tilde{\mu}))$, we will need a bound in the number of coefficients $\tilde{\mu}$ which are not $0$ or $1$, because $\mass(\partial A(\tau;\underline{s},\tilde{\mu}))$ is supported in $\mathcal{A}$
union certain scaled copies of $\mathcal{A}\cap D$ for  the cells $D$ of $\Sigma$ having some of the $\tilde{\mu}$ coefficients associated to them different from $0$ and $1$  (for example, if $\tilde{\mu}^{E}_{j}\in(0,1)$ for some $E\in\mathcal{F}^{2}$ we need to add $\tilde{\mu}^{E}_{j}\cdot E\cap\Sigma^{4}$). That will also help to have an upper bound for the mass of $\Cone_{3}(\tau;\underline{s},\tilde{\mu})$ as it is supported in $\Sigma^{3}$ union certain rescaled copies of $D\cap\Sigma^{3}$ for the cells $D$ having some $\tilde{\mu}$ coefficient outside of $\{0,1\}$ as before (for example, if $\tilde{\mu}^{E}_{j}\in(0,1)$ we need to add $\mu^{E}_{j}\cdot E\cap\Sigma^{3}$). But for $\Cone_{1}(\tau;\underline{s},\tilde{\mu})$ and $\Cone_{2}(\tau;\underline{s},\tilde{\mu})$, we will need to have a mass cancellation as in \cite{GL22}.

To get such cancellation, we need $\tilde{\mu}^{D}_{i}$ to be decreasing in $i$ and $\tilde{\mu}^{E}_{j},\tilde{\mu}^{DE}_{ij}$ to be decreasing in $j$. This can be obtained by starting from an arbitrary $\mu$ which consists of $\mu^{D}_{i},\mu^{E}_{j},\mu^{DE}_{ij}\in[0,1]$ for $2\leq i,j\leq q$, $D\in\mathcal{F}^{1}$ and $E\in\mathcal{F}^{2}$ and defining
\begin{align*}
    \tilde{\mu}^{D}_{i} 
 & =\prod_{i'=2}^{i}\mu^{D}_{i'},\\
    \tilde{\mu}^{E}_{j} & =\prod_{j'=2}^{j}\mu^{E}_{j'},\\
\tilde{\mu}^{DE}_{ij} & =\prod_{j'=2}^{j}\mu^{DE}_{ij'}.
\end{align*}
Given such $\mu$, we denote $A(\tau;\underline{s},\mu)=A(\tau;\underline{s},\tilde{\mu})$ and we replace $\mu$ by $\tilde{\mu}$ similarly in the notation.

Now let us study the mass cancellation for the cones. We start with $\Cone_{1}(\tau;\underline{s},\mu)$. Observe that
\begin{equation*}
    \Cone_{1}(\tau;s,\mu)=\Cone_{1}(\tau;s_{1}^{1},s_{1}^{2})+\sum_{i=2}^{q}\sum_{D\in\mathcal{F}^{1}}\tilde{\mu}^{D}_{i}\cdot \Cone_{1}(\tau;\Delta_{i}s_{1},s_{2}^{1})_{D}
\end{equation*}
because as $ C(\tau;s_{1}^{1},\Delta_{j}s_{2})_{E}$ and $ C(\tau;\Delta_{i}s_{1},\Delta_{j}s_{2})_{DE}$ are supported in $\Sigma^{2}$ which is disjoint from $\interior(\Sigma^{1})$, $\Cone_{1}$ over them vanishes. Defining $\tilde{\mu}^{D}_{1}=1$ for every $D$ and denoting $\Delta_{1}s_{1}=s_{1}^{1}$, we can rewrite the previous as
\begin{align*}
    \Cone_{1}(\tau;\underline{s},\mu) & =\sum_{D\in\mathcal{F}^{1}}\bigg[\Cone_{q_{D}}(\partial C(\tau;s_{1}^{1},s_{2}^{1})\llcorner D^{0})+\sum_{i=2}^{q}\tilde{\mu}^{D}_{i}\cdot\Cone_{q_{D}}(\partial C(\tau;\Delta_{i}s_{1},s_{2}^{1})_{D}\llcorner D^{0})\bigg]\\
    & =\sum_{D\in\mathcal{F}^{1}}\sum_{i=1}^{q}\tilde{\mu}^{D}_{i}\cdot\Cone_{q_{D}}(\partial C(\tau;\Delta_{i}s_{1},s_{2}^{1})\llcorner D^{0})
\end{align*}
because
\begin{align*}
    \partial C(\tau;\Delta_{i}s_{1},s_{2}^{1})\llcorner D^{0} & =\sum_{D'\text{ top cell of }\Sigma^{1}}\partial C(\tau;\Delta_{i}s_{1},s_{2}^{1})_{D'}\llcorner D^{0}\\
    & =\partial C(\tau;\Delta_{i}s_{1},s_{2}^{1})_{D}\llcorner D^{0}
\end{align*}
due to the fact that $D'\cap D^{0}=\emptyset$ if $D'\neq D$. Hence setting $\tilde{\mu}^{D}_{q+1}=0$ for every $D$ and $q$,
\begin{align*}
    \Cone_{1}(\tau;\underline{s},\mu) & =\sum_{D\in\mathcal{F}^{1}}\sum_{i=1}^{q}\tilde{\mu}^{D}_{i}\cdot\Cone_{q_{D}}(\partial C(\tau;s_{1}^{i},s_{2}^{1})\llcorner D^{0})-\tilde{\mu}^{D}_{i+1}\cdot\Cone_{q_{D}}(\partial C(\tau;s_{1}^{i},s_{2}^{1})\llcorner D^{0})
\end{align*}
therefore as each $D\in\mathcal{F}_{1}$ has width $\rho_{1}$,
\begin{align}\label{Cone 1 mass bound}
    \mass(\Cone_{1}(\tau;\underline{s},\mu))\leq\sum_{D\in\mathcal{F}^{1}}\sum_{i=1}^{q}(\tilde{\mu}^{D}_{i}-\tilde{\mu}^{D}_{i+1})\rho_{1}\mass(\partial C(\tau;s_{1}^{i},s_{2}^{1})\llcorner D^{0}).
\end{align}

Regarding $\Cone_{2}(\tau;\underline{s},\mu)$, for each value of $\Delta_{i}$ we will have the same kind of cancellation when adding over $j$. First of all, by setting $\tilde{\mu}^{E}_{1}=\tilde{\mu}^{DE}_{i1}=1$ we can rewrite
\begin{align*}
    \Cone_{2}(\tau;\underline{s},\mu) = & \sum_{E\in\mathcal{F}^{2}}\sum_{j=1}^{q}\bigg[\tilde{\mu}^{E}_{j}\cdot\Cone_{q_{E}}(\partial C(\tau;\Delta_{1}s_{1},\Delta_{j}s_{2})\llcorner E^{0})\\
    & +\sum_{D\in\mathcal{F}_{1}}\sum_{i=2}^{q}\tilde{\mu}^{D}_{i}\cdot\big(\tilde{\mu}_{ij}^{DE}\cdot\Cone_{q_{E}}(\partial C(\tau;\Delta_{i}s_{1},\Delta_{j}s_{2})_{D}\llcorner E^{0}) \big)\bigg].
\end{align*}

Using the same argument as in $\Cone_{1}$, the first line of the RHS equals
\begin{multline*}
    \sum_{E\in\mathcal{F}^{2}}\sum_{j=1}^{q}\tilde{\mu}^{E}_{j}\cdot\Cone_{q_{E}}(\partial C(\tau;\Delta_{1}s_{1},s_{2}^{j})\llcorner E^{0})-\tilde{\mu}^{E}_{j+1}\cdot\Cone_{q_{E}}(\partial C(\tau;\Delta_{1}s_{1},s_{2}^{j})\llcorner E^{0})
\end{multline*}
whose mass is bounded by
\begin{equation*}
    \sum_{E\in\mathcal{F}^{2}}\sum_{j=1}^{q}(\tilde{\mu}^{E}_{j}-\tilde{\mu}^{E}_{j+1})\rho_{2}\mass(\partial C(\tau;\Delta_{1}s_{1},s_{2}^{j})\llcorner E^{0}).
\end{equation*}
To bound the second line of RHS, the previous calculation can be done for each $D\in\mathcal{F}^{1}$ and $2\leq i\leq q$ (replacing $s_{1}^{1}=\Delta_{1}s_{1}$ by $\Delta_{i}s_{1}$), which yields
\begin{align}\label{Cone 2 mass bound}
    \mass(\Cone_{2}(\tau;\underline{s},\mu)) \leq & \sum_{E\in\mathcal{F}^{2}}\sum_{j=1}^{q}(\tilde{\mu}^{E}_{j}-\tilde{\mu}^{E}_{j+1})\rho_{2}\mass(\partial C(\tau;\Delta_{1}s_{1},s_{2}^{j})\llcorner E^{0})\\
    & +\sum_{D\in\mathcal{F}^{1}}\sum_{i=2}^{q}\sum_{E\in\mathcal{F}^{2}}\sum_{j=1}^{q}(\tilde{\mu}^{DE}_{ij}-\tilde{\mu}^{DE}_{i,j+1})\rho_{2}\mass(\partial C(\tau;\Delta_{i}s_{1},s_{2}^{j})_{D}\llcorner E^{0})
\end{align}
as $\tilde{\mu}^{D}_{i}\leq 1$.

Taking the previous into account, the strategy will be to extend $F'(x)$ inductively skeleton by skeleton in a way such that the following holds.

\textbf{Inductive property for the $l$-skeleton of $X$.} For every $p'$-dimensional cell $C$ of $X$ with $l\leq p'\leq p$ and every point $y\in C$ the following holds. Consider enumerations $x_{1}^{1},...,x_{1}^{2^{p'}}$ and $x_{2}^{1},...,x_{2}^{2^{p'}}$ of the vertices of $C$ such that
\begin{equation*}
    s_{1}(x_{1}^{1})\geq s_{1}(x_{1}^{2})\geq...\geq s_{1}(x_{1}^{2^{p'}})
\end{equation*}
and 
\begin{equation*}
    s_{2}(x_{2}^{1})\geq s_{2}(x_{2}^{2})\geq...\geq s_{2}(x_{2}^{2^{p'}}).
\end{equation*}
Then for each $1\leq i,j\leq 2^{p'}$, $D\in\mathcal{F}^{1}$ and $E\in\mathcal{F}^{2}$, there exist continuous functions $\mu^{D}_{i},\mu^{E}_{j},\mu^{DE}_{ij}:C\cap X_{l}\to[0,1]$ and a continuous family $e^{y}:C\cap X_{k}\to\mathcal{Z}_{1}(M)$ such that
\begin{equation}
    F'(x)=e^{y}(x)+A(F(y);\underline{s},\mu(x))
\end{equation}
where $\underline{s}$ is given by the $s_{i}^{j}=s_{i}(x_{i}^{j})$ (and therefore it depends only on the cell $C$ but not on $x$) and $\mu(x)$ by the collection of $\mu^{D}_{i}(x)$, $\mu^{E}_{j}(x)$ and $\mu^{DE}_{ij}(x)$. The following properties hold for $e^{y}(x)$ and $\mu(x)$.
\begin{enumerate}[label=(A\arabic*)]
    \item Denote
    \begin{align*}
        A(x) & =\#\{(D,i):\mu^{D}_{i}(x)\in(0,1)\}\\
        B(x) & =\#\{(E,j):\mu^{E}_{j}(x)\in(0,1)\}\\
        C(x) & =\#\{(D,E,i,j):\mu^{DE}_{ij}(x)\in(0,1)\}.
    \end{align*}
    Then $A(x)+B(x)+C(x)\leq\dim(E(x))$, where $E(x)$ is the unique cell of $X$ which contains $x$ in its interior. Moreover, if
    \begin{equation*}
        i(x)=\max\{i:x_{1}^{i}\in E(x)\}
    \end{equation*}
    and
    \begin{equation*}
        j(x)=\max\{j:x_{2}^{j}\in E(x)\}
    \end{equation*}
    then
    \begin{enumerate}
        \item $\mu^{D}_{i}(x)=0$ if $i>i(x)$ and $\mu^{D}_{i}(x)=1$ if $i\leq i(x)$ and $x_{i}\notin E(x)$.
        \item $\mu^{E}_{j}(x)=0$ (resp. $\mu^{DE}_{ij}(x)=0$) if $j>j(x)$ and $\mu^{E}_{j}(x)=1$ (resp. $\mu^{DE}_{ij}(x)=1$) if $j\leq j(x)$ and $x_{j}\notin E(x)$.
    \end{enumerate}
    \item The following are true
    \begin{equation*}
        \sum_{D\in\mathcal{F}^{1}}\sum_{i=1}^{q}(\tilde{\mu}^{D}_{i}(x)-\tilde{\mu}^{D}_{i+1}(x))\mass(\partial C(\tau;s_{1}^{i},s_{2}^{1})\llcorner D^{0})\leq\frac{\mass(F(x))}{\varepsilon_{1}}
    \end{equation*}
    and
    \begin{multline*}
        \sum_{E\in\mathcal{F}^{2}}\sum_{j=1}^{q}(\tilde{\mu}^{E}_{j}(x)-\tilde{\mu}^{E}_{j+1}(x))\mass(\partial C(\tau;s_{1}^{1},s_{2}^{j})\llcorner E^{0})\\
        +\sum_{D\in\mathcal{F}^{1}}\sum_{i=2}^{q}\sum_{E\in\mathcal{F}^{2}}\sum_{j=1}^{q}(\tilde{\mu}^{DE}_{ij}(x)-\tilde{\mu}^{DE}_{i,j+1}(x))\mass(\partial C(\tau;\Delta_{i}s_{1},s_{2}^{j})_{D}\llcorner E^{0})\leq 2\frac{\mass(F(x))}{\varepsilon_{2}}.
    \end{multline*}
    \item There exists a constant $c(l)$ depending only on $l$ such that
    \begin{equation*}
        \mass(e^{y}(x))\leq c(l)\varepsilon
    \end{equation*}
    for every $x\in C\cap X^{l}$.
\end{enumerate}

It is shown in the next section that (A1) yields a bound on $\mass(\Cone_{3}(F(y);\underline{s},\mu(x))$ and also $\mass(\partial A(F(y);\underline{s},\mu(x)))$. On the other hand, by (\ref{Cone 1 mass bound}) and (\ref{Cone 2 mass bound}) it follows that condition (A2) provides an upper bound for $\mass(\Cone_{1}(F(y);\underline{s},\mu(x)))$ and $\mass(\Cone_{2}(F(y);\underline{s},\mu(x)))$. Finally, (A3) provides a bound for the length of $e^{y}(x)$. Putting everything together, we get the desired bounds for $\mass(F'(x))$ and $\mass(\partial F'(x))$.

\section{Parametric Coarea Inequality for $1$-cycles in $n$-manifolds}\label{Section n arbitrary}

In this section, we are going to extend the constructions which were done previously for families of $1$-cycles in $3$ and $4$-manifolds to $n$-dimensional ambient manifolds $M$, $n\geq 5$, with the aim of generalizing the Parametric Coarea Inequality for $1$-cycles. We start by motivating some of those constructions, and then we proceed to make them rigorous in the subsequent subsections.

Recall that for $n=4$, we introduced two coarea inequality cuts: $s_{1}(x)\in(0,\varepsilon_{1})$ and $s_{2}(x)\in (0,\varepsilon_{2})$. These cuts are made with respect to the distance function to two PL submanifolds of $M$: $A_{1}=\Sigma$ and $A_{2}=\Exp(\Sigma^{2}\times[0,\varepsilon_{1}])$, where $\Sigma^{2}$ is the codimension-$1$ skeleton of $\Sigma$ under a triangulation of width $\rho_{1}<<\varepsilon_{1}$. Notice that we can think of $A_{2}$ as $r_{1}^{-1}(\Sigma^{2})$, where $r_{1}=S_{1}|_{N_{\varepsilon_{1}}\Sigma}$ and $S_{1}:M\to M$ is the previously defined $(1+\beta)$-Lipschitz map (for $\beta$ small) which maps $\Exp(x,t)\mapsto x$ on $N_{\varepsilon_{1}}\Sigma$. We also considered a $(1+\beta)$-Lipschitz map $S_{2}:M\to M$ such that $S_{2}(N_{\varepsilon_{2}}A_{2})\subseteq A_{2}$. The chain resulting from removing the part of $\tau\in\mathcal{I}_{1}(M)_{\Sigma}$ contained in $N_{s_{l}}A_{l}$ for $l=1,2$ was denoted by $\overline{C}(\tau;s_{1},s_{2})$ and its corresponding deformation via applying $R_{2}=S_{1}\circ S_{2}$ was called $C(\tau;s_{1},s_{2})\in\mathcal{I}_{1}(M)_{\Sigma}$. In order to be able to interpolate in high-dimensional cells, we added cones over $\partial C(F(x);s_{1}(x),s_{2}(x))$ obtaining the chains $A(F(x);s_{1}(x),s_{2}(x))$ which are the values of $F'$ at the $0$-skeleton. An important property of our construction is that we considered cones of radius $\rho_{1}$ for points lying in the interior of the $3$-cells of $\Sigma^{1}=\Sigma$ and cones of radius $\rho_{2}$ for points lying in the interior of $2$-cells of $\Sigma^{2}$. This was because the total number of points of the first kind is bounded above by $\frac{\mass(F(x))}{\varepsilon_{1}}$ but for the second type the upper bound is $\frac{\mass(F(x))}{\varepsilon_{2}}$. That justifies the different scales $\rho_{1}$ and $\rho_{2}$ considered to triangulate $\Sigma^{1}$ and $\Sigma^{2}$, and the requirement $\rho_{l}<<\varepsilon_{l}$ with the purpose of ensuring that the total mass of the cones we are adding goes to $0$ as $p\to\infty$. We denote $T_{1}$ the triangulation of $\Sigma^{1}$ with width $\rho_{1}$ and $T_{2}$ that of $\Sigma^{2}$ with width $\rho_{2}$. Roughly speaking, the strategy to extend $F'$ to the higher dimensional skeleta was to decompose the interpolation as the sum of two parts: one involving $1$-chains of the form $C(\tau;\Delta s_{1},s_{2})_{D}$ supported in the $3$-dimensional $\Sigma^{1}=\Sigma$, whose boundary has controlled mass precisely because we introduced the cut $s_{2}$; and another one involving $1$-chains of the form $C(\tau;s_{1},\Delta s_{2})_{E}$ which are supported in the $2$-dimensional $\Sigma^{2}$. For the latter, we used the construction of Guth and Liokumovich which consists in adding segments (or cones) supported in the boundary of each $2$-cell $E$ of $\Sigma^{2}$. 

Now suppose that $n=5$. We can define $\Sigma^{1}$, $\Sigma^{2}$, $A_{1}$, $A_{2}$, $S_{1}$ and $S_{2}$ in the same way as for $n=4$ (in particular, $A_{2}=r_{1}^{-1}(\Sigma^{2})$). Nonetheless, now $C(\tau;s_{1},\Delta s_{2})$ is supported in the $3$-dimensional $\Sigma^{2}$ so we cannot add segments on the boundary of the corresponding cells and use mass cancellation as when $n=4$. Instead, we will define a new PL submanifold $A_{3}$ of $M$ and introduce an extra coarea inequality cut at distance $s_{3}\in (0,\varepsilon_{3})$ from $A_{3}$ so that the corresponding chains $C(\tau;s_{1},\Delta s_{2},s_{3})_{E}$ have controlled boundary mass for any top cell $E$ of $(\Sigma^{2},T_{2})$; in the same way as we introduced $s_{2}$ to get that property for $C(\tau;\Delta s_{1},s_{2})_{D}$ for any top cell $D$ of $(\Sigma^{1},T_{1})$. Thus, for $5$-dimensional ambient manifolds, the interpolation will be given by sums of rescalings of three types of chains: 
\begin{itemize}
    \item Those of the form $C(\tau;\Delta s_{1},s_{2},s_{3})_{D}$ which are supported in top dimensional cells $D$ of the $4$-dimensional $\Sigma^{1}$. Their boundary mass is controlled because we introduced the cut $s_{2}$.
    \item Those of the form $C(\tau;s_{1},\Delta s_{2},s_{3})_{E}$ which are supported in top dimensional cells $E$ of the $3$-dimensional $\Sigma^{2}$. Their boundary mass is controlled because we introduced the cut $s_{3}$.
    \item Those of the form $C(\tau;s_{1},s_{2},\Delta s_{3})_{F}$ which are supported in top dimensional cells $F$ of the $2$-dimensional $\Sigma^{3}$. We add segments supported in $\partial F$ as Guth and Liokumovich did in order to have control in the mass of the boundaries.
\end{itemize}
Following the analogy with the case $n=4$, the chains $\overline{C}(\tau;s_{1},s_{2},s_{3})$ will result from removing the portions of $\tau$ supported in $N_{s_{l}}A_{l}$ for $l=1,2,3$. Then we will apply the map $R_{3}=S_{1}\circ S_{2}\circ S_{3}$ in order to ``fill in'' the gap we just created by the coarea cuts and obtain a new chain $C(\tau;s_{1},s_{2},s_{3})\in \mathcal{I}_{1}(M)_{\Sigma}$. This will guarantee that $C(\tau;s_{1},\Delta s_{2},s_{3})$ is supported in $\Sigma^{2}$ but also that it can be decomposed as
\begin{equation}\label{Eq decomposition}
    C(\tau;s_{1},\Delta s_{2},s_{3})=\sum_{E\text{ top face of }\Sigma^{2}}C(\tau;s_{1},\Delta s_{2},s_{3})_{E}
\end{equation}
where each $C(\tau;s_{1},\Delta s_{2},s_{3})_{E}$ is supported in the corresponding cell $E$ and has controlled boundary mass. $A_{3}$ is defined so that the previous decomposition is possible, analogously to $A_{2}$ when $n=4$ (see Definition \ref{Definition C(s1s2)_D}). To be precise, we know that $\overline{C}(\tau;s_{1},\Delta s_{2})$ is supported in $N_{2}=N_{\varepsilon_{2}}A_{2}$ and also that $R_{2}(N_{2})=S_{1}\circ S_{2}(N_{2})\subseteq\Sigma^{2}$. Denote $r_{2}=R_{2}|_{N_{2}}$. We can decompose $N_{2}$ as the non-disjoint union of the parts which corresponds to the different cells of $\Sigma^{2}$:
\begin{equation*}
    N_{2}=\bigcup_{E\text{ top cell of }\Sigma^{2}}r_{2}^{-1}(E).
\end{equation*}
Let $\Sigma^{3}$ be the codimension-$1$ skeleton of $(\Sigma^{2},T_{2})$. Introducing
\begin{equation*}
    A_{3}=r_{2}^{-1}(\Sigma^{3})
\end{equation*}
allows us to separate each of the previous pieces by expressing $N_{2}\setminus A_{3}$ as a disjoint union
\begin{equation*}
    N_{2}\setminus A_{3}=\bigcup_{E\text{ top cell of }\Sigma^{2}}r_{2}^{-1}(E^{0}).
\end{equation*}
Moreover, if we perform a third coarea inequality cut $s_{3}$ away from $A_{3}$, we can decompose 
\begin{equation*}
    \overline{C}(\tau;s_{1},\Delta s_{2},s_{3})=\sum_{E\text{ top cell of }(\Sigma^{2},T_{2})}\overline{C}(\tau;s_{1},\Delta s_{2},s_{3})\llcorner r_{2}^{-1}(E)
\end{equation*}
where the chains on the RHS will be the $\overline{C}(\tau;s_{1},\Delta s_{2},s_{3})_{E}$ and will have controlled boundary mass. Notice that the previous does not hold without performing the third cut, as $\overline{C}(\tau;s_{1},\Delta s_{2})$ may intersect $A_{3}$ in a $0$-cycle of arbitrarily large mass; which is the exact reason why $C(\tau;s_{1},\Delta s_{2})$ may intersect $\partial E$ in a $0$-cycle of very large mass for some $3$-cells $E$ of $\Sigma^{2}$.

\begin{remark}
    For $l=1,2,3$, can think of $N_{l}$ as an ``enlarged version'' of $\Sigma^{l}$, where the correspondence between them is given by the $(1+\beta)$-Lipschitz map $r_{l}:N_{l}\to\Sigma^{l}$. Under this correspondence, $\Sigma^{l+1}$ corresponds to $A_{l+1}$ and chains of the form $\overline{C}(\tau;...,\Delta s_{l},...)$ correspond to chains of the form $C(\tau;...,\Delta s_{l},...)$. The reason to consider these enlargements is precisely to have enough room to do the coarea inequality cuts to $\tau$ which induce a decomposition like (\ref{Eq decomposition}). 
\end{remark}

The previous strategy can be generalized to arbitrary $n$. We can define $\Sigma^{1}$,...,$\Sigma^{n-2}$, PL submanifolds $A_{1}$,...,$A_{n-2}$ of $M$, open neighborhoods $N_{1}$,...,$N_{n-2}$ of the previous and $(1+\beta)$-Lipschitz maps $S_{l}:M\to M$ satisfying $S_{l}(N_{l})=A_{l}$. We introduce $(n-2)$ coarea inequality cuts $s_{1},...,s_{n-2}$, $s_{l}\in(0,\varepsilon_{l})$, each of them with respect to the distance function to the corresponding $A_{l}$. The constructions are done in a way such that we can always decompose $C(\tau;s_{1},...,s_{l-1},\Delta s_{l},s_{l+1},...,s_{n-2})$ as sum of pieces supported in each top cell $D$ of $\Sigma^{l}$ as in (\ref{Eq decomposition}) with controlled boundary mass (provided $l\leq n-3$). Thus the interpolation to extend to the higher-dimensional skeleta will be reduced to a sum of rescalings of $n-2$ types of chains:
\begin{itemize}
    \item The $l$-th type for $1\leq l\leq n-3$ is given by $C(\tau;s_{1},...,s_{l-1},\Delta s_{l},s_{l+1},...,s_{n-2})_{D}$ for some top cell $D$ of the $(n-l)$-dimensional $\Sigma^{l}$. It is supported in the corresponding cell $D$ of $\Sigma^{l}$. The boundary mass is controlled because the cut $s_{l+1}$ was made.
    \item  The last type is $C(\tau;s_{1},...,s_{n-3},\Delta s_{n-2})_{D}$ which is supported in a top cell $D$ of the $2$-dimensional $\Sigma^{n-2}$. Here we add the segments (or cones) supported in $\partial D$ as in the previous cases, based on \cite{GL22}.
\end{itemize}
Each $\Sigma^{l}$ is provided with a triangulation $T_{l}$ of a certain width $\rho_{l}$, where 
\begin{equation*}
\varepsilon_{1}>>\rho_{1}>>\varepsilon_{2}>>\rho_{2}>>...>>\varepsilon_{n-2}>>\rho_{n-2}.
\end{equation*}
The previous makes the necessary constructions possible and also guarantees that the mass of the cones added to go from chains of the form $C(\tau;...)$ to those of the form $A(\tau;...)$ (step necessary for the interpolation in high-dimensional cells, same as in the cases $n=3$ and $n=4$) is accurately bounded.

\subsection{Definition of $C(\tau;s_{1},...,s_{k})$}\label{Section cuts}

To simplify the exposition, in Sections \ref{Section cuts} to \ref{Section Proof of Coarea Ineq} we will assume that $M$ a rectangular domain in $\mathbb{R}^{n}$ provided with a rectangular structure $T_{0}$ of width at least $1$ (see Section \ref{Section rectangular complexes} for the definitions). In Section \ref{Extension to triangulable domains}, we show how to extend the result to almost $1$-Lipschitz triangulable piecewise smooth Riemannian manifolds with boundary. Fix a small number $\alpha>0$, denote $\alpha'=\frac{2}{(n-1)(n-2)}\alpha$ and define for each $p\in\mathbb{N}$ a collection of positive real numbers
\begin{equation*}
    1=\rho_{0}>>\varepsilon_{1}>>\rho_{1},\rho'_{1}>>\varepsilon_{2}>>\rho_{2},\rho'_{2}>>...>>\varepsilon_{n-2}>>\rho_{n-2},\rho'_{n-2}
\end{equation*}
with $\varepsilon_{l}=\varepsilon_{l}(p)$, $\rho_{l}=\rho_{l}(p)$ and $\rho'_{l}=\rho'_{l}(p)$ as follows. Set $n_{p}=\lfloor p^{\alpha'}\rfloor$,
\begin{align*}
    \varepsilon_{1} & = p^{-\frac{1}{n}}\\
    \rho_{1} & = \frac{1}{\lfloor p^{\frac{1}{n-1}}\rfloor} \sim p^{-\frac{1}{n-1}}\\
    \rho'_{1} & =(1-2\varepsilon_{1})\rho_{1}
\end{align*}
and
\begin{align*}
    \rho_{l} & =\frac{1}{n_{p}}\rho_{l-1}\\
    \varepsilon_{l} & = \frac{1}{\sqrt{n_{p}}}\rho'_{l-1} \\
    \rho'_{l} & = \frac{1}{n_{p}}(1-\frac{2}{\sqrt{n_{p}}})\rho'_{l-1}
\end{align*}
if $2\leq l\leq n-2$. Notice that for every $l$, $\rho'_{l}\leq \rho_{l}$, $\lim_{p\to\infty}\frac{\rho'_{l}}{\rho_{l}}=1$, $\rho_{l}<<\varepsilon_{l}$, $\varepsilon_{l}<<\rho_{l-1}$ and $\rho'_{l}=\frac{1}{n_{p}}(\rho'_{l-1}-2\varepsilon_{l})$ (the latter width comes from applying Lemma \ref{Lemma map S} for $\rho=\rho'_{l-1}$, $\varepsilon=\varepsilon_{l}$ and $q=n_{p}$).

\begin{remark}
    In case $\rho_{0}\neq 1$, the same constructions can be done provided $p\geq p_{0}=(\frac{3}{\rho_{0}})^{n}$ which guarantees $\rho_{0}>3\varepsilon_{1}$.
\end{remark}

For each $1\leq l\leq n-2$, we define a rectangular structure $T_{l}$ on $M$ as follows. $T_{1}$ is obtained from $T_{0}$ by performing a $\lfloor p^{\frac{1}{n-1}}\rfloor$-refinement. Then we inductively obtain $T_{l+1}$ from $T_{l}$ by performing an $n_{p}$-refinement. Observe that $T_{l}$ has width at least $\rho_{l}$ for every $l$. We define a finite sequence $\{\Sigma^{l}\}_{1\leq l\leq n}$ such that each $\Sigma^{l}$ is a rectangular subpolyhedron of $(\Sigma,T_{l-1})$ and $\dim(\Sigma^{l})=n-l$. We set $\Sigma^{1}=\Sigma$ and we recursively define $\Sigma^{l+1}$ as the $n-(l+1)$-skeleton  of $(\Sigma^{l},T_{l})$. In other words, at each step we consider an $n_{p}$-refinement of $\Sigma^{l}$ (which by definition had the structure $T_{l-1}$) and then we take the codimension-$1$ skeleton to obtain $\Sigma^{l+1}$. The previous is done for $l\leq n-2$. For $l=n-1$ we define $\Sigma^{n-1}$ as the $1$-skeleton of $(\Sigma^{n-2},T_{n-2})$ and for $l=n$ we let $(\Sigma^{n},T^{n})$ be the $0$-skeleton of $(\Sigma^{n-2},T_{n-2})$. Unless otherwise is specified, the cubical structure assigned to $\Sigma^{l}$ will always be $T_{l}$.

\begin{remark}
    Observe that the previous refinements can be done provided $n_{p}\geq 9$, which is true if $p\geq p_{0}(n,\alpha)$.
\end{remark}

\begin{definition}
    We denote by $|\cdot|_{\infty}$ the infinity norm in $\mathbb{R}^{n}$, given by
    \begin{equation*}
        |(x_{1},...,x_{n})|_{\infty}=\max\{|x_{i}|:1\leq i\leq n\}.
    \end{equation*}
    We also denote by $|\cdot|$ the Euclidean norm
    \begin{equation*}
        |(x_{1},...,x_{n})|=\sqrt{x_{1}^{2}+...+x_{n}^{2}}.
    \end{equation*}
\end{definition}

\begin{definition}
    Given $\varepsilon>0$ and a PL submanifold $A\subseteq M$, denote 
    \begin{align*}
        N_{\varepsilon}A & =\{x\in M:\dist(x,A)<\varepsilon\}\\
        \overline{N_{\varepsilon}A} & =\{x\in M:\dist(x,A)\leq\varepsilon\}\\
        D_{\varepsilon}A & =\{x\in M:\dist(x,A)=\varepsilon\}.
    \end{align*}
    Here distances are measured with respect to the Euclidean metric. If we use the $|\cdot|_{\infty}$ metric instead, we denote the previous sets by $N_{\varepsilon}^{\infty}A$, $\overline{N_{\varepsilon}^{\infty}A}$ and $D_{\varepsilon}^{\infty}A$.
\end{definition}

\begin{remark}
    Notice that $N_{\varepsilon}A\subseteq N_{\varepsilon}^{\infty}A$ for every $\varepsilon>0$ and every $A$.
\end{remark}

\begin{definition}
    Let $\mathcal{I}_{1}(M)_{\Sigma}=\{\tau\in\mathcal{I}_{1}(M):\support(\partial\tau)\subseteq\Sigma\}$ be the space of absolute flat $1$-chains in $M$ with $\mathbb{Z}_{2}$ coefficients whose boundary is supported in $\Sigma=\partial M$.
\end{definition}

For each $1\leq l\leq n-2$, we are going to define a rectangular structure $T'_{l}$ on $M$, a rectangular subcomplex $A_{l}$ of $T'_{l-1}$, an open neighborhood $N_{l}$ of $A_{l}$ in the $|\cdot|_{\infty}$-metric and a map $S_{l}:M\to M$ such that
\begin{enumerate}[label=(B\arabic*)]
    \item \label{ite : 2.1} $T'_{0}=T_{0}$, $T'_{l}$ is a refinement of $T'_{l-1}$ and each $T'_{l}$ has width at least $\rho'_{l}$.
    \item \label{ite : 2.2} $N_{l}=N_{\varepsilon_{l}}^{\infty}A_{l}$ and it is a rectangular subcomplex of $(M,T'_{l})$.
    \item \label{ite : 2.3} Each $S_{l}$ is $\frac{\rho'_{l-1}}{\rho'_{l-1}-2\varepsilon_{l}}$-Lipschitz and $S_{l}(N_{l})=A_{l}$.
    \item \label{ite : 2.4} The map $R_{l}=S_{1}\circ...\circ S_{l}:(M,T'_{l})\to(M,T_{l})$ is rectangular for every $1\leq l\leq n-2$ and $R_{l}(N_{l})\subseteq \Sigma^{l}$.
    \item \label{ite : 2.5} $A_{l+1}\subseteq N_{l}$. Moreover, if we denote $r_{l}=R_{l}|_{N_{l}}$ then $A_{l+1}=r_{l}^{-1}(\Sigma^{l+1})$.
    
\end{enumerate}

We set $A_{1}=\Sigma$, $N_{1}=N_{\varepsilon_{1}}^{\infty}\Sigma$ and $S_{1}=S_{(M,T_{0},\varepsilon_{1})}$. By Lemma \ref{Lemma map S}, there exists a rectangular structure $T_{1}'$ of width at least $(1-2\varepsilon_{1})\rho_{1}$ on $M$ such that $S_{1}:(M,T'_{1})\to (M,T_{1})$ is a rectangular map, $S_{1}(N_{1})\subseteq \Sigma=A_{1}$ and $N_{1}$ is a rectangular subcomplex of $(M,T'_{1})$.

Suppose that we have defined the previous objects for $1\leq l'\leq l$ in a way such that \ref{ite : 2.1} to \ref{ite : 2.5} hold for except for \ref{ite : 2.5} for $l'=l$. We know that $r_{l}:(N_{l},T'_{l})\to(\Sigma^{l},T_{l})$ is a rectangular map. Therefore if we set $A_{l+1}:=r_{l}^{-1}(\Sigma^{l+1})$, as $\Sigma^{l+1}$ is a subcomplex of $(\Sigma^{l},T_{l})$ it follows that $A_{l+1}$ is a subcomplex of $(N_{l},T'_{l})$. Recall that the width of $T'_{l}$ is at least $\rho'_{l}$ which is bigger than $\varepsilon_{l+1}$. Set $N_{l+1}=N_{\varepsilon_{l+1}}^{\infty}A_{l+1}$ and $S_{l+1}=S_{(M,T'_{l},\varepsilon_{l+1})}$ as defined in Lemma \ref{Lemma map S}. Then by Lemma \ref{Lemma map S}, there exists a rectangular structure $T'_{l+1}$ of width at least $\rho'_{l+1}$ such that $S_{l+1}:(M,T'_{l+1})\to(M,\Refine^{n_{p}}T'_{l})$ is rectangular (see Definition \ref{Def q-refinement}). This implies that $R_{l+1}:(M,T'_{l+1})\to(M,T_{l+1})$ is rectangular. In addition, $S_{l+1}(N_{l+1})=A_{l+1}$ hence $r_{l+1}(N_{l+1})=r_{l}(A_{l+1})=\Sigma^{l+1}$ and all properties are verified.

\begin{remark}\label{Rk 1}
    Observe that $S_{k}(N_{l})\subseteq N_{l}$ if $k>l$ and moreover: given an $(n-l)$-dimensional cell $D$ of $\Sigma^{l}$ it holds $S_{k}(r_{l}^{-1}(D))\subseteq r_{l}^{-1}(D)$. This implies that
    \begin{equation*}
        R_{k}(r_{l}^{-1}(D))=R_{l}\circ S_{l+1}\circ...\circ S_{k}(r_{l}^{-1}(D)\subseteq R_{l}(r_{l}^{-1}(D))\subseteq D.
    \end{equation*}
\end{remark}

\begin{definition}
    Let $\tau\in\mathcal{I}_{1}(M)_{\Sigma}$. We say that $s\in(0,\varepsilon_{l})$ is an $l$-admissible cut for $\tau$ if $\tau$ is transverse to $D_{s}A_{l}$ (see Definition \ref{Def tubular neighborhood}) and
    \begin{equation*}
        \mass(\tau\llcorner D_{s}A_{l})\leq\frac{\mass(\tau)}{\varepsilon_{l}}.
    \end{equation*}
    Here distances are measured with respect to the Euclidean norm $|\cdot|$. The existence of such a cut is guaranteed by the coarea inequality. A tuple $(s_{1},...,s_{k})$ is said to be an admissible tuple for $\tau$ if $s_{l}\in(0,\varepsilon_{l})$ is an $l$-admissible cut for $\tau$ for every $1\leq l\leq k$ and in addition $\tau$ is transverse to $D_{s_{l}}A_{l}\cap D_{s_{l'}}A_{l'}$ for every $l\neq l'$, $l,l'\in\{1,...,k\}$ (this implies that $\tau\llcorner [D_{s_{l}}A_{l}\cap D_{s_{l'}}A_{l'}]=0$).
\end{definition}

\begin{definition}
    Let $(s_{1},...,s_{k})$ be an admissible tuple for $\tau$. We define
    \begin{align*}
        \overline{C}(\tau;s_{1},...,s_{k}) & =\tau\llcorner(M\setminus(\bigcup_{l=1}^{k} N_{s_{l}}A_{l}))\\
        C(\tau;s_{1},...,s_{k}) & = R_{k}(\overline{C}(\tau;s_{1},...,s_{k})).
    \end{align*}
\end{definition}

\begin{definition}
    Given $1\leq l\leq k$, we set
    \begin{equation*}
        \partial_{l}\overline{C}(\tau;s_{1},...,s_{k})=(\tau\llcorner D_{s_{l}}A_{l})\llcorner(M\setminus\bigcup_{\substack{1\leq l'\leq k \\ l'\neq l}}\overline{N_{s_{l'}}A_{l'}}).
    \end{equation*}
\end{definition}

\begin{remark}
    By definition of admissible cut, it holds
    \begin{equation*}
        \mass(\partial_{l}\overline{C}(\tau;s_{1},...,s_{k}))\leq \mass(\tau\llcorner D_{s_{l}}A_{l})\leq\frac{\mass(\tau)}{\varepsilon_{l}}.
    \end{equation*}
\end{remark}

\begin{proposition}\label{Prop 1}
    It holds
    \begin{equation*}
        \partial\overline{C}(\tau;s_{1},...,s_{k})=\sum_{l=1}^{k}\partial_{l}\overline{C}(\tau;s_{1},...,s_{k}).
    \end{equation*}
\end{proposition}

\begin{proof}
    We proceed by induction on $k$. If $k=1$, the result is clear. Suppose it is true for $k$. Observe that
    \begin{equation*}
        \overline{C}(\tau;s_{1},...,s_{k+1})=\overline{C}(\tau;s_{1},...,s_{k})\llcorner(M\setminus N_{s_{k+1}}A_{k+1})
    \end{equation*}
    therefore
    \begin{align*}
        \partial\overline{C}(\tau;s_{1},...,s_{k+1}) & =\partial\overline{C}(\tau;s_{1},...,s_{k})\llcorner(M\setminus \overline{N_{s_{k+1}}A_{k+1}})+\overline{C}(\tau;s_{1},...,s_{k})\llcorner D_{s_{k+1}}A_{k+1}\\
        & =\sum_{l=1}^{k}\partial_{l}\overline{C}(\tau;s_{1},...,s_{k})\llcorner(M\setminus \overline{N_{s_{k+1}}A_{k+1}})+(\tau\llcorner D_{s_{k+1}}A_{k+1})\llcorner(M\setminus\bigcup_{l=1}^{k}N_{s_{l}}A_{l})\\
        & =\sum_{l=1}^{k}\partial_{l}\overline{C}(\tau;s_{1},...,s_{k})\llcorner(M\setminus \overline{N_{s_{k+1}}A_{k+1}})+(\tau\llcorner D_{s_{k+1}}A_{k+1})\llcorner(M\setminus\bigcup_{l=1}^{k}\overline{N_{s_{l}}A_{l}})\\
        & =\sum_{l=1}^{k+1}\partial_{l}\overline{C}(\tau;s_{1},...,s_{k+1})
    \end{align*}
    by inductive hypothesis. We have used the fact that $\tau\llcorner[D_{s_{k+1}}A_{k+1}\cap D_{s_{l}}A_{l}]=0$ for every $1\leq l\leq k$.
\end{proof}

\begin{definition}
    We set $\partial_{l}C(\tau;s_{1},...,s_{k})=R_{k}(\partial_{l}\overline{C}(\tau;s_{1},...,s_{k}))$.
\end{definition}

\begin{remark}
    As $\partial_{l}\overline{C}(\tau;s_{1},...,s_{k})$ is supported in $N_{l}$ (because $D_{s_{l}}A_{l}\subseteq N_{l}$), by Remark \ref{Rk 1} it holds
    \begin{equation*}
        \support(\partial_{l}C(\tau;s_{1},...,s_{k}))\subseteq\Sigma^{l}
    \end{equation*}
    and using Proposition \ref{Prop 1} it follows
    \begin{equation*}
        \partial C(\tau;s_{1},...,s_{k})=\sum_{l=1}^{k}\partial_{l}C(\tau;s_{1},...,s_{k}).
    \end{equation*}
\end{remark}

\subsection{Definition of $C(\tau;\Delta_{i_{1}}s_{1},...,\Delta_{i_{l}}s_{l},s_{l+1}^{i_{l+1}},...,s_{k}^{i_{k}})$}

Given a set of admissible tuples $\{(s^{i}_{1},s^{i}_{2},...,s^{i}_{k}):1\leq i\leq q\}$ for $\tau$, we will need to interpolate between $C(\tau;s_{1}^{i},...,s_{k}^{i})$ for the different values of $1\leq i\leq q$. Throughout this section, we will assume that $s_{j}^{1}\geq s_{j}^{2}\geq...\geq s_{j}^{q}$ for every $1\leq j\leq k$. With the purpose of interpolating, we will need to define $C(\tau;\Delta_{i_{1}} s_{1},...,\Delta_{i_{l}} s_{l},s_{l+1}^{i_{l+1}},...,s_{k}^{i_{k}})$ for each $1\leq l\leq k$ and for each tuple $(i_{1},...,i_{k})\in\{1,...,q\}^{k}$.

    \begin{definition}
        We define $\overline{C}(\tau;\Delta_{i_{1}} s_{1},...,\Delta_{i_{l}} s_{l},s_{l+1}^{i_{l+1}},...,s_{k}^{i_{k}})$ inductively. For $l=1$, we set

        \begin{equation*}
            \overline{C}(\tau;\Delta_{1}s_{1},s_{2}^{i_{2}},...,s_{k}^{i_{k}})  =\overline{C}(\tau;s_{1}^{1},s_{2}^{i_{2}},...,s_{k}^{i_{k}})
        \end{equation*}
        for $i_{1}=1$ and
        \begin{equation*}
            \overline{C}(\tau;\Delta_{i_{1}} s_{1},s_{2}^{i_{2}},...,s_{k}^{i_{k}})  =\overline{C}(\tau;s^{i_{1}}_{1},s_{2}^{i_{2}},...,s_{k}^{i_{k}})-\overline{C}(\tau;s_{1}^{i_{1}-1},s_{2}^{i_{2}},...,s_{k}^{i_{k}})
        \end{equation*}
        for $2\leq i_{1}\leq q$. Assume $\overline{C}(\tau;\Delta_{i_{1}} s_{1},...,\Delta_{i_{l}} s_{l},s_{l+1}^{i_{l+1}},...,s_{k}^{i_{k}})$ has been defined. We set
        \begin{equation*}
            \overline{C}(\tau;\Delta_{i_{1}}s_{1},...,\Delta_{i_{l}}s_{l},\Delta_{1}s_{l+1},s_{l+2}^{i_{l+2}},...,s_{k}^{i_{k}})= \overline{C}(\tau;\Delta_{i_{1}}s_{1},...,\Delta_{i_{l}}s_{l},s_{l+1}^{1},s_{l+2}^{i_{l+2}},...,s_{k}^{i_{k}})
        \end{equation*}
        for $i_{l+1}=1$ and
        \begin{align*}
            \overline{C}(\tau;\Delta_{i_{1}} s_{1},...,\Delta_{i_{l+1}} s_{l+1},s_{l+2}^{i_{l+2}},...,s_{k}^{i_{k}})=\overline{C}(\tau;\Delta_{i_{1}} s_{1},...,\Delta_{i_{l}} s_{l},s_{l+1}^{i_{l+1}},...,s_{k}^{i_{k}}) \\
            -\overline{C}(\tau;\Delta_{i_{1}} s_{1},...,\Delta_{i_{l}} s_{l},s_{l+1}^{i_{l+1}-1},...,s_{k}^{i_{k}})
        \end{align*}
        for $2\leq l\leq q$. We can define $C(\tau;\Delta_{i_{1}} s_{1},...,\Delta_{i_{l}} s_{l},s_{l+1}^{i_{l+1}},...,s_{k}^{i_{k}})$ in the same way and it holds
        \begin{equation*}
            C(\tau;\Delta_{i_{1}} s_{1},...,\Delta_{i_{l}} s_{l},s_{l+1}^{i_{l+1}},...,s_{k}^{i_{k}})=R_{k}(\overline{C}(\tau;\Delta_{i_{1}} s_{1},...,\Delta_{i_{l}} s_{l},s_{l+1}^{i_{l+1}},...,s_{k}^{i_{k}})).
        \end{equation*}
    \end{definition}

    \begin{definition}\label{Def Delta_i s}
    We will denote 
    \begin{align*}
        \underline{s} & =(s_{j}^{i})_{\substack{1\leq j\leq k \\ 1\leq i\leq q}}\\
        \underline{i} & =(i_{1},...,i_{k})\\
        \Delta_{\underline{i}}\underline{s} & =(\Delta_{i_{1}} s_{1},...,\Delta_{i_{k}}s_{k}).
    \end{align*}
\end{definition}

    \begin{definition}
        For each $1\leq l\leq k$ and $s>0$ denote
        \begin{align*}
            \overline{M}^{l}_{s} & =M\setminus N_{s}A_{l}\\
            M^{l}_{s} & =M\setminus\overline{N_{s}A_{l}}.
        \end{align*}
        Given a collection $\{(s_{1}^{i},s_{2}^{i},...,s_{k}^{i}):1\leq i\leq q\}$ as defined in the beginning of the section, we denote
        \begin{align*}
            M^{l}_{\Delta_{1}\underline{s}} & =M^{l}_{s_{l}^{1}},\\
            \overline{M}^{l}_{\Delta_{1}\underline{s}} & =\overline{M}^{l}_{s_{l}^{1}}
        \end{align*}
        and
        \begin{align*}
        M^{l}_{\Delta_{i}\underline{s}} & =M^{l}_{s_{l}^{i}}\setminus\overline{M}^{l}_{s_{l}^{i-1}},\\
        \overline{M}^{l}_{\Delta_{i}\underline{s}} & =\overline{M}^{l}_{s_{l}^{i}}\setminus M^{l}_{s_{l}^{i-1}}
        \end{align*}
        for $2\leq i\leq q$. Finally, we set
        \begin{align*}
            M_{(\Delta_{i_{1}}s_{1},...,\Delta_{i_{l}}s_{l},s_{l+1}^{i_{l+1}},...,s_{k}^{i_{k}})}& =[\bigcap_{j=1}^{l}M^{j}_{\Delta_{i_{j}}\underline{s}}]\bigcap[\bigcap_{j=l+1}^{k}M^{j}_{s_{j}^{i_{j}}}],\\
            \overline{M}_{(\Delta_{i_{1}}s_{1},...,\Delta_{i_{l}}s_{l},s_{l+1}^{i_{l+1}},...,s_{k}^{i_{k}})}& =[\bigcap_{j=1}^{l}\overline{M}^{j}_{\Delta_{i_{j}}\underline{s}}]\bigcap[\bigcap_{j=l+1}^{k}\overline{M}^{j}_{s_{j}^{i_{j}}}].
        \end{align*}
        One case of interest of the previous definition is when $l=k$, where using the notation from Definition \ref{Def Delta_i s} we have
        \begin{align*}
            M_{\Delta_{\underline{i}}\underline{s}} & =\bigcap_{j=1}^{k}M^{j}_{\Delta_{i_{j}}\underline{s}}\\
            \overline{M}_{\Delta_{\underline{i}}\underline{s}} & =\bigcap_{j=1}^{k}\overline{M}^{j}_{\Delta_{i_{j}}\underline{s}}.
        \end{align*}

    \end{definition}

    \begin{remark}\label{Remark disjoint omegas}
        Notice that
        \begin{equation*}
            \bigcup_{\underline{i}\in\{1,...,q\}^{k}}\overline{M}_{\Delta_{\underline{i}}\underline{s}}=\overline{M}_{(s_{1}^{q},...,s_{k}^{q})}
        \end{equation*}
         and more generally, given $(j_{1},...,j_{k})\in\{1,...,q\}^{k}$
          \begin{equation*}
            \bigcup_{\underline{i}:1\leq i_{r}\leq j_{r}}\overline{M}_{\Delta_{\underline{i}}\underline{s}}=\overline{M}_{(s_{1}^{j_{1}},...,s_{k}^{j_{k}})}
        \end{equation*}
        and the union is disjoint (except for the boundaries of the involved domains).
    \end{remark}

    \begin{remark}\label{Remark CDelta}
        Observe that
        \begin{align*}
            \overline{C}(\tau;\Delta_{\underline{i}}\underline{s}) & =\tau\llcorner \bigcap_{j=1}^{k}\overline{M}^{j}_{\Delta_{i_{j}} s}\\
            & =\tau\llcorner M_{\Delta_{\underline{i}}\underline{s}}.
        \end{align*}
        Therefore, using Remark \ref{Remark disjoint omegas} and the fact that $\tau\pitchfork D_{s_{j}^{i}}A_{j}$ for every $i,j$ we can see that
        \begin{equation*}
        \sum_{i_{1}=1}^{q}...\sum_{i_{k}=1}^{q}\overline{C}(\tau;\Delta_{\underline{i}}\underline{s})=\overline{C}(\tau;s_{1}^{q},...,s_{k}^{q})
        \end{equation*}
        and more generally, given $(j_{1},...,j_{k})\in\{1,...,q\}^{k}$
        \begin{equation*}
        \sum_{i_{1}=1}^{j_{1}}...\sum_{i_{k}=1}^{j_{k}}\overline{C}(\tau;\Delta_{\underline{i}}\underline{s})=\overline{C}(\tau;s_{1}^{j_{1}},...,s_{k}^{j_{k}})
        \end{equation*}
        and the mass functional is additive when applied to both sides of the previous equation. Therefore, we can think of the $\overline{C}(\tau;\Delta_{\underline{i}}\underline{s})$ as small portions of $\tau$ that we can glue together to construct any $\overline{C}(\tau;s_{1}^{j_{1}},...,s_{k}^{j_{k}})$.
    \end{remark}

    \begin{definition}
        Given  a collection $\{(s_{1}^{i},s_{2}^{i},...,s_{k}^{i}):1\leq i\leq q\}$ as above, $1\leq l\leq k$ and $1\leq i\leq q$, we define $\tau\llcorner D^{l}_{\Delta_{i}\underline{s}}$ as follows. For $i=1$ we set
        \begin{equation*}
            \tau\llcorner D^{l}_{\Delta_{1}\underline{s}}=\tau\llcorner D_{s_{l}^{1}}A_{l}
        \end{equation*}
        and for $2\leq i\leq q$,
        \begin{equation*}
            \tau\llcorner D^{l}_{\Delta_{i}\underline{s}}=\tau\llcorner D_{s^{i}_{l}}A_{l}-\tau\llcorner D_{s^{i-1}_{l}}A_{l}.
        \end{equation*}
    \end{definition}

    \begin{definition}
        For each $1\leq j\leq l$, we define
        \begin{equation*}
            \partial_{j}\overline{C}(\tau;\Delta_{i_{1}}s_{1},...,\Delta_{i_{l}}s_{l},s_{l+1}^{i_{l+1}},...,s_{k}^{i_{k}})=[\tau\llcorner D^{j}_{\Delta_{i_{j}}\underline{s}}]\llcorner[\bigcap_{\substack{1\leq j'\leq l \\ j'\neq j}}M^{j'}_{\Delta_{i_{j'}}\underline{s}}]\cap[\bigcap_{j'=l+1}^{k}M^{j'}_{s_{j'}^{i_{j'}}}]
        \end{equation*}
        and for $l<j\leq k$,
        \begin{equation*}
            \partial_{j}\overline{C}(\tau;\Delta_{i_{1}}s_{1},...,\Delta_{i_{l}}s_{l},s_{l+1}^{i_{l+1}},...,s_{k}^{i_{k}})=[\tau\llcorner D_{s_{j}^{i_{j}}}A_{j}]\llcorner[\bigcap_{1\leq j'\leq l}M^{j'}_{\Delta_{i_{j'}} s}]\cap[\bigcap_{\substack{l+1\leq j'\leq k \\ j'\neq j}}M^{j'}_{s^{i_{j'}}_{j'}}].
        \end{equation*}
        We also set
        \begin{equation*}
            \partial_{j}C(\tau;\Delta_{i_{1}}s_{1},...,\Delta_{i_{l}}s_{l},s_{l+1}^{i_{l+1}},...,s_{k}^{i_{k}})=R_{k}(\partial_{j}\overline{C}(\tau;\Delta_{i_{1}}s_{1},...,\Delta_{i_{l}}s_{l},s_{l+1}^{i_{l+1}},...,s_{k}^{i_{k}})).
        \end{equation*}
    \end{definition}

    \begin{remark}\label{Remark boundary Delta}
        Notice that if $i_{l}>1$,
        \begin{multline*}
        \partial_{j}\overline{C}(\tau;\Delta_{i_{1}}s_{1}...\Delta_{i_{l}}s_{l},s_{l+1}^{i_{l+1}},...,s_{k}^{i_{k}})=\partial_{j}\overline{C}(\tau;\Delta_{i_{1}}s_{1}...\Delta_{i_{l-1}}s_{l-1},s_{l}^{i_{l}},...,s_{k}^{i_{k}})\\
            -\partial_{j}\overline{C}(\tau;\Delta_{i_{1}}s_{1}...\Delta_{i_{l-1}}s_{l-1},s_{l}^{i_{l}-1},...,s_{k}^{i_{k}}).
        \end{multline*}
    \end{remark}

    \begin{proposition}\label{Proposition boundary}
        It holds
        \begin{equation*}
            \partial\overline{C}(\tau;\Delta_{i_{1}}s_{1},...,\Delta_{i_{l}}s_{l},s_{l+1}^{i_{l+1}},...,s_{k}^{i_{k}})=\sum_{j=1}^{k}\partial_{j}\overline{C}(\tau;\Delta_{i_{1}}s_{1},...,\Delta_{i_{l}}s_{l},s_{l+1}^{i_{l+1}},...,s_{k}^{i_{k}})
        \end{equation*}
        and the same is true replacing $\overline{C}$ by $C$. In addition,
        \begin{equation*}
            \support(\partial_{j}C(\tau;\Delta_{i_{1}}s_{1},...,\Delta_{i_{l}}s_{l},s_{l+1}^{i_{l+1}},...,s_{k}^{i_{k}}))\subseteq\Sigma^{j}
        \end{equation*}
        for every $1\leq j\leq k$.
    \end{proposition}

    \begin{proof}
        We can proceed by induction in $l$, using Remark \ref{Remark boundary Delta} and Proposition \ref{Prop 1}.
    \end{proof}

    \subsection{Definition of $C(\tau;\Delta_{\underline{i}}\underline{s})_{\underline{D}}$}

    In this section we focus on the case $k=n-2$, which is the relevant one in order to prove the Parametric Coarea Inequality.

    \begin{definition}
    Consider an enumeration $\{D_{i}\}_{i}$ of the closed top dimensional cells of $\Sigma^{n-2}$. If $D=D_{i}$, denote $n(D)=i$. We define
    \begin{equation*}
        \tilde{D}=D\setminus\bigcup_{i=1}^{n(D)-1}D_{i}.
    \end{equation*}
    Notice that the $\tilde{D}$'s are disjoint Borel subsets of $\Sigma^{n-2}$ and their union is $\Sigma^{n-2}$. For each $1\leq l<n-2$ and each top dimensional cell $D$ of $\Sigma^{l}$, we will denote $\tilde{D}=D$ (we introduce this notation to avoid having distinguishing between the cases $l<n-2$ and $l=n-2$ in the following definitions).
\end{definition}

\begin{definition}
    For $1\leq j\leq n-2$, let $\mathcal{F}^{j}$ denote the set of top dimensional faces of $\Sigma^{j}$. Let $\mathcal{F}^{n-1}$ denote the set of $1$-faces of $\Sigma^{n-2}$ and $\mathcal{F}^{n}$ denote the set of its $0$-cells.
\end{definition}

\begin{definition}
    Given $1\leq j\leq n-2$, we denote
    \begin{equation*}
        \interior(\Sigma^{j})=\bigcup_{D\in\mathcal{F}^{j}}\interior(D).
    \end{equation*}
\end{definition}

\begin{definition}
    Let $1\leq j\leq n-2$. We will denote $I_{j}=(i_{j},D_{j})$ a pair where $1\leq i_{j}\leq q$ and
    \[
    D_{j}=\begin{cases}
        * & \text{ if } i_{j}=1\\
        \text{a top dimensional cell of }\Sigma^{j} & \text{ if }2\leq j\leq q.
    \end{cases}
    \]
    We set $\mathcal{I}_{j}$ to be the set of all possible $I_{j}$, this means
    \begin{equation*}
        \mathcal{I}_{j}=\{(1,*)\}\cup\{(i_{j},D_{j}):2\leq i_{j}\leq q,D_{j}\in\mathcal{F}^{j}\}.
    \end{equation*}
    Given $I=(I_{1},...,I_{n-2})$ with $I_{j}\in\mathcal{I}_{j}$ for every $j$, we denote $\underline{i}=(i_{1},...,i_{n-2})$ and $\underline{D}=(D_{1},...,D_{n-2})$. Additionally, given $1\leq k\leq n-2$ we set $\underline{i}_{k}=(i_{1},...,i_{k})$ and $\underline{D}_{k}=(D_{1},...,D_{k})$. We also set
    \begin{equation*}
        \mathcal{I}=\{(I_{1},...,I_{n-2}):I_{j}\in\mathcal{I}_{j}\text{ }\forall\text{ } 1\leq j\leq n-2\}.
    \end{equation*}
\end{definition}

    \begin{definition}
        Let $I\in\mathcal{I}$ and $1\leq k\leq n-2$. We want to define $\overline{C}(\tau;\Delta_{\underline{i}}\underline{s})_{\underline{D}_{k}}$. We do if by induction on $k$.
        If $k=1$, we set
        \[
        \overline{C}(\tau;\Delta_{\underline{i}}\underline{s})_{\underline{D}_{1}}=\begin{cases}
            \overline{C}(\tau;\Delta_{\underline{i}}\underline{s}) & \text{ if }i_{1}=1\\
            \overline{C}(\tau;\Delta_{\underline{i}}\underline{s})\llcorner r_{1}^{-1}(\tilde{D}_{1}) & \text{ if } 2\leq i_{1}\leq q.
        \end{cases}
        \]
        Assume we already defined $\overline{C}(\tau;\Delta_{\underline{i}}\underline{s})_{\underline{D}_{k-1}}$. We let
         \[
        \overline{C}(\tau;\Delta_{\underline{i}}\underline{s})_{\underline{D}_{k}}=\begin{cases}
            \overline{C}(\tau;\Delta_{\underline{i}}\underline{s})_{\underline{D}_{k-1}} & \text{ if }i_{k}=1\\
            \overline{C}(\tau;\Delta_{\underline{i}}\underline{s})_{\underline{D}_{k-1}}\llcorner r_{k}^{-1}(\tilde{D}_{k}) & \text{ if } 2\leq i_{k}\leq q.
        \end{cases}
        \]
       As usual, we set
       \begin{equation*}
           C(\tau;\Delta_{\underline{i}}\underline{s})_{\underline{D}_{k}}=R_{n-2}(\overline{C}(\tau;\Delta_{\underline{i}}\underline{s})_{\underline{D}_{k}}).
       \end{equation*}
        
    \end{definition}

    \begin{remark}\label{Rk sum over Dk}
        If $i_{k}>1$,
        \begin{align*}
            \sum_{D_{k}\in\mathcal{F}^{k}}\overline{C}(\tau;\Delta_{\underline{i}}\underline{s})_{(\underline{D}_{k-1},D_{k})} & =\sum_{D_{k}\in\mathcal{F}^{k}}\overline{C}(\tau;\Delta_{\underline{i}}\underline{s})_{\underline{D}_{k-1}}\llcorner r_{k}^{-1}(\tilde{D}_{k})\\
            & =\overline{C}(\tau;\Delta_{\underline{i}}\underline{s})_{\underline{D}_{k-1}}
        \end{align*}
        and therefore
        \begin{equation*}
            \sum_{D_{k}\in\mathcal{F}^{k}}C(\tau;\Delta_{\underline{i}}\underline{s})_{(\underline{D}_{k-1},D_{k})}=C(\tau;\Delta_{\underline{i}}\underline{s})_{\underline{D}_{k-1}}.
        \end{equation*}
    \end{remark}

    \begin{definition}
        Given $\underline{i}\in\{1,...,q\}^{n-2}$ and $1\leq k\leq n-2$ let $1\leq j_{1}<j_{2}<...<j_{l}\leq k$ be the integers such that  $\{j_{1},...,j_{l}\}=\{1\leq j\leq k:i_{j}\neq 1\}$. Notice that $l=l(k)$ depends on $k$ (and on $\underline{i})$.
    \end{definition}

    \begin{remark}\label{Remark CD}
        Given $I\in\mathcal{I}$, if $l=l(n-2)$ by Remark \ref{Rk sum over Dk}
        \begin{equation*}
            \sum_{D_{j_{1}}\in\mathcal{F}^{j_{1}}}...\sum_{D_{j_{l}}\in\mathcal{F}^{j_{l}}}\overline{C}(\tau;\Delta_{\underline{i}}\underline{s})_{\underline{D}_{n-2}}=\overline{C}(\tau;\Delta_{\underline{i}}\underline{s})
        \end{equation*}
        and the mass functional is additive when applied to both sides of the previous equation. 
    \end{remark}

    \begin{remark}\label{Remark boundary}
        Assume $1\leq k<n-2$ or $k=n-2$ and $i_{n-2}=1$. As $\{r_{l}^{-1}(D)\setminus N_{s_{l+1}^{i_{l+1}}}A_{l+1}:D\text{ top cell of }\Sigma^{l}\}$ are subsets of $M$ with disjoint closures for every $l\leq k$, we can see that
        \begin{align*}
            \partial[\overline{C}(\tau;\Delta_{\underline{i}}\underline{s})_{\underline{D}_{k}}]& =\partial\overline{C}(\tau;\Delta_{\underline{i}}\underline{s})\llcorner\bigg[ \bigcap _{r=1}^{l(k)}r_{j_{r}}^{-1}(D_{j_{r}})\bigg]\\
           & =\partial\overline{C}(\tau;\Delta_{\underline{i}}\underline{s})\llcorner\bigg[ \bigcap _{r=1}^{l(k)}r_{j_{r}}^{-1}(D_{j_{r}}^{0})\bigg]
        \end{align*}
        where $D_{j_{r}}^{0}$ denotes the interior of the cell $D_{j_{r}}$.
    \end{remark}

    \begin{definition}
        For each $I\in\mathcal{I}$, $1\leq k\leq n-2$  and $1\leq j\leq n-2$, we define
        \begin{equation*}
            \partial_{j}\overline{C}(\tau;\Delta_{\underline{i}}\underline{s})_{\underline{D}_{k}}=\partial_{j}\overline{C}(\tau;\Delta_{\underline{i}}\underline{s})\llcorner\bigg[ \bigcap _{r=1}^{l(k)}r_{j_{r}}^{-1}(D_{j_{r}}^{0})\bigg].
        \end{equation*}
        and
        \begin{equation*}
             \partial_{j}C(\tau;\Delta_{\underline{i}}\underline{s})_{\underline{D}_{k}}=R_{n-2}( \partial_{j}\overline{C}(\tau;\Delta_{\underline{i}}\underline{s})_{\underline{D}_{k}})
        \end{equation*}
        By Remark \ref{Remark boundary} and Proposition \ref{Proposition boundary},
        \begin{equation}\label{Eq boundary sum}
            \sum_{j=1}^{n-2}\partial_{j}\overline{C}(\tau;\Delta_{\underline{i}}\underline{s})_{\underline{D}_{k}}=\partial[\overline{C}(\tau;\Delta_{\underline{i}}\underline{s})_{\underline{D}_{k}}]
        \end{equation}
        if $1\leq k< n-2$ or $k=n-2$ and $i_{n-2}=1$. Also notice that for $k=n-2$ and $i_{n-2}\neq 1$, as $r_{n-2}=R_{n-2}|_{N_{n-2}}$ it holds
        \begin{equation}\label{Eq boundary n-2}
            \partial_{j}C(\tau;\Delta_{\underline{i}}\underline{s})_{\underline{D}_{n-2}}=\partial_{j}C(\tau;\Delta_{\underline{i}}\underline{s})_{\underline{D}_{n-3}}\llcorner D_{n-2}^{0}.
        \end{equation}
    \end{definition}

    \begin{remark}\label{Remark boundary sum}
        If $i_{n-2}=1$ and $l=l(n-2)$, for every and $1\leq j\leq n-2$ it holds
        \begin{equation*}
            \sum_{D_{j_{1}}\in\mathcal{F}^{j_{1}}}...\sum_{D_{j_{l}}\in\mathcal{F}^{j_{l}}}\partial_{j}\overline{C}(\tau;\Delta_{\underline{i}}\underline{s})_{\underline{D}_{n-2}}=\partial_{j}\overline{C}(\tau;\Delta_{\underline{i}}\underline{s}).
        \end{equation*}
        The mass functional is additive when applied to both sides of the equation. The same equality holds replacing $\overline{C}$ by $C$. When $i_{n-2}\neq 1$, it holds
        \begin{equation*}
            \sum_{D_{j_{1}}\in\mathcal{F}^{j_{1}}}...\sum_{D_{j_{l}}\in\mathcal{F}^{j_{l}}}\partial_{j} C(\tau;\Delta_{\underline{i}}\underline{s})_{\underline{D}_{n-2}}=\partial_{j} C(\tau;\Delta_{\underline{i}}\underline{s})\llcorner\interior(\Sigma^{n-2})
        \end{equation*}
        and the mass is also additive.
    \end{remark}

    We want an equation similar to (\ref{Eq boundary sum}) for $k=n-2$ and $i_{n-2}\neq 1$. For that purpose, we need the following lemma.

   \begin{lemma}\label{Lemma 1}
        Given a $1$-chain $\eta$ supported in $\Sigma^{n-2}$ and $D$ a closed top cell of $\Sigma^{n-2}$, we can write
        \begin{equation*}
            \partial[\eta\llcorner\tilde{D}]=\partial\eta\llcorner\interior(D)+\mu
        \end{equation*}
        for some $0$-chain $\mu$ supported in $\partial D$.
    \end{lemma}

    \begin{proof}
        We know
        \begin{equation*}
            \partial[\eta\llcorner D]=\partial\eta\llcorner \interior(D)+\mu_{1}
        \end{equation*}
        for some $0$-chain $\mu_{1}$ supported in $\partial D$ (when $\eta\pitchfork\partial D$ this is true for $\mu_{1}=\eta\llcorner\partial D$, and an arbitrary $\eta\in\mathcal{I}_{1}(\Sigma^{n-2})$ can be obtained as a limit of a sequence $\{\eta_{n}\}_{n}$ in the flat topology with $\eta_{n}\pitchfork\partial D$). We know that $\tilde{D}=D\setminus E$ where $E$ is a (possibly empty) Borel subset of $\partial D$. Thus
        \begin{align*}
            \partial[\eta\llcorner\tilde{D}] & =\partial[\eta\llcorner D]-\partial[\eta\llcorner E]\\
            & =\partial\eta\llcorner\interior(D)+\mu_{1}-\partial[\eta\llcorner E].
        \end{align*}
        
        As $\mu_{1}-\partial[\eta\llcorner E]$ is supported in $\partial D$, we get the desired result.
    \end{proof}

    \begin{definition}
        Applying the previous lemma to the $1$-chain $C(\tau;\Delta_{\underline{i}}\underline{s})_{\underline{D}_{n-3}}$ with $i_{n-2}>1$ and a top cell $D_{n-2}$ of $\Sigma^{n-2}$, we obtain
    \begin{align}\label{Eq 1}
        \partial C(\tau;\Delta_{\underline{i}}\underline{s})_{\underline{D}} & =\partial C(\tau;\Delta_{\underline{i}}\underline{s})_{\underline{D}_{n-3}}\llcorner\interior(D_{n-2})+B(\tau;\Delta_{\underline{i}}\underline{s})_{\underline{D}}
    \end{align}
    for a certain $0$-chain $B(\tau;\Delta_{\underline{i}}\underline{s})_{\underline{D}}$ supported in $\partial D$.
    \end{definition}

    \begin{remark}
    By (\ref{Eq 1}), (\ref{Eq boundary sum}) and (\ref{Eq boundary n-2}), if $i_{n-2}>1$
        \begin{align*}
            \partial C(\tau;\Delta_{\underline{i}}\underline{s})_{\underline{D}} & =\bigg[\sum_{j=1}^{n-2}\partial_{j}C(\tau;\Delta_{\underline{i}}\underline{s})_{\underline{D}_{n-3}}\llcorner\interior(D_{n-2})\bigg]+B(\tau;\Delta_{\underline{i}}\underline{s})_{\underline{D}}\\
            & =\big[\sum_{j=1}^{n-2}\partial_{j}C(\tau;\Delta_{\underline{i}}\underline{s})_{\underline{D}}\big]+B(\tau;\Delta_{\underline{i}}\underline{s})_{\underline{D}}.
        \end{align*}
    \end{remark}

    \subsection{Definition of $A(\tau;\Delta_{\underline{i}}\underline{s})_{\underline{D}}$}\label{Section AD}

    Fix a collection of admissible tuples $\underline{s}=(s_{j}^{i})_{\substack{1\leq j\leq n-2 \\ 1\leq i\leq q}}$ for $\tau\in\mathcal{I}_{1}(M)_{\Sigma}$.

    \begin{definition}
        Given $1\leq l\leq n-2$ and $D$ a top dimensional cell of $\Sigma^{l}$, we denote $q_{D}$ the central point of the cell $D$. For $D$ a $1$-cell of $\Sigma^{n-2}$ we choose $q_{D}$ to be one of its vertices and for $D$ a $0$-cell we just denote $q_{D}=D$.
    \end{definition}

    \begin{definition}
        Given $\eta\in\mathcal{I}_{1}(M)_{\Sigma}$ and $1\leq l\leq n-1$, we define
        \begin{equation*}
            \Cone_{l}(\eta)=\sum_{E\in\mathcal{F}^{l}}\Cone_{q_{E}}(\partial\eta\llcorner E^{0})
        \end{equation*}
        and we set
        \begin{equation*}
            \Cone(\eta)=\sum_{l=1}^{n-1}\Cone_{l}(\eta).
        \end{equation*}
    \end{definition}
    
\begin{definition}\label{Definition of A}
    Given $I=(I_{1},...,I_{n-2})\in\mathcal{I}$ and $1\leq k\leq n-2$ we define
    \begin{equation*}
        A(\tau;\Delta_{\underline{i}}\underline{s})_{\underline{D}_{k}}=C(\tau;\Delta_{\underline{i}}\underline{s})_{\underline{D}_{k}}+\Cone(C(\tau;\Delta_{\underline{i}}\underline{s})_{\underline{D}_{k}}).
    \end{equation*}
\end{definition}

\begin{remark}\label{Rk boundary A}
    Observe that 
    \begin{equation*}
        \support(\partial A(\tau;\Delta_{\underline{i}}\underline{s})_{\underline{D}_{k}})\subseteq\bigcup_{l=1}^{n}\bigcup_{E\in\mathcal{F}^{l}}\{q_{E}\}
    \end{equation*}
    for every $I\in\mathcal{I}$ and every $1\leq k\leq n-2$. Moreover, given $j$ such that $i_{j}\neq 1$ it holds
    \begin{equation*}
        \support(\partial A(\tau;\Delta_{\underline{i}}\underline{s})_{\underline{D}_{k}})\subseteq\bigcup_{l=j}^{n}\bigcup_{\substack{E\in\mathcal{F}^{l} \\ E\text{ face of }D_{j} }}\{q_{E}\}.
    \end{equation*}
\end{remark}

\begin{remark}
    Fix $I\in\mathcal{I}$ and $1\leq k\leq n-2$. By Remark \ref{Remark CD}, if $1\leq j_{1}<...<j_{l}\leq k$ are the positions $1\leq j\leq k$ for which $i_{j}\neq 1$ then
    \begin{equation*}
            \sum_{D_{j_{1}}\in\mathcal{F}^{j_{1}}}...\sum_{D_{j_{l}}\in\mathcal{F}^{j_{l}}} A(\tau;\Delta_{\underline{i}}\underline{s})_{\underline{D}_{k}}= A(\tau;\Delta_{\underline{i}}\underline{s}).
        \end{equation*}
\end{remark}

\subsection{Interpolation formula}\label{Section Interpolation formula}

Fix $\tau\in\mathcal{I}_{1}(M)_{\Sigma}$ and a collection $\underline{s}=(s_{1}^{i},...,s_{n-2}^{i})_{1\leq i\leq q}$ of admissible cuts for $\tau$ with $s_{j}^{1}\geq s_{j}^{2}\geq...\geq s_{j}^{q}$ for every $1\leq j\leq n-2$. We want to define a continuous way to interpolate between the chains $A(\tau;s_{1}^{i_{1}},s_{2}^{i_{2}},...,s_{n-2}^{i_{n-2}})$, $(i_{1},...,i_{n-2})\in\{1,...,q\}^{n-2}$ with control over the mass and the mass of the boundary of the intermediate chains. That will be done by rescaling chains of the form $A(\tau;\Delta_{\underline{i}}\underline{s})_{\underline{D}}$, because given two pairs $\underline{i}=(i_{1},...,i_{n-2})$ and $\underline{j}=(j_{1},...,j_{n-2})$ the difference $A(\tau;s_{1}^{i_{1}},...,s_{n-2}^{i_{n-2}})-A(s_{1}^{j_{1}},...,s_{n-2}^{j_{n-2}})$ can be expressed as a sum of terms of the form $A(\tau;\Delta_{\underline{i}}\underline{s})$ as stated in Remark \ref{Remark CDelta}. For that purpose, given $1\leq k\leq n-2$ and $(I_{1},...,I_{k})\in \mathcal{I}_{1}\times...\times\mathcal{{I}}_{k}$, we will need a coefficient $\mu^{\underline{D}_{k}}_{\underline{i}_{k}}=\mu^{D_{1}...D_{k}}_{i_{1}...i_{k}}\in[0,1]$ and we will denote
\begin{equation*}
    \tilde{\mu}^{\underline{D}_{k}}_{\underline{i}_{k}}=\prod_{j=1}^{i_{k}} \mu^{D_{1}...D_{k}}_{i_{1}...i_{k-1}j}.
\end{equation*}
The coefficient $\mu^{\underline{D}_{k}}_{\underline{i}_{k}}$ (and therefore $\tilde{\mu}^{\underline{D}_{k}}_{\underline{i}_{k}}$ as well) must be equal to $1$ in case $i_{k}=1$. Let $\tilde{T}^{\underline{D}_{k}}_{\underline{i}_{k}}:D_{k}\to D_{k}$ be the homotecy
\begin{equation*}
    x\mapsto\tilde{\mu}^{\underline{D}_{k}}_{\underline{i}_{k}}\cdot x
\end{equation*}
centered at $q_{D_{k}}$ of ratio $\tilde{\mu}^{\underline{D}_{k}}_{\underline{i}_{k}}$. Denote $T^{\underline{D}_{k}}_{\underline{i}_{k}}:D_{k}\to D_{1}$ the composition
\begin{equation*}
    T^{\underline{D}_{k}}_{\underline{i}_{k}}=\tilde{T}^{\underline{D}_{1}}_{\underline{i}_{1}}\circ\tilde{T}^{\underline{D}_{2}}_{\underline{i}_{2}}\circ...\circ\tilde{T}^{\underline{D}_{k}}_{\underline{i}_{k}}.
\end{equation*}

We denote by $\mu$ a collection of coefficients $\mu^{\underline{D}_{k}}_{\underline{i}_{k}}$ for each $1\leq k\leq n-2$ and each $(I_{1},...,I_{k})\in\mathcal{I}_{1}\times...\times\mathcal{I}_{k}$ as above.

\begin{definition}
    Let
    \begin{align*}
        A(\tau;\underline{s},\mu) & =\sum_{I\in\mathcal{I}}T^{\underline{D}_{n-2}}_{\underline{i}_{n-2}}(A(\tau;\Delta_{\underline{i}}\underline{s})_{\underline{D}_{n-2}})\\
        & =\sum_{I_{1}\in\mathcal{I}_{1}}...\sum_{I_{n-2}\in\mathcal{I}_{n-2}}\tilde{T}^{\underline{D}_{1}}_{\underline{i}_{1}}\circ\tilde{T}^{\underline{D}_{2}}_{\underline{i}_{2}}\circ ...\circ\tilde{T}^{\underline{D}_{n-2}}_{\underline{i}_{n-2}}(A(\tau;\Delta_{\underline{i}}\underline{s})_{\underline{D}_{n-2}}).
    \end{align*}
\end{definition}

We want to find an upper bound for $\mass(A(\tau;\underline{s},\mu))$. Notice that by Definition \ref{Definition of A}, we can write
\begin{equation*}
    A(\tau;\underline{s},\mu)=C(\tau;\underline{s},\mu)+\Cone(\tau;\underline{s},\mu)
\end{equation*}
where
\begin{equation*}
    C(\tau;\underline{s},\mu) =\sum_{I\in\mathcal{I}}T^{\underline{D}_{n-2}}_{\underline{i}_{n-2}}(C(\tau;\Delta_{\underline{i}}\underline{s})_{\underline{D}_{n-2}})
\end{equation*}
and
\begin{equation*}
    \Cone(\tau;\underline{s},\mu) =\sum_{I\in\mathcal{I}}T^{\underline{D}_{n-2}}_{\underline{i}_{n-2}}(\Cone(C(\tau;\Delta_{\underline{i}}\underline{s})_{\underline{D}_{n-2}})).
\end{equation*}
Observe that as the $T^{\underline{D}_{n-2}}_{\underline{i}_{n-2}}$ are compositions of homothecies of ratio less or equal than $1$,
\begin{align*}
    \mass(C(\tau;\underline{s},\mu)) & \leq\sum_{I\in\mathcal{I}}\mass(C(\tau;\Delta_{\underline{i}}\underline{s})_{\underline{D}_{n-2}})\\
    & \leq (1+\beta)\sum_{I\in\mathcal{I}}\mass(\overline{C}(\tau;\Delta_{\underline{i}}\underline{s})_{\underline{D}_{n-2}})\\
    & =(1+\beta)\mass(\overline{C}(\tau;s_{1}^{q},...,s_{n-2}^{q}))\\
    & \leq (1+\beta)\mass(\tau)
\end{align*}
where $1+\beta=\prod_{l=1}^{n-2}\frac{1}{1-2\frac{\varepsilon_{l}}{\rho'_{l-1}}}=(1-\frac{2}{\sqrt{n_{p}}})^{-(n-2)}$ is the Lipschitz constant of $R_{n-2}$ (see Section \ref{Section cuts} and observe that $\lim_{p\to\infty}\beta_{p}=0$ and that we get the same $\beta_{p}$ if $M$ has width at least $\rho_{0}$ instead of at least $1$). Hence it suffices to find an upper bound for $\mass(\Cone(\tau;\underline{s},\mu))$.

\begin{definition}
    Given $1\leq k\leq n-2$, let us denote
\begin{equation}\label{Def cone k}
    \Cone_{k}(\tau;\underline{s},\mu)=\sum_{I\in\mathcal{I}}T^{\underline{D}_{n-2}}_{\underline{i}_{n-2}}(\Cone_{k}(C(\tau;\Delta_{\underline{i}}\underline{s})_{\underline{D}_{n-2}}))
\end{equation}
being
\begin{equation*}
    \Cone(\tau;\underline{s},\mu)=\sum_{k=1}^{n-2}\Cone_{k}(\tau;\underline{s},\mu).
\end{equation*}
\end{definition}

\begin{remark}
    Notice that
    \begin{equation}\label{Def 2 cone k}
    \Cone_{k}(\tau;\underline{s},\mu)=\sum_{I_{1}\in\mathcal{I}_{1}}...\sum_{I_{k}\in\mathcal{I}_{k}}T^{\underline{D}_{k}}_{\underline{i}_{k}}(\Cone_{k}(C(\tau;\Delta_{\underline{i}}\underline{s})_{\underline{D}_{k}}))
\end{equation}
    were $\underline{i}=(i_{1},...,i_{k},1,...,1)$.
\end{remark}

\begin{proof}
    Observe that if $I\in\mathcal{I}$ verifies $i_{j}>1$ for some $j>k$, $C(\tau;\Delta_{\underline{i}}\underline{s})_{\underline{D}_{n-2}}$ is supported in $\Sigma^{j}$ and hence
\begin{align*}
    \Cone_{k}(C(\tau;\Delta_{\underline{i}}\underline{s})_{\underline{D}_{n-2}})& =\sum_{E\in\mathcal{F}^{k}}\Cone_{q_{E}}(\partial C(\tau;\Delta_{\underline{i}}\underline{s})_{\underline{D}_{n-2}}\llcorner E^{0}) \\
    & =0
\end{align*}
as $E^{0}\cap\Sigma^{j}=\emptyset$ for each $E\in\mathcal{F}^{k}$. Therefore, in (\ref{Def cone k}) we only need to consider those $I\in\mathcal{I}$ for which $I_{j}=(1,*)$ for every $j>k$ which yields the desired result.
\end{proof}

\begin{definition}
    For each $1\leq j\leq n-2$, we define $\mathcal{I}_{(j)}=\mathcal{I}_{1}\times...\times\mathcal{I}_{j}$. Let $I_{(j)}=(I_{1},...,I_{j})$ denote an element of $\mathcal{I}_{(j)}$.
\end{definition}

\begin{definition}
    Given $\underline{i}_{k-1}=(i_{1},...,i_{k-1})$ we denote
\begin{equation*}
    (\Delta_{\underline{i}_{k-1}},s_{k}^{j})=(\Delta_{i_{1}}s_{1},...,\Delta_{i_{k-1}}s_{k-1},s_{k}^{j},s_{k+1}^{1},...,s_{n-2}^{1}).
\end{equation*}    
\end{definition}

\begin{proposition}\label{Prop cone k bound}
    If $1\leq k\leq n-2$,
    \begin{multline*}
    \mass(\Cone_{k}(\tau;\underline{s},\mu)) \leq \\
    \sum_{I_{(k-1)}\in\mathcal{I}_{(k-1)}}\sum_{D_{k}\in\mathcal{F}_{k}}\sum_{i_{k}=1}^{q}\rho_{k}(\tilde{\mu}^{(\underline{D}_{k-1},D_{k})}_{(\underline{i}_{k-1},i_{k})}-\tilde{\mu}^{(\underline{D}_{k-1},D_{k})}_{(\underline{i}_{k-1},i_{k}+1)})\mass(\partial C(\tau;(\Delta_{\underline{i}_{k-1}}\underline{s},s_{k}^{i_{k}}))_{\underline{D}_{k-1}})\llcorner D_{k}^{0})).
\end{multline*}
\end{proposition}

\begin{proof}
We will write
\begin{equation*}
    \Cone_{k}(\tau;\underline{s},\mu)=\Cone_{k}^{1}(\tau;\underline{s},\mu)+\Cone_{k}^{2}(\tau;\underline{s},\mu)
\end{equation*}
where $\Cone_{k}^{1}(\tau;\underline{s},\mu)$ corresponds to adding the terms with $I_{k}=(1,*)$ in (\ref{Def 2 cone k}) and $\Cone^{k}_{2}(\tau;\underline{s},\mu)$ is the sum of all the other terms. Hence it holds
\begin{align*}
    \Cone_{k}^{1}(\tau;\underline{s},\mu) & =\sum_{I_{1}\in\mathcal{I}_{1}}...\sum_{I_{k-1}\in\mathcal{I}_{k-1}}T^{(\underline{D}_{k-1},*)}_{(\underline{i}_{k-1},1)}(\Cone_{k}(C(\tau;(\Delta_{\underline{i}_{k-1}}\underline{s},s_{k}^{1}))_{\underline{D}_{k-1}}))\\
    & =\sum_{I_{(k-1)}\in\mathcal{I}_{(k-1)}}T^{\underline{D}_{k-1}}_{\underline{i}_{k-1}}\circ\Tilde{T}^{(\underline{D}_{k-1},*)}_{(\underline{i}_{k-1},1)}(\Cone_{k}(C(\tau;(\Delta_{\underline{i}_{k-1}}\underline{s},s_{k}^{1}))_{\underline{D}_{k-1}}))\\
    & =\sum_{I_{(k-1)}\in\mathcal{I}_{(k-1)}}T^{\underline{D}_{k-1}}_{\underline{i}_{k-1}}\bigg (\sum_{D_{k}\in\mathcal{F}^{k}} \tilde{T}^{(\underline{D}_{k-1},*)}_{(\underline{i}_{k-1},1)}(\Cone_{q_{D_{k}}}(\partial C(\tau;(\Delta_{\underline{i}_{k-1}}\underline{s},s_{k}^{1}))_{\underline{D}_{k-1}}\llcorner D_{k}^{0}))) \bigg ).
\end{align*}
On the other hand,
\begin{multline*}
    \Cone_{k}^{2}(\tau;\underline{s},\mu)=\\\sum_{I_{1}\in\mathcal{I}_{1}}...\sum_{I_{k-1}\in\mathcal{I}_{k-1}}\sum_{i_{k}=2}^{q}\sum_{D_{k}\in\mathcal{F}_{k}}T^{(\underline{D}_{k-1},D_{k})}_{(\underline{i}_{k-1},i_{k})}(\Cone_{k}(\tau;(\Delta_{\underline{i}_{k-1}}\underline{s},\Delta_{i_{k}}s_{k}))_{(\underline{D}_{k-1},D_{k})})\\
     =\sum_{I_{(k-1)}\in\mathcal{I}_{(k-1)}}\sum_{i_{k}=2}^{q}\sum_{D_{k}\in\mathcal{F}_{k}}T^{(\underline{D}_{k-1},D_{k})}_{(\underline{i}_{k-1},i_{k})}(\Cone_{q_{D_{k}}}(\partial C(\tau;(\Delta_{\underline{i}_{k-1}}\underline{s},\Delta_{i_{k}}s_{k}))_{(\underline{D}_{k-1},D_{k})})\llcorner D_{k}^{0})\\
    =\sum_{I_{(k-1)}\in\mathcal{I}_{(k-1)}}T^{\underline{D}_{k-1}}_{\underline{i}_{k-1}}\bigg (\sum_{D_{k}\in\mathcal{F}_{k}}\sum_{i_{k}=2}^{q}\tilde{T}^{(\underline{D}_{k-1},D_{k})}_{(\underline{i}_{k-1},i_{k})}(\Cone_{q_{D_{k}}}(\partial C(\tau;(\Delta_{\underline{i}_{k-1}}\underline{s},\Delta_{i_{k}}s_{k}))_{\underline{D}_{k-1}})\llcorner D_{k}^{0})\bigg ).
\end{multline*}
In the previous we used that if $j>1$, by Remark \ref{Rk sum over Dk}
\begin{align*}
    \partial C(\tau;(\Delta_{\underline{i}_{k-1}}\underline{s},\Delta_{j}s_{k}))_{\underline{D}_{k-1}}\llcorner D_{k}^{0} 
 & =\sum_{D'_{k}\in\mathcal{F}_{k}}\partial C(\tau;(\Delta_{\underline{i}_{k-1}}\underline{s},\Delta_{j}s_{k}))_{(\underline{D}_{k-1},D'_{k})}\llcorner D_{k}^{0}\\
    & =\partial C(\tau;(\Delta_{\underline{i}_{k-1}}\underline{s},\Delta_{j}s_{k}))_{(\underline{D}_{k-1},D_{k})}\llcorner D_{k}^{0}
\end{align*}
because $\support(\partial C(\tau;(\Delta_{\underline{i}_{k-1}}\underline{s},\Delta_{j}s_{k}))_{(\underline{D}_{k-1},D'_{k})})\subseteq D'_{k}$ and $D'_{k}\cap D_{k}^{0}=\emptyset$ if $D_{k}'\neq D_{k}$. Because of similar reasons,
\begin{align*}
    \Cone_{k}(\tau;(\Delta_{\underline{i}_{k-1}},\Delta_{j}s_{k}))_{(\underline{D}_{k-1},D_{k})}) & =\sum_{D'_{k}\in\mathcal{F}^{k}}\Cone_{q_{D'_{k}}}(\partial C(\tau;(\Delta_{\underline{i}_{k-1}},\Delta_{j}s_{k}))_{(\underline{D}_{k-1},D_{k})}\llcorner \interior(D'_{k}))\\
    & =\Cone_{q_{D_{k}}}(\partial C(\tau;(\Delta_{\underline{i}_{k-1}},\Delta_{j}s_{k}))_{(\underline{D}_{k-1},D_{k})}\llcorner \interior(D_{k})).
\end{align*}
Therefore, setting  $\tilde{T}^{(\underline{D}_{k-1},D_{k})}_{(\underline{i}_{k-1},1)}=\tilde{T}^{(\underline{D}_{k-1},*)}_{(\underline{i}_{k-1},1)}$ (which in fact is just the identity), we can see that
\begin{multline*}
    \Cone_{k}(\tau;\underline{s},\mu)= \\
    \sum_{I_{(k-1)}\in\mathcal{I}_{(k-1)}}T^{\underline{D}_{k-1}}_{\underline{i}_{k-1}}\bigg (\sum_{D_{k}\in\mathcal{F}_{k}}\sum_{i_{k}=1}^{q}\tilde{T}^{(\underline{D}_{k-1},D_{k})}_{(\underline{i}_{k-1},i_{k})}(\Cone_{q_{D_{k}}}(\partial C(\tau;(\Delta_{\underline{i}_{k-1}}\underline{s},\Delta_{i_{k}}s_{k}))_{\underline{D}_{k-1}})\llcorner D_{k}^{0})\bigg ).
\end{multline*}
We can regroup the previous sum as
\begin{multline*}
    \Cone_{k}(\tau;\underline{s},\mu)=\sum_{I_{(k-1)}\in\mathcal{I}_{(k-1)}}T^{\underline{D}_{k-1}}_{\underline{i}_{k-1}}\bigg (\sum_{D_{k}\in\mathcal{F}_{k}}\sum_{i_{k}=1}^{q}\tilde{\mu}^{(\underline{D}_{k-1},D_{k})}_{(\underline{i}_{k-1},i_{k})}\cdot(\Cone_{q_{D_{k}}}(\partial C(\tau;(\Delta_{\underline{i}_{k-1}}\underline{s},s_{k}^{i_{k}}))_{\underline{D}_{k-1}})\llcorner D_{k}^{0})\\
    -\tilde{\mu}^{(\underline{D}_{k-1},D_{k})}_{(\underline{i}_{k-1},i_{k}+1)}\cdot(\Cone_{q_{D_{k}}}(\partial C(\tau;(\Delta_{\underline{i}_{k-1}}\underline{s},s_{k}^{i_{k}}))_{\underline{D}_{k-1}})\llcorner D_{k}^{0})\bigg )
\end{multline*}
where we set $\tilde{\mu}^{(\underline{D}_{k-1},D_{k})}_{(\underline{i}_{k-1},q+1)}=0$. Taking mass and using that the maps $T^{\underline{D}_{k-1}}_{\underline{i}_{k-1}}$ are mass decreasing,
\begin{align*}
    \mass(\Cone_{k}(\tau;\underline{s},\mu)) \leq \\
    \sum_{I_{(k-1)}\in\mathcal{I}_{(k-1)}}\sum_{D_{k}\in\mathcal{F}_{k}}\sum_{i_{k}=1}^{q}(\tilde{\mu}^{(\underline{D}_{k-1},D_{k})}_{(\underline{i}_{k-1},i_{k})}-\tilde{\mu}^{(\underline{D}_{k-1},D_{k})}_{(\underline{i}_{k-1},i_{k}+1)})\mass(\Cone_{q_{D_{k}}}(\partial C(\tau;(\Delta_{\underline{i}_{k-1}}\underline{s},s_{k}^{i_{k}}))_{\underline{D}_{k-1}})\llcorner D_{k}^{0})).
\end{align*}
Using that $\diameter(D_{k})\leq \rho_{k}$, we get
\begin{align*}
    \mass(\Cone_{k}(\tau;\underline{s},\mu)) \leq \\
    \sum_{I_{(k-1)}\in\mathcal{I}_{(k-1)}}\sum_{D_{k}\in\mathcal{F}_{k}}\sum_{i_{k}=1}^{q}\rho_{k}(\tilde{\mu}^{(\underline{D}_{k-1},D_{k})}_{(\underline{i}_{k-1},i_{k})}-\tilde{\mu}^{(\underline{D}_{k-1},D_{k})}_{(\underline{i}_{k-1},i_{k}+1)})\mass(\partial C(\tau;(\Delta_{\underline{i}_{k-1}}\underline{s},s_{k}^{i_{k}}))_{\underline{D}_{k-1}})\llcorner D_{k}^{0})).
\end{align*}
as desired.
\end{proof}

\subsection{Proof of the Parametric Coarea Inequality for rectangular domains}\label{Section Proof of Coarea Ineq}

The goal of this section is to prove Theorem \ref{Thm Parametric Coarea}. Fix a rectangular domain $M$ and a number $\alpha\in(0,1)$. For each $p\in\mathbb{N}$ define a set of numbers $\{\varepsilon_{l}(p),\rho_{l}(p),\rho'_{l}(p):1\leq l\leq n-2\}$ as in Section \ref{Section cuts}. Let $p_{0}=p_{0}(M,\alpha)$ be such that the constructions from the previous subsections can be performed with the $\varepsilon_{l},\rho_{l},\rho'_{l}$ we just defined. Let $p\geq p_{0}$ and let $F:X^{p}\to \mathcal{Z}_{1}(M,\Sigma)$ be a continuous map without concentration of mass. Fix $\varepsilon>0$. By Theorem \ref{Thm delta localized approximation}, we can assume that $F$ is $(N,\delta)$-localized with $N=N(p)$ and $\delta$ sufficiently small to be determined later. We want to construct a new family $F':X\to\mathcal{I}_{1}(M)_{\Sigma}$ such that $\mass(F'(x))$ is not much larger than $\mass(F(x))$ and also $\mass(\partial F'(x))$ is bounded. For that purpose, for each cell $C$ of $X$ denote $\{U^{C}_{i}\}_{i\in I_{C}}$ the (finite) admissible collection of open balls such that
\begin{equation*}
    \sum_{i\in I_{C}}\radius(U^{C}_{i})< \delta
\end{equation*}
and given $x,y\in C$
\begin{equation*}
    \support(F(x)-F(y))\subseteq \bigcup_{i\in I_{C}}U^{C}_{i}.
\end{equation*}
Denote  
\begin{equation*}
    n(X)=\max\{\text{number of top dimensional cells }C\text{ containing }x:x\in X_{0}\}.
\end{equation*}
Notice that $n(X)\leq 2^d$ if $X$ is a cubical subcomplex of $I^{d}(q)$. For each vertex $x\in X_{0}$, using \cite[Lemma~2.4]{GL22} we can construct an $3n(X)\delta$-admissible family of balls $\{U_{i}^{x}\}_{i\in I_{x}}$ such that
\begin{equation*}
    \bigcup_{i\in I_{x}}U^{x}_{i}\supseteq\bigcup_{\substack{C\text{ cell} \\ \text{containing }x}}\bigcup_{i\in I_{C}}U_{i}^{C}.
\end{equation*}
By considering $\delta'=\delta'_{p}>0$ such that $\delta'<<\rho_{n-2}$ (we can set $\delta'=p^{-\alpha'}\rho_{n-2})$ and taking $\delta<\frac{1}{3n(X)}\delta'$, we can define cuts $s_{j}(x)$ for $1\leq j\leq n-2$ for each $x\in X$ such that $s_{j}(x)$ is $j$-admissible for $F(x)$ as defined in Section \ref{Section cuts} and
\begin{equation*}
    D_{s_{j}(x)}A_{j}\cap\bigg[\bigcup_{i\in I_{x}}U^{x}_{i} \bigg]=\emptyset.
\end{equation*}
We also want to have the property that if $x,y\in C$ for a certain cell $C$ of $X$, then $\mass(F(x)-F(y))\leq\varepsilon$. This can be obtained by choosing $\varepsilon$ and $\delta$ accurately in Theorem \ref{Thm delta localized approximation} when approximating our original family by a $\delta$-localized one, in order to guarantee
\begin{equation*}
    \mass(F(x)\llcorner\bigcup_{i\in I_{C}}U_{i}^{C})\leq\frac{\varepsilon}{2}
\end{equation*}
for every $x\in C$ and hence
\begin{equation*}
    \mass((F(x)-F(y))=\mass((F(x)-F(y))\llcorner(\bigcup_{i\in I_{C}}U_{i}^{C}))\leq\varepsilon.
\end{equation*}
We define
\begin{align}\label{F' at X0}
    F'(x) & =A(F(x),s_{1}(x),...,s_{n-2}(x))\\
    & =C(F(x);s_{1}(x),...,s_{n-2}(x))+\Cone(C(F(x);s_{1}(x),...,s_{n-2}(x)))
\end{align}
for $x\in X_{0}$. We will extend the definition of $F'(x)$ skeleton to skeleton using a similar procedure to that in \cite{GL22}.

\textbf{Inductive property for the $l$-skeleton of $X$.}
For every $p'$-dimensional cell $C$ of $X$, $l\leq p'\leq p$, and for every point $y\in C$ the following holds. Given $1\leq j\leq n-2$, let $x_{j}^{1},x_{j}^{2},...,x_{j}^{2^{p'}}$ be an enumeration of the vertices of $C$ such that $s_{j}(x_{j}^{1})\geq s_{j}(x_{j}^{2})\geq...\geq s_{j}(x_{j}^{2^{p'}})$. Denote $s_{j}^{i}=s_{j}(x_{j}^{i})$. Then:  
\begin{enumerate}[label=(C\arabic*)]
    \item \label{ite : 3.1} For each $1\leq k\leq n-2$ and each $(I_{1},...,I_{k})\in\mathcal{I}_{(k)}=\mathcal{I}_{1}\times...\times\mathcal{I}_{k}$, there exists a continuous function $\mu^{\underline{D}_{k}}_{\underline{i}_{k}}=\mu^{D_{1}...D_{k}}_{i_{1}...i_{k}}:C\cap X^{l}\to[0,1]$ such that
    \begin{enumerate}[label=(\alph*)]
        \item \label{ite : 3.1.1} \begin{equation*}
            \sum_{k=1}^{n-2}\#\{I_{(k)}\in\mathcal{I}_{(k)}:\mu^{\underline{D}_{k}}_{\underline{i}_{k}}(x)\in(0,1)\}\leq\dim(E(x))
        \end{equation*}
        for every $x\in C\cap X_{l}$, where $E(x)$ is the unique cell of $X$ such that $x\in\interior(E(x))$. Moreover if
        \begin{equation*}
            i_{k}(x)=\max\{i:x^{i}_{k}\in E(x)\}
        \end{equation*}
        then $\mu^{D_{1}...D_{k}}_{i_{1}...i_{k}}(x)=0$ if $i_{k}>i_{k}(x)$ and $\mu^{D_{1}...D_{k}}_{i_{1}...i_{k}}(x)=1$ for every $i_{k}\leq i_{k}(x)$ such that $x_{k}^{i_{k}}\notin E(x)$.
        \item \label{ite : 3.1.2} Using the notation of the previous sections, given $1\leq k\leq n-2$ it holds
       \begin{multline}\label{Ineq cones}
           \sum_{I_{k-1}\in\mathcal{I}_{(k-1)}}\sum_{D_{k}\in\mathcal{F}_{k}}\sum_{i_{k}=1}^{q}(\tilde{\mu}^{\underline{D}_{k}}_{(\underline{i}_{k-1},i_{k})}(x)-\tilde{\mu}^{\underline{D}_{k}}_{(\underline{i}_{k-1},i_{k}+1)}(x))\mass(\partial C(F(y);(\Delta_{\underline{i}_{k-1}}\underline{s},s_{k}^{i_{k}}))_{\underline{D}_{k-1}})\llcorner D_{k}^{0}))\\
           \leq k\frac{\mass(F(y))}{\varepsilon_{k}}
       \end{multline}
       for every $x\in C\cap X^{l}$.
    \end{enumerate}
    \item \label{ite: 3.2} There exists a continuous family $e^{y}:C\cap X^{l}\to\mathcal{Z}_{1}(M)$ of absolute cycles of length at most $ c(l)\varepsilon$ and supported in a $(1+\beta)\delta$-admissible family of open sets of $M$.
\end{enumerate}
The previous are such that if $\underline{s}=(s_{j}^{i})_{1\leq j\leq l,1\leq i\leq 2^{p'}}$ and $\mu(x)$ is given by the $\mu^{\underline{D}_{k}}_{\underline{i}_{k}}(x)$, it holds
\begin{equation}\label{Eq F'}
    F'(x)=e^{y}(x)+A(F(y);\underline{s},\mu(x)).
\end{equation}

We will show by induction that $F'$ can be extended to the $l$-skeleton verifying the Inductive Property above.

\textbf{Base case.} For $l=0$, pick a cell $C$ of dimension $0\leq p'\leq p$ and for each $1\leq j\leq n-2$ enumerate the vertices of $C$ as $x_{j}^{1},...,x_{j}^{2^{p'}}$ with $s(x^{1}_{j})\geq...\geq s(x_{j}^{2^{p'}})$. Define
\[
\mu^{\underline{D}_{k}}_{\underline{i}_{k}}(x)=\begin{cases}
    1 & \text{ if } i_{j}\leq i_{j}(x) \text{ for every }1\leq j\leq k\\
    0 & \text{ otherwise}
\end{cases}
\]
where $i_{j}(x)\in\{1,...,2^{p'}\}$ is the unique one such that $x_{j}^{i_{j(x)}}=x$. Thus given $x\in C\cap X_{0}$ we have $\tilde{\mu}^{\underline{D}_{k}}_{\underline{i}_{k}}(x)=\mu^{\underline{D}_{k}}_{\underline{i}_{k}}(x)$
and hence the sum in \ref{ite : 3.1}\ref{ite : 3.1.2} becomes
\begin{multline*}
    \sum_{I_{k-1}\in\mathcal{I}_{(k-1)}}\sum_{D_{k}\in\mathcal{F}_{k}}\mass(\partial C(F(y);(\Delta_{\underline{i}_{k-1}}\underline{s},s_{k}^{i_{k}(x)}))_{\underline{D}_{k-1}})\llcorner D_{k}^{0}))\\
    =\sum_{I_{k-1}\in\mathcal{I}_{(k-1)}}\mass(\partial C(F(y);(\Delta_{\underline{i}_{k-1}}\underline{s},s_{k}^{i_{k}(x)}))_{\underline{D}_{k-1}})\llcorner\interior(\Sigma^{k}))\\
    \leq \sum_{I_{k-1}\in\mathcal{I}_{(k-1)}}\mass(\partial \overline{C}(F(y);(\Delta_{\underline{i}_{k-1}}\underline{s},s_{k}^{i_{k}(x)}))_{\underline{D}_{k-1}})\llcorner R_{n-2}^{-1}(\interior(\Sigma^{k})))\\
    =\mass(\partial\overline{C}(F(y);s_{1}^{2^{p'}},...,s_{k-1}^{2^{p'}},s_{k}^{i_{k}(x)},s_{k+1}^{1},...,s_{n-2}^{1})\llcorner R_{n-2}^{-1}(\interior(\Sigma^{k}))).
\end{multline*}
Now using the fact that $\partial\overline{C}(F(y);s_{1}^{i_{1}},...,s_{n-2}^{i_{n-2}})=\sum_{j=1}^{n-2}\partial_{j}\overline{C}(F(y);s_{1}^{i_{1}},...,s_{n-2}^{i_{n-2}})$, that $\partial_{j}\overline{C}(F(y);s_{1}^{i_{1}},...,s_{n-2}^{i_{n-2}})$ is supported in $R_{n-1}^{-1}(\Sigma^{j})$ and that $\Sigma^{j}\cap\interior(\Sigma^{k})=\emptyset$ if $j>k$, the previous equals
\begin{align*}
    \sum_{j=1}^{k}\mass(\partial_{j}\overline{C}(F(y);s_{1}^{2^{p'}},...,s_{k-1}^{2^{p'}},s_{k}^{i_{k}(x)},s_{k+1}^{1},...,s_{n-2}^{1})\llcorner R_{n-2}^{-1}(\interior(\Sigma^{k})))\\
    \leq\sum_{j=1}^{k}\mass(\partial_{j}\overline{C}(F(y);s_{1}^{2^{p'}},...,s_{k-1}^{2^{p'}},s_{k}^{i_{k}(x)},s_{k+1}^{1},...,s_{n-2}^{1}))\\
    \leq\sum_{j=1}^{k}\frac{\mass(F(y)}{\varepsilon_{j}}\\
    \leq k\frac{\mass(F(y))}{\varepsilon_{k}}
\end{align*}
as desired. Condition \ref{ite : 3.1}\ref{ite : 3.1.1} is also clearly satisfied. Finally, condition \ref{ite: 3.2} is equivalent to
\begin{equation*}
    A(F(x);\underline{s},\mu(x))=e^{y}(x)+A(F(y);\underline{s},\mu(x))
\end{equation*}
which comes from the fact that the family $F$ is $\delta$-localized and the cuts are made away from $\{U_{i}^{C}\}_{i\in I_{C}}$ (the $(1+\beta)\delta$-admissible families $\{\tilde{U}^{C}_{i}\}_{i\in I_{C}}$ for $F'$ are given by $\tilde{U}_{i}^{C}=R_{n-2}(U_{i}^{C})$ where $c(0)=1+\beta$).

\textbf{Inductive step.} Assume the inductive property holds for the $(l-1)$-skeleton. We want to extend $F'$ to the $l$-skeleton in such a way that the inductive property also holds there. Let $C$ be an $l$-cell of $X$ and $y\in \partial C$. Enumerate the vertices of $C$ as $x_{j}^{1},...,x_{j}^{2^{l}}$ so that $s_{j}(x_{j}^{i})=s_{j}^{i}$ is decreasing in $i$ for every $1\leq j\leq n-2$. By inductive hypothesis and the fact that $\partial C=C\cap X_{l-1}$, we have continuous maps $\mu^{\underline{D}_{k}}_{\underline{i}_{k}}:\partial C\to[0,1]$ and $e^{y}:\partial C\to\mathcal{Z}_{1}(M)$ verifying the properties in \ref{ite : 3.1} and \ref{ite: 3.2} above and such that
\begin{equation*}
    F'(x)=e^{y}(x)+A(F(y);\underline{s},\mu(x))
\end{equation*}
for every $x\in \partial C$.

We want to extend the coefficients $\mu^{\underline{D}_{k}}_{\underline{i}_{k}}$ continuously to $C$. For that purpose, we will use the fact that $C$ is homeomorphic to a cone over $\partial C$. So the idea will be to define functions $\mu^{\underline{D}_{k}}_{\underline{i}_{k}}:\partial C\times I_{0}\to[0,1]$ for a certain interval $I_{0}=[0,N_{0}]$ such that $\mu_{\underline{i}_{k}}^{\underline{D}_{k}}(x,0)=\mu_{\underline{i}_{k}}^{\underline{D}_{k}}(x)$ and $\mu^{\underline{D}_{k}}_{\underline{i}_{k}}(x,N_{0})$ is a constant function of $x$, inducing functions whose domain is a cone over $\partial C$ which will provide the desired extensions to $C$. 

Given $I_{(k-1)}\in\mathcal{I}_{k-1}$ and $D_{k}\in\mathcal{F}^{k}$, let $i_{k}(I_{(k-1)},D_{k})\in\{1,...,2^{l}\}$ be such that 
\begin{multline}\label{Def of ik0}
    \mass(\partial C(F(y);(\Delta_{\underline{i}_{k-1}}\underline{s},s_{k}^{i_{k}(I_{(k-1)},D_{k})}))_{\underline{D}_{k-1}}\llcorner D_{k}^{0}) \\
    =\min\{\mass(\partial C(F(y);(\Delta_{\underline{i}_{k-1}}\underline{s},s_{k}^{i_{k}}))_{\underline{D}_{k-1}}\llcorner D_{k}^{0}):1\leq i_{k}\leq 2^{l}\}.
\end{multline}
To construct the $\mu^{\underline{D}_{k}}_{\underline{i}_{k}}(x,t)$, what we are going to do is given $I_{(k-1)}\in\mathcal{I}_{(k-1)}$ and $D_{k}\in\mathcal{F}^{k}$, contract the coefficient $\mu^{\underline{D}_{k}}_{(\underline{i}_{k-1},i_{k})}(x)$ to $1$ if $i_{k}\leq i_{k}(I_{(k-1)},D_{k})$ and to $0$ otherwise (this generalizes the strategy in \cite{GL22}[Section~4.5]). We want to contract them one at a time (i.e. use different intervals of time inside $[0,N_{0}]$ for different values of $k$, $\underline{i}_{k}$ and $\underline{D}_{k}$) to ensure \ref{ite : 3.1}\ref{ite : 3.1.1} and in a way that makes the sum (\ref{Ineq cones}) non increasing in $t$ when replacing $x$ by $(x,t)$ in order to guarantee \ref{ite : 3.1}\ref{ite : 3.1.2}.

To carry out the previous, consider the set
\begin{equation*}
    A'=\{(k,I_{(k-1)},D_{k}):1\leq k\leq n-2,I_{(k-1)}\in\mathcal{I}_{(k-1)},D_{k}\in\mathcal{F}^{k}\},
\end{equation*}
denote $N'_{0}=\#A'$, fix an order on $A'$ and given $(k,I_{(k-1)},D_{k})\in A'$ let $m'(k,I_{(k-1)},D_{k})$ be the position of the element $(k,I_{(k-1)},D_{k})$ under that order. We set $N_{0}=(2^{l}-1)N'_{0}$, observe that if
\begin{equation*}
    A=A'\times\{2,...,2^{l}\}
\end{equation*}
then $N_{0}=\# A$. Given $(k,I_{(k-1)},D_{k})\in A'$ we consider an order $o_{(k,I_{(k-1)},D_{k})}$ on $\{2,...,2^{l}\}$ given by
\begin{equation*}
    i_{k}(I_{(k-1)},D_{k})+1,i_{k}(I_{(k-1)},D_{k}),...,3,2,i_{k}(I_{(k-1)},D_{k})+2,i_{k}(I_{(k-1)},D_{k})+3,...,2^{l}.
\end{equation*}

We endow $A$ with the order $m(k,I_{(k-1)},D_{k},i_{k})<m(k',I'_{(k'-1)},D'_{k'},i'_{k'})$ if either $m'(k,I_{(k-1)},D_{k})<m'(k',I'_{(k'-1)},D'_{k'})$ or $(k,I_{(k-1)},D_{k})=(k',I'_{(k'-1)},D'_{k'})$ and \linebreak $o_{(k,I_{(k-1)},D_{k})}(i_{k})<o_{(k,I_{(k-1)},D_{k})}(i'_{k})$. We denote $m(k,I_{(k-1)},D_{k},i_{k})$ the position of the element $(k,I_{(k-1)},D_{k},i_{k})\in A$ under that order.

Now we define $\mu^{\underline{D}_{k}}_{\underline{i}_{k}}(x,t)$ for each $t\in[0,N_{0}]$. We only deal with the case $i_{k}\neq 1$ and set $\mu^{\underline{D}_{k}}_{(\underline{i}_{k-1},1)}(x,t)=1$. Hence the set $A$ with the order $m$ enumerates all the $\mu^{\underline{D}_{k}}_{\underline{i}_{k}}$ we must define. Given $(k,I_{(k-1)},D_{k},i_{k})$ we denote $m=m(k,I_{(k-1)},D_{k},i_{k})$ for short.
\begin{enumerate}
    \item For $0\leq t\leq m-1=m(k,I_{(k-1)},D_{k},i_{k})-1$, we set
    \begin{equation*}
        \mu^{\underline{D}_{k}}_{\underline{i}_{k}}(x,t)=\mu^{\underline{D}_{k}}_{\underline{i}_{k}}(x).
    \end{equation*}
    \item For $m(k,I_{(k-1)},D_{k},i_{k})-1\leq t\leq m(k,I_{(k-1)},D_{k},i_{k})$, we set
    \small
    \[
    \mu^{\underline{D}_{k}}_{\underline{i}_{k}}(x,t)=\begin{cases}
        1+t-m+(m-t)\mu^{\underline{D}_{k}}_{\underline{i}_{k}}(x) & \text{ if } i_{k}\leq i_{k}(I_{(k-1)},D_{k})\\
        (m-t)\mu^{\underline{D}_{k}}_{\underline{i}_{k}}(x) & \text{ otherwise}.
    \end{cases}
    \]
    \normalsize
    \item For $m(k,I_{(k-1)},D_{k},i_{k})\leq t\leq N_{0}$ we set
    \[
    \begin{cases}
        1 & \text{ if } i_{k}\leq i_{k}(I_{(k-1)},D_{k})\\
        0 & \text{ otherwise}.
    \end{cases}
    \]
\end{enumerate}

Extend $e^{y}:\partial C\to\mathcal{Z}_{1}(\bigcup_{i\in I_{C}}U^{C}_{i})$ to $C$ by radial contraction on each ball $U^{C}_{i}$ and define
\begin{equation*}
    F'(x)=e^{y}(x)+A(F(y);\underline{s},\mu(x)).
\end{equation*}
The previous extends $F'$ to the $l$-skeleton. 

\begin{lemma}
    Property \ref{ite : 3.1}\ref{ite : 3.1.1} holds on each $l$-dimensional cell $C$ of $X$.
\end{lemma}

\begin{proof}
    It suffices to verify it for $x\in\interior(C)$ as for $x\in\partial C$ it is given by inductive hypothesis. From the previous construction we can identify $x$ with $(x',t)$ for some $x'\in\partial C$ and some $t\in(0,1]$ being $\dim(B(x'))\leq l-1$. We know that
     \begin{equation*}
            \sum_{k=1}^{n-2}\#\{I_{(k)}\in\mathcal{I}_{(k)}:\mu^{\underline{D}_{k}}_{\underline{i}_{k}}(x')\in(0,1)\}\leq\dim(B(x')).
        \end{equation*}
     Let $m$ be the unique integer such that $m-1\leq t < m+1$ and let $(k,I_{(k-1)},D_{k},i_{k})$ be the unique element of $A$ such that $m=m(k,I_{(k-1)},D_{k},i_{k})$. The only $(k',I'_{(k'-1)},D'_{k'},i'_{k'})\in A'$ such that $\mu^{\underline{D}'_{k'}}_{\underline{i}'_{k'}}(x',t)$ could be in $(0,1)$ without having $\mu^{\underline{D}'_{k'}}_{\underline{i}'_{k'}}(x')\in (0,1)$ is $(k',I'_{(k'-1)},D'_{k'},i'_{k'})=(k,I_{(k-1)},D_{k},i_{k})$. This implies
     \begin{equation*}
            \sum_{k=1}^{n-2}\#\{I_{(k)}\in\mathcal{I}_{(k)}:\mu^{\underline{D}_{k}}_{\underline{i}_{k}}(x)\in(0,1)\}\leq\dim(B(x'))+1\leq l=\dim(E(x)).
        \end{equation*}
      The property that $\mu^{\underline{D}_{k}}_{\underline{i}_{k}}(x)=0$ if $i_{k}>i_{k}(x)$ and $\mu^{\underline{D}_{k}}_{\underline{i}_{k}}(x)=1$ when $i_{k}\leq i_{k}(x)$ is such that $x_{k}^{i_{k}}\notin E(x)$ follows from the inductive hypothesis if $x\in\partial C$ and is vacuously true for $x\in\interior(C)$.
\end{proof}

\begin{lemma}
    Property \ref{ite : 3.1}\ref{ite : 3.1.2} holds on each $l$-dimensional the cell $C$ of $X$.
\end{lemma}

\begin{proof}
    We identify $\tilde{x}\in C$ with $(x,t)\in\partial C\times[0,N_{0}]$ as above. We fix $x\in\partial C$ and $1\leq k\leq n-2$, and we consider
    \begin{equation*}
        g(t)= \sum_{I_{k-1}\in\mathcal{I}_{(k-1)}}\sum_{D_{k}\in\mathcal{F}_{k}}\sum_{i_{k}=1}^{q}(\tilde{\mu}^{\underline{D}_{k}}_{(\underline{i}_{k-1},i_{k})}(x,t)-\tilde{\mu}^{\underline{D}_{k}}_{(\underline{i}_{k-1},i_{k}+1)}(x,t))\mass(\partial C(F(y);(\Delta_{\underline{i}_{k-1}}\underline{s},s_{k}^{i_{k}}))_{\underline{D}_{k-1}})\llcorner D_{k}^{0}))
    \end{equation*}
    for $t\in[0,N_{0}]$. We know that $g(0)\leq k\frac{\mass(F(y))}{\varepsilon_{k}}$ so it is sufficient to show that $g(t)$ is decreasing as a function of $t$. But we claim that for each $I_{k-1}\in\mathcal{I}_{(k-1)}$ and each $D_{k}\in\mathcal{F}^{k}$ the function
    \begin{equation*}
        g_{I_{(k-1)},D_{k}}(t)=\sum_{i_{k}=1}^{q}(\tilde{\mu}^{\underline{D}_{k}}_{(\underline{i}_{k-1},i_{k})}(x,t)-\tilde{\mu}^{\underline{D}_{k}}_{(\underline{i}_{k-1},i_{k}+1)}(x,t))\mass(\partial C(F(y);(\Delta_{\underline{i}_{k-1}}\underline{s},s_{k}^{i_{k}}))_{\underline{D}_{k-1}})\llcorner D_{k}^{0}))
    \end{equation*}
    is decreasing. Observe that $g_{I_{(k-1)},D_{k}}(t)$ is constant on each of the two components of the complement of
    \begin{equation*}
        [(m'(k,I_{(k-1)},D_{k})-1)(2^{l}-1),m'(k,I_{(k-1)},D_{k})(2^{l}-1)]
    \end{equation*}
    and that the previous interval is the union of the unit intervals of the form $[m(k,I_{(k-1)},D_{k},i_{k})-1,m(k,I_{(k-1)},D_{k},i_{k})]$ for $2\leq i_{k}\leq 2^{l}$ where the order of such intervals within $[(m'(k,I_{(k-1)},D_{k})-1)(2^{l}-1),m'(k,I_{(k-1)},D_{k})(2^{l}-1)]$ is given by $o_{(k,I_{(k-1)},D_{k})}$. Thus it suffices to check that $g_{I_{(k-1)},D_{k}}(t)$ decreases along each of the previously mentioned unit intervals. In order to do that, let us first observe the following. Suppose a certain coefficient $\mu^{\underline{D}_{k}}_{(\underline{i}_{k-1},i_{k}^{*})}(x,t)$ is linearly changed to $s\in\{0,1\}$ for $t$ along a certain unit interval $I'$ and all other coefficients are kept constant. Then
    \begin{enumerate}
        \item If $s=0$, we can say the following about the behaviour of $\tilde{\mu}^{\underline{D}_{k}}_{(\underline{i}_{k-1},i_{k})}(x,t)-\tilde{\mu}^{\underline{D}_{k}}_{(\underline{i}_{k-1},i_{k}+1)}(x,t)$ as a function of $t\in I'$.
        \begin{enumerate}
            \item If $i_{k}<i_{k}^{*}-1$, it is constant.
            \item If $i_{k}=i_{k}^{*}-1$, it increases.
            \item If $i_{k}\geq i_{k}^{*}$, then it is multiplied by a factor of $\frac{\mu^{\underline{D}_{k}}_{(\underline{i}_{k},i_{k}^{*})}(x,t)}{\mu^{\underline{D}_{k}}_{(\underline{i}_{k},i_{k}^{*})}(x,0)}$ with respect to its value at $t=0$ and hence it decreases.
        \end{enumerate}
        
        \item If $s=1$, we can say the following about the behaviour of $\tilde{\mu}^{\underline{D}_{k}}_{(\underline{i}_{k-1},i_{k})}(x,t)-\tilde{\mu}^{\underline{D}_{k}}_{(\underline{i}_{k-1},i_{k}+1)}(x,t)$ as a function of $t\in I'$.
        \begin{enumerate}
            \item If $i_{k}<i_{k}^{*}-1$, it is constant.
            \item If $i_{k}=i_{k}^{*}-1$, it decreases.
            \item If $i_{k}\geq i_{k}^{*}$, then it is multiplied by a factor of $\frac{\mu^{\underline{D}_{k}}_{(\underline{i}_{k},i_{k}^{*})}(x,t)}{\mu^{\underline{D}_{k}}_{(\underline{i}_{k},i_{k}^{*})}(x,0)}$ with respect to its value at $t=0$ and hence it increases.
        \end{enumerate}
    \end{enumerate}
    
    Now let us show that $g_{I_{(k-1)},D_{k}}(t)$ is decreasing in $t$ along each interval $I'=[m(k,I_{(k-1)},D_{k},i_{k}^{*})-1,m(k,I_{(k-1)},D_{k},i_{k}^{*})]$. Denote $i_{k}^{0}=i_{k}(I_{(k-1)},D_{k})$ for short. We have three cases.

    \begin{enumerate}
        \item If $i_{k}^{*}=i_{k}^{0}+1$, $s=0$ and thus $\tilde{\mu}^{\underline{D}_{k}}_{(\underline{i}_{k-1},i_{k})}(x,t)-\tilde{\mu}^{\underline{D}_{k}}_{(\underline{i}_{k-1},i_{k}+1)}(x,t)$ only increases for $i_{k}=i_{k}^{0}$. As $\sum_{i_{k}=1}^{2^{l}}\tilde{\mu}^{\underline{D}_{k}}_{(\underline{i}_{k-1},i_{k})}(x,t)-\tilde{\mu}^{\underline{D}_{k}}_{(\underline{i}_{k-1},i_{k}+1)}(x,t)$ is constant (in fact, it is equal to $1$), by (\ref{Def of ik0}) we deduce that  the function $g_{I_{(k-1)},D_{k}}(t)$ is decreasing for $t\in I'$.
        \item If $2\leq i_{k}^{*}\leq i_{k}^{0}$, we have that $s=1$ and also that
        \begin{enumerate}
            \item $\tilde{\mu}^{\underline{D}_{k}}_{(\underline{i}_{k-1},i_{k})}(x,t)-\tilde{\mu}^{\underline{D}_{k}}_{(\underline{i}_{k-1},i_{k}+1)}(x,t)=0$ for all $t\in I'$ and all $i_{k}^{*}\leq i_{k}<i_{k}^{0}$ because $\mu^{\underline{D}_{k}}_{(\underline{i}_{k-1},i_{k})}(x,t)=1$ for $i_{k}^{*}<i_{k}\leq i_{k}^{0}$ and $t\in I'$.
            \item $\tilde{\mu}^{\underline{D}_{k}}_{(\underline{i}_{k-1},i_{k})}(x,t)-\tilde{\mu}^{\underline{D}_{k}}_{(\underline{i}_{k-1},i_{k}+1)}(x,t)=0$ for all $i_{k}>i_{k}^{0}$ because $\mu^{\underline{D}_{k}}_{(\underline{i}_{k-1},i_{k}^{0}+1)}(x,t)=0$ for all $t\in I'$.
        \end{enumerate}
        Thus the only value of $i_{k}$ for which $\tilde{\mu}^{\underline{D}_{k}}_{(\underline{i}_{k-1},i_{k})}(x,t)-\tilde{\mu}^{\underline{D}_{k}}_{(\underline{i}_{k-1},i_{k}+1)}(x,t)$ increases is $i_{k}=i_{k}^{0}$. Using again the fact that $\sum_{i_{k}=1}^{2^{l}}\tilde{\mu}^{\underline{D}_{k}}_{(\underline{i}_{k-1},i_{k})}(x,t)-\tilde{\mu}^{\underline{D}_{k}}_{(\underline{i}_{k-1},i_{k}+1)}(x,t)$ is constantly equal to $1$ we can see that $g_{I_{(k-1)},D_{k}}(t)$ is decreasing in $I'$.
        \item If $i_{k}^{*}>i_{k}^{0}+1$, $s=0$ and thus $\tilde{\mu}^{\underline{D}_{k}}_{(\underline{i}_{k-1},i_{k})}(x,t)-\tilde{\mu}^{\underline{D}_{k}}_{(\underline{i}_{k-1},i_{k}+1)}(x,t)$ is constant for $i_{k}\leq i_{k}^{0}$ and it is equal to $0$ for $i_{k}> i_{k}^{0}$ as $\mu^{\underline{D}_{k}}_{(\underline{i}_{k-1},i_{k}^{0}+1)}(x,t)=0$. Therefore $g_{I_{(k-1)},D_{k}}(t)$ is decreasing in $I'$.
    \end{enumerate}
    
\end{proof}

Thus given an $l$-cell $C$ of $X$, properties \ref{ite : 3.1}\ref{ite : 3.1.1} and \ref{ite : 3.1}\ref{ite : 3.1.2} follow and we get the desired extension
\begin{equation*}
    F'(x)=e^{y}(x)+A(F(y);\underline{s},\mu(x))
\end{equation*}
for each $x\in C$. Picking a different $y'\in C$ will lead to a new family $e^{y'}:X\to \mathcal{Z}_{1}(\bigcup_{i\in I_{C}}U^{C}_{i})$ such that (\ref{Eq F'}) holds because the map $F$ we start with is $\delta$-localized and the cuts are done away from the admissible family $\{U_{i}^{C}\}_{i}$. Notice that
\begin{equation*}
    \mass(e^{y}(x)-e^{y'}(x))\leq\varepsilon
\end{equation*}
and therefore $\mass(e^{y'}(x))\leq c(k+1)\varepsilon$ for every $y'\in C$.

For an arbitrary $p'$-cell $\tilde{C}$ of $X$ with $l\leq p'\leq p$, we can proceed as follows. Given $x\in \tilde{C}\cap X_{l}$, if $C=E(x)$ we know $\dim(C)\leq l$ and thus we have the coefficients $\mu^{\underline{D}_{k}}_{(\underline{i}_{k-1},i_{k})}(x)$ already defined in the cell $C$. That allows us to define $\mu^{\underline{D}_{k}}_{(\underline{i}_{k-1},i_{k})}(x)$ for the values of $i_{k}$ such that $x_{k}^{i_{k}}\in C$. For the other vertices, we choose either $0$ or $1$ according to \ref{ite : 3.1}\ref{ite : 3.1.2}. Under this definition, it follows that \ref{ite : 3.1} and \ref{ite: 3.2} hold for $\tilde{C}$. This completes the proof of the inductive step.

By induction, we obtain a map $F':X^{p}\to\mathcal{I}_{1}(M)_{\Sigma}$ which is continuous in the flat topology and satisfies that given a cell $C$ of $X$ and $y\in C$ it holds
\begin{equation*}
    F'(x)=e^{y}(x)+A(F(y);\underline{s},\mu(x))
\end{equation*}
for every $x\in C$. Now we need to give upper bounds for $\mass(F'(x))$ and for $\mass(\partial F'(x))$, using the inductive construction and the choice of the coefficients $\rho_{l}$ and $\varepsilon_{l}$ for each $1\leq l\leq n-2$. As $e^{y}(x)$ has length bounded by $c(p)\varepsilon$ by hypothesis, it suffices to bound the mass of $A(F(y);\underline{s},\mu(x))$ and of its boundary. We know that
\begin{equation*}
    A(F(y);\underline{s},\mu)=C(F(y);\underline{s},\mu(x))+\sum_{k=1}^{n-1}\Cone_{k}(F(y);\underline{s},\mu(x)).
\end{equation*}
We have previously seen that
\begin{equation*}
    \mass(C(F(y);\underline{s},\mu(x)))\leq (1+\beta)\mass(F(y))
\end{equation*}
where
\begin{equation*}
    1+\beta=(1-\frac{2}{\sqrt{n_{p}}})^{-(n-2)}
\end{equation*}
and that for each $1\leq k\leq n-2$
\begin{multline*}
    \mass(\Cone_{k}(F(y);\underline{s},\mu(x))\leq \\
    \sum_{I_{k-1}\in\mathcal{I}_{(k-1)}}\sum_{D_{k}\in\mathcal{F}_{k}}\sum_{i_{k}=1}^{q}\rho_{k}(\tilde{\mu}^{\underline{D}_{k}}_{(\underline{i}_{k-1},i_{k})}(x)-\tilde{\mu}^{\underline{D}_{k}}_{(\underline{i}_{k-1},i_{k}+1)}(x))\mass(\partial C(\tau;(\Delta_{\underline{i}_{k-1}}\underline{s},s_{k}^{i_{k}}))_{\underline{D}_{k-1}})\llcorner D_{k}^{0}))\\
    \leq k\rho_{k}\frac{\mass(F(y))}{\varepsilon_{k}}
\end{multline*}
by Proposition \ref{Prop cone k bound} and (\ref{Ineq cones}). It remains to bound $\mass(\Cone_{n-1}(F(y);\underline{s},\mu(x))$.

Recall that
\begin{equation*}
    \Cone_{n-1}(F(y),\underline{s},\mu(x))=\sum_{I\in\mathcal{I}}T^{\underline{D}_{n-2}}_{\underline{i}_{n-2}}(\Cone_{n-1}(C(F(y),\Delta_{\underline{i}}\underline{s})_{\underline{D}_{n-2}})).
\end{equation*}
We divide the terms of the previous sum into two groups
\begin{enumerate}
    \item Those $I\in\mathcal{I}$ for which $T^{\underline{D}_{n-2}}_{\underline{i}_{n-2}}(x)=id$ or $T^{\underline{D}_{n-2}}_{\underline{i}_{n-2}}(x)=0$. They are supported in $\Sigma^{n-1}$, the $1$-skeleton of $\Sigma^{n-2}$.
    \item The collection $\mathcal{I}(x)=\{I\in\mathcal{I}:T^{\underline{D}_{n-2}}_{\underline{i}_{n-2}}(x)\neq 0,id\}$. For such $I$, let $1\leq k\leq n-2$ be the smallest integer such that $T^{\underline{D}_{k}}_{\underline{i}_{k}}(x)=T^{\underline{D}_{n-2}}_{\underline{i}_{n-2}}(x)$ (i.e. the unique integer such that $\mu^{\underline{D}_{k}}_{\underline{i}_{k}}(x)\in(0,1)$ and $\mu^{\underline{D}_{j}}_{\underline{i}_{j}}(x)\in \{0,1\}$ for all $j>k$) and consider the map $\phi:\mathcal{I}(x)\to \mathcal{A}(x)$ which assigns $I\mapsto(\underline{i}_{k},\underline{D}_{k})$, where
    \begin{equation*}
        \mathcal{A}(x)=\bigcup_{k=1}^{n-2}\{(\underline{i}_{k},\underline{D}_{k})\in\mathcal{I}_{(k)}:\mu^{\underline{D}_{k}}_{\underline{i}_{k}}(x)\in(0,1)\}.
    \end{equation*}
\end{enumerate}
Thus we can see that $\Cone_{n-1}(F(y);\underline{s},\mu(x))$ is supported in
\begin{equation*}
    \Sigma(x)=\Sigma^{n-1}\cup\bigcup_{(\underline{D}_{k},\underline{i}_{k})\in\mathcal{A}(x)}T^{\underline{D}_{k}}_{\underline{i}_{k}}(D_{k}\cap\Sigma^{n-1}).
\end{equation*}
We want to find an upper bound for $\mass(\Sigma(x))$. In order to do that, we prove the following lemma. 

\begin{lemma}\label{Lemma number of cells}
    Let $1\leq k<l\leq n-2$ and $D\in\mathcal{F}^{k}$. Then there exists a constant $C(n)$ depending only on $n$ such that
    \begin{equation*}
        \#\{\text{top-cells of }\Sigma^{l}\text{ contained in }D\}\leq C(n)p^{\alpha}.
    \end{equation*}
\end{lemma}

\begin{proof}
    Observe that there are at most $2(n-l')(\frac{\rho_{l'}}{\rho_{l'+1}})^{n-l'-1}$ top faces of $\Sigma^{l'+1}$ contained in each top face $D'$ of $\Sigma^{l'}$. Therefore
    \begin{align*}
        \#\{\text{top-cells of }\Sigma^{l}\text{ contained in }D\} & \leq\prod_{l'=k}^{l-1}2(n-l')(\frac{\rho_{l'}}{\rho_{l'+1}})^{n-l'-1}\\
        & \leq\prod_{l'=1}^{n-3}2(n-l')(\frac{\rho_{l'}}{\rho_{l'+1}})^{n-l'-1}\\
        & \leq C(n)\prod_{l'=1}^{n-3}p^{\alpha'(n-l'-1)}\\
        &\leq C(n) p^{\alpha'\frac{(n-1)(n-2)}{2}}\\
        & =C(n)p^{\alpha}.
    \end{align*}
    We used the fact that $\frac{\rho_{l'}}{\rho_{l'+1}}=p^{\alpha'}$ and $\alpha'=\frac{2}{(n-1)(n-2)}\alpha$.
\end{proof}

Given $D\in\mathcal{F}^{k}$ for some $1\leq k\leq n-2$, by Lemma \ref{Lemma number of cells}
\begin{align*}
    \mass(D\cap\Sigma^{n-1}) & \leq \rho_{n-2}\#\{1\text{-cells of }\Sigma^{n-2}\text{ contained in }D_{k}\}\\
    & \leq 4\rho_{n-2}\#\{\text{top-cells of }\Sigma^{n-2}\text{ contained in }D\}\\
    & \leq C(n)p^{\alpha}\rho_{n-2}\\
    & \leq C(n)p^{\alpha-\frac{1}{n-1}}
\end{align*}
using that $\rho_{n-2}\leq \rho_{1}=p^{-\frac{1}{n-1}}$. This implies 
\begin{align*}
    \mass(\Sigma^{n-1}) & \leq \#\{\text{top-cells of }\Sigma^{1}\}C(n)p^{\alpha-\frac{1}{n-1}}\\
    & \leq \frac{\Vol_{n-1}(\Sigma)}{\rho_{1}^{n-1}} C(n)p^{\alpha-\frac{1}{n-1}}\\
    & \leq C(n)\Vol_{n-1}(\Sigma)p^{1+\alpha-\frac{1}{n-1}}.
\end{align*}
As $\#\mathcal{A}(x)\leq p$ and 
\begin{equation*}
    \mass(T^{\underline{D}_{k}}_{\underline{i}_{k}}(D_{k}\cap\Sigma^{n-1}))\leq\mass(D_{k}\cap\Sigma^{n-1})\leq C(n)p^{\alpha-\frac{1}{n-1}}
\end{equation*}
for each $(\underline{i}_{k},\underline{D}_{k})\in\mathcal{A}(x)$, we can see
\begin{align*}
    \mass(\Sigma(x))& \leq C(n)\Vol_{n-1}(\Sigma)p^{1+\alpha-\frac{1}{n-1}}+pC(n)p^{\alpha-\frac{1}{n-1}}\\
    & = C(n)(\Vol_{n-1}(\Sigma)+1)p^{1+\alpha-\frac{1}{n-1}}
\end{align*}
which by mod $2$ cancellation implies
\begin{equation*}
    \mass(\Cone_{1}(F(y);\underline{s},\mu(x)))\leq C(n)(\Vol_{n-1}(\Sigma)+1)p^{1+\alpha-\frac{1}{n-1}}.
\end{equation*}

Now we are going to bound $\mass(\partial A(F(y);\underline{s},\mu(x)))$. Recall that $\Sigma^{n}$ denotes the $0$-skeleton of $\Sigma^{n-2}$. Given $1\leq k\leq n-2$ and a top dimensional cell $D$ of $\Sigma^{k}$, let
\begin{equation*}
    \centers(D)=\big (\Sigma^{n}\cap D\big )\cup\bigcup_{l=k}^{n-2}\bigcup_{\substack{E\subseteq D \\ E\in\mathcal{F}^{l}}}\{q_{E}\}.
\end{equation*}
We can see that 
\begin{align*}
    \mass(\centers(D)) & \leq \mass(\Sigma^{n}\cap D)+\sum_{l=k}^{n-2}\#\{E\in\mathcal{F}^{l}:E\subseteq D\}\\
    & \leq 4\#\{E\in\mathcal{F}^{n-2}:E\subseteq D\}+\sum_{l=k}^{n-2}\#\{E\in\mathcal{F}^{l}:E\subseteq D\}\\
    & \leq 5\sum_{l=k}^{n-2}\#\{E\in\mathcal{F}^{l}:E\subseteq D\}\\
    & \leq C'(n)p^{\alpha}.
\end{align*}
where $C'(n)=5(n-2)C(n)$ and in the last step we applied Lemma \ref{Lemma number of cells}. Now as
\begin{equation*}
    \partial A(F(y);\underline{s},\mu(x))=\sum_{I\in\mathcal{I}}T^{\underline{D}_{n-2}}_{\underline{i}_{n-2}}(\partial A(F(y);\Delta_{\underline{I}}\underline{s})_{\underline{D}_{n-2}})
\end{equation*}
and
\begin{equation*}
    \support(\partial A(F(y);\Delta_{\underline{i}}\underline{s})_{\underline{D}_{n-2}})\subseteq\centers(D_{k})
\end{equation*}
in case $I\in \mathcal{I}(x)$ and $\phi(I)=(\underline{i}_{k},\underline{D}_{k})$, and otherwise
\begin{equation*}
    \support(\partial A(F(y);\Delta_{\underline{i}}\underline{s})_{\underline{D}_{n-2}})\subseteq \bigcup_{D\in\mathcal{F}_{1}}\centers(D)
\end{equation*}
we deduce that
\begin{equation*}
    \support(\partial A(F(y);\underline{s},\mu(x)))\subseteq \bigcup_{D\in\mathcal{F}_{1}}\centers(D)\cup\bigcup_{(\underline{i}_{k},\underline{D}_{k})\in\mathcal{A}(x)}T^{\underline{D}_{k}}_{\underline{i}_{k}}(\centers(D_{k}))
\end{equation*}
and hence
\begin{align*}
    \mass(\partial A(F(y);\underline{s},\mu(x))) &\leq C'(n)p^{\alpha}(\#\mathcal{F}^{1}+p)\\ 
    & \leq C'(n)p^{\alpha}(\frac{\Vol_{n-1}(\Sigma)}{\rho_{1}^{n-1}}+p)\\
    & =C'(n)(\Vol_{n-1}(\Sigma)+1)p^{1+\alpha}.
\end{align*}
In summary, we have shown that 

\begin{align}\label{Equations coarea}
     \mass(F'(x)) & \leq \mass(F(x))(1+\beta+\sum_{k=1}^{n-2}k\frac{\rho_{k}}{\varepsilon_{k}})+C(\Sigma)p^{1+\alpha-\frac{1}{n-1}}\\
     \mass(\partial F'(x)) & \leq C(\Sigma)p^{1+\alpha}.
\end{align}

    Using the fact that $\frac{\rho_{k}}{\varepsilon_{k}}\sim p^{-\frac{\alpha'}{2}}$ for every $1\leq k\leq n-2$, we define
    \begin{equation*}
        \gamma_{p}=\beta_{p}+\frac{(n-1)(n-2)}{2}p^{-\frac{\alpha'}{2}}.
    \end{equation*}
    Recall that $\beta_{p}\to 0$ as $p\to\infty$. Therefore 
    \begin{equation*}
        \lim_{p\to\infty}\gamma_{p}=0.
    \end{equation*}
    Theorem \ref{Thm Parametric Coarea} follows from (\ref{Equations coarea}) and these considerations.

\subsection{Extension to piecewise smooth Riemannian manifolds}\label{Extension to triangulable domains}

We now extend the proof of the Parametric Coarea Inequality from rectangular domains to almost $1$-Lipschitz triangulable piecewise smooth Riemannian manifolds. For that purpose, we prove the following result, which is a generalization of Lemma \ref{Lemma map S} to Euclidean polyhedra. We use the same notation as in Section \ref{Section rectangular complexes}.

\begin{proposition}\label{Prop retraction simplicial}
    Let $P$ be a Euclidean polyhedron provided with a triangulation $T$. Then there exist $c=c(P,T)>0$ and a Lipschitz homotopy $\Psi:[0,1)\times P\to P$ such that $\Psi_{0}=id$ and for each $0<\varepsilon<1$ the map $\Psi_{\varepsilon}$ has the following properties:
    \begin{enumerate}
        \item $\Psi_{\varepsilon}$ is piecewise linear and $(1+\varepsilon)$-Lipschitz.
        \item Given a face $F$ of $P$, $\Psi_{\varepsilon}(N_{c\varepsilon} (F))\subseteq F$.
    \end{enumerate}
\end{proposition}

\begin{proof}
    Fix $0<\varepsilon<1$. We will define $\Psi_{\varepsilon}$ on each face of $P$ in a consistent way. First we need to introduce some notation. If $F$ is an $n$-dimensional face of $P$, we denote by $O_{F}$ the center of the $(n-1)$-dimensional sphere where $F$ is circumscribed and we let $r_{F}$ be its radius, so that $|v-O_{F}|=r_{F}$ for every $v\in V(F)$. We denote by $\Pi_{F}$ the $n$-plane which contains $F$, and given $x\in\Pi_{F}$ and $a\in\mathbb{R}$ we denote $a\cdot_{F}x$ the scalar product of the number $a$ and the vector $x$ in vector space $\Pi_{F}$ with the origin at $O_{F}$. We denote by $T_{\varepsilon}^{F}:\Pi_{F}\to\Pi_{F}$ the homotecy $T_{\varepsilon}^{F}(x)=(1+\varepsilon)\cdot_{F} x$ and $F_{\varepsilon}=T_{\varepsilon}^{F}(F)$. We set $\Phi^{F}:\Pi_{F}\to F$ to be the nearest point projection onto $F$, which is a $1$-Lipschitz map. We define $\Psi_{\varepsilon}^{F}=\Phi^{F}\circ T^{F}_{\varepsilon}:F\to F$. We claim that if $G$ is a face of $P$ which contains $F$, then $\Psi^{G}_{\varepsilon}|_{F}=\Psi^{F}_{\varepsilon}$.

    It suffices to prove it when $F$ is a codimension-$1$ face of $G$. In that case, consider $\Pi_{G}$ as a vector space with origin $O_{F}$ and decompose it as $\Pi_{G}=\Pi_{F}\oplus L$ where $L$ is the line through $O_{F}$ and $O_{G}$ (which is perpendicular to $\Pi_{F}$). This allows to represent each point in $\Pi_{G}$ uniquely as a pair $(y,s)$ where $y\in\Pi_{F}$ and $s\in\mathbb{R}$ are chosen so that $O_{F}=(0,0)$ and $O_{G}=(0,D)$, where $D$ is the distance between $O_{F}$ and $O_{G}$. Observe that using this coordinates and the corresponding scalar product $\cdot _{F}$, if $x=(y,s)$ is a point in $\Pi_{G}$
    \begin{align*}
        T^{G}_{\varepsilon}(x) & =O_{G}+(1+\varepsilon)\cdot_{F}(x-O_{G})\\
        & =-\varepsilon O_{G}+(1+\varepsilon)\cdot_{F}(y,s)\\
        & =(T^{F}_{\varepsilon}(y),(1+\varepsilon)s-sD).
    \end{align*}
    In particular, when $y\in\Pi_{F}$ as $s=0$ it holds
    \begin{equation*}
        T^{G}_{\varepsilon}(y)=(T_{\varepsilon}^{F}(y),-D\varepsilon).
    \end{equation*}
    We can use the previous to show that if $y\in F$ then $\Psi^{G}_{\varepsilon}(y)=\Psi^{F}_{\varepsilon}(y)$. Denote $y'=T^{F}_{\varepsilon}(y)$ and $y''=T^{G}_{\varepsilon}(y)$, which are related by
    \begin{equation*}
        y''=(y',-D\varepsilon).
    \end{equation*}
    Then as $G\subseteq\{(y,s):y\in F,s\geq 0\}$, we can show that the nearest point projection $y^{*}$ of $y'$ onto $F$ equals the nearest point projection of $y''$ onto $G$. This is because $y^{*}\in G$ and if $z=(\overline{y},s)\in G$, as $\overline{y}\in F$ and $s\geq 0$,
    \begin{align*}
        |y''-z| & =\sqrt{|y'-\overline{y}|^{2}+|s+D\varepsilon|^{2}}\\
        & \geq \sqrt{|y'-y^{*}|^{2}+D^{2}\varepsilon^{2}}\\
        & = |(y',-D\varepsilon)-(y^{*},0)|\\
        & = |y''-y^{*}|.
    \end{align*}
    Therefore $\Phi^{F}(y')=\Phi^{G}(y'')$ which by definition implies that $\Psi^{F}_{\varepsilon}(y)=\Psi^{G}_{\varepsilon}(y)$. This proves that there exists a well defined $(1+\varepsilon)$-Lipschitz map $\Psi_{\varepsilon}:P\to P$ such that $\Psi_{\varepsilon}|_{F}=\Psi^{F}_{\varepsilon}$ for every face $F$ of $P$. 
    
    The previous argument allows to prove by induction that the nearest point projection onto a linear simplex is piecewise linear. Indeed, this is clear for $1$-dimensional simplices, and given an $n$-dimensional simplex $G$, if we write $\Pi_{G}$ as the union of the cones $C_{F}$ centered at $O_{G}$ over $F$ for each $(n-1)$-dimensional face $F$, we can see that in each $C_{F}\setminus G$ it holds
    \begin{equation*}
        \Psi^{G}=\Phi^{F}\circ P^{GF}
    \end{equation*}
    where $P^{GF}:\Pi^{G}\to\Pi_{F}$ is the orthogonal projection onto $\Pi_{F}$.
    
    We observe that by definition
    \begin{equation*}
        \Psi^{F}(\varepsilon,x)=\Phi^{F}((1+\varepsilon)\cdot_{F} x)
    \end{equation*}
    and hence $\Psi$ is a Lipschitz function of the two variables $(\varepsilon,x)$.
    Let $\tilde{c}$ be the minimum of $|O_{F}-O_{G}|$ over all pair of faces $F,G$ of $P$ with $F\subseteq G$, $\dim(F)=\dim(P)-1$ and $\dim(G)=\dim(P)$ ($\tilde{c}=\tilde{c}(T)$ plays for Euclidean polyhedra the role that the width $\rho$ of a rectangular structure $T$ played for cubical and rectangular complexes, see Definitions \ref{Def width cubical complex} and \ref{Def width rectangular complex}) . Let $x\in G$ be such that $\dist(x,F)\leq \frac{D\varepsilon}{1+\varepsilon}$, $\dim(G)=\dim(P)$. We can write $x=(y,s)$ with $y\in F$ and $s\in[0,\frac{D\varepsilon}{1+\varepsilon})]$. Thus 

    \begin{align*}
        \Psi^{G}_{\varepsilon}(x) & =\Phi^{G}(T_{\varepsilon}^{F}(y),s(1+\varepsilon)-D\varepsilon))\\
        & =\Psi_{\varepsilon}^{F}(y)
    \end{align*}
    because $s(1+\varepsilon)-D\varepsilon\leq 0$. Then $\Psi^{G}_{\varepsilon}(x)\in F$, which proves (2) when $F$ is a codimension-$1$ face of $P$ if we set $c=\frac{\tilde{c}}{2}$ (because $c\varepsilon=\frac{\tilde{c}}{2}\varepsilon\leq\frac{\tilde{c}\varepsilon}{1+\varepsilon}\leq \frac{D\varepsilon}{1+\varepsilon}$). For $F$ of arbitrary dimension, this follows from writing $F$ as a finite intersection of codimension-$1$ faces $F_{1},...,F_{k}$ and observing that if $x\in N_{c\varepsilon}(F)$ then $x\in N_{c\varepsilon} (F_{i})$ for $i=1,...,k$ and hence $\Psi_{\varepsilon}(x)\in\bigcap_{i=1}^{k} F_{i}=F$.
    
\end{proof}

Using Proposition \ref{Prop retraction simplicial} instead of Lemma \ref{Lemma map S}, we can adapt the proof of the Parametric Coarea Inequality for cubical domains to compact PL submanifolds with boundary of $\mathbb{R}^{n}$. In order to extend it to almost $1$-Lipschitz triangulable piecewise smooth Riemannian manifolds with boundary $(M^{n},g)$, we proceed as follows. Given $\varepsilon=\frac{1}{m}$, there exists a compact PL manifold with boundary $(P_{m},g_{m})$ and a $(1+\frac{1}{m})$-bilipschitz map $\Xi_{m}:(P_{m},g_{m})\to (M^{n},g)$. Let $p_{m}$ be such that if $p\geq p_{m}$ and $F:X^{p}\to\mathcal{Z}_{1}(P_{m},\partial P_{m})$ then we have a perturbed map $F':X^{p}\to\mathcal{Z}_{1}(P_{m},\partial P_{m})$ with the properties in Theorem \ref{Thm Parametric Coarea}. We can assume that $(p_{m})_{m\in\mathbb{N}}$ is an increasing sequence. Given $p\geq p_{m}$ and $F:X^{p}\to\mathcal{Z}_{1}(M,\partial M)$, there exists $F':X^{p}\to\mathcal{Z}_{1}(M,\partial M)$ which is homotopic to $F$ and verifies
\begin{enumerate}
    \item $\mass(F'(x))\leq (1+\frac{1}{m})^{2}[\mass(F(x))(1+\gamma_{p})+c(n)\Vol_{n-1}(\Sigma)p^{\frac{n-2}{n-1}+\alpha}]$.
    \item $\mass(\partial F'(x))\leq c(n)\Vol_{n-1}(\Sigma)p^{1+\alpha}$.
\end{enumerate}
Given $p\geq p_{1}$, let $m_{p}=\max\{m:p_{m}\leq p\}$. The previous implies that for every $p\geq p_{1}$, if $F:X^{p}\to\mathcal{Z}_{1}(M,\partial M)$ is a continuous family, there exists another family $F'$ which is arbitrarily close to $F$ in the flat topology and verifies
\begin{enumerate}
    \item $\mass(F'(x))\leq (1+\frac{1}{m_{p}})^{2}[\mass(F(x))(1+\gamma_{p})+c(n)\Vol_{n-1}(\Sigma)p^{\frac{n-2}{n-1}+\alpha}]$.
    \item $\mass(\partial F'(x))\leq c(n)\Vol_{n-1}(\Sigma)p^{1+\alpha}$
\end{enumerate}
which implies that Theorem \ref{Thm Parametric Coarea} holds for $(M,g)$ by replacing $1+\gamma_{p}$ by $(1+\frac{1}{m_{p}})^{2}(1+\gamma_{p})$, which still converges to $1$ as $p\to\infty$.

\bibliography{Bibliography}
\bibliographystyle{amsplain}

\newcommand{\Addresses}{{
  \bigskip
  \footnotesize

  \textsc{Department of Mathematics, Rice University, Houston, TX 77005, USA}\par\nopagebreak
  \textit{Email address}: \texttt{bruno.staffa@rice.edu}
  }}

\Addresses

\end{document}